\documentclass[numbers]{article}

\usepackage[dvipsnames]{xcolor}

\usepackage[preprint]{neurips_2025}

\everypar{\looseness=-1}
\usepackage{makecell}

\usepackage{amsmath}
\usepackage[utf8]{inputenc} 
\usepackage[T1]{fontenc}    
\usepackage{hyperref}       
\usepackage{cleveref}
\usepackage{url}            
\usepackage{booktabs}       
\usepackage{amsfonts}       
\usepackage{nicefrac}       
\usepackage{microtype}      
\usepackage{xcolor}         
\usepackage{comment}
\usepackage{booktabs}       
\usepackage{amsfonts}       
\usepackage{nicefrac}       
\usepackage{microtype}      
\usepackage{comment}
\usepackage{amssymb, amsmath, latexsym}
\usepackage{url}

\usepackage{algorithm}
\usepackage{algorithmic}

\usepackage{tabularx}
\usepackage{paralist}
\usepackage{mathtools}

\usepackage{bbm} 
\usepackage{wrapfig}
\usepackage{makecell}
\usepackage{multirow}
\usepackage{booktabs}

\usepackage{nicefrac}       

\usepackage{boxhandler}
\usepackage[flushleft]{threeparttable} 

\usepackage{caption}
\usepackage{multirow}
\usepackage{colortbl}
\definecolor{bgcolor}{rgb}{0.8,1,1}
\definecolor{bgcolor2}{rgb}{0.8,1,0.8}
\definecolor{niceblue}{rgb}{0.0,0.19,0.56}

\usepackage{hyperref}
\hypersetup{colorlinks,linkcolor={blue},citecolor={niceblue},urlcolor={blue}}

\usepackage{pifont}
\definecolor{PineGreen}{RGB}{0,110,51}
\definecolor{BrickRed}{RGB}{143,20,2}

\usepackage{tikz-cd} 

\newcommand{\R}{\mathbb{R}}

\def\<#1,#2>{\left\langle #1,#2\right\rangle}

\newcolumntype{Y}{>{\centering\arraybackslash}X}

\usepackage{xspace}

\newcommand{\algname}[1]{{\sf  #1}\xspace}

\usepackage[colorinlistoftodos,bordercolor=orange,backgroundcolor=orange!20,linecolor=orange,textsize=scriptsize]{todonotes}

\newcommand{\cO}{{\cal O}}

\newcommand{\sign}{\mathrm{sign}}

\newcommand{\EE}{\mathbb{E}}

\def\clip{\texttt{clip}}

\usepackage{hyperref}
\graphicspath{{plots/}}

\usepackage{makecell}

\usepackage{accents}
\newlength{\dhatheight}

\usepackage{pgfplotstable} 
\usetikzlibrary{automata, positioning, arrows, shapes, fit, calc, intersections}
\usepgfplotslibrary{statistics}

\def\la{\langle}
\def\ra{\rangle}
\usepackage{enumitem}
\setenumerate{partopsep=0cm, topsep=0cm}
\usepackage{enumitem}
\usepackage{subcaption}
\usepackage{graphicx}
\usepackage{pbox}

\usepackage{amsthm}
\newtheorem{assumption}{Assumption}
\newtheorem{lemma}{Lemma}
\newtheorem{proposition}{Proposition}
\newtheorem{theorem}{Theorem}
\newtheorem{corollary
}{Corollary
}

\newtheorem{example}{Example}

\usepackage{xcolor}

\begin{document}

\title{Sign Operator for Coping with Heavy-Tailed Noise in Non-Convex Optimization: High Probability Bounds Under $(L_0, L_1)$-Smoothness}

\author{Nikita Kornilov\\
  MIPT, Skoltech\\
  \texttt{kornilov.nm@phystech.edu} 
  \And 
  Philip Zmushko \\
  Yandex, MIPT\\
  \And
  Andrei Semenov \\
  EPFL \\
  \And
  Mark Ikonnikov \\
  MIPT \\
  \And
  Alexander Gasnikov \\
  MIPT, Skoltech, Innopolis University \\
  \And
  Alexander Beznosikov \\
  MIPT
}






\maketitle

\begin{abstract}
In recent years, non-convex optimization problems are more often described by generalized $(L_0, L_1)$-smoothness assumption rather than standard one. Meanwhile, severely corrupted data used in these problems has increased the demand for methods capable of handling heavy-tailed noises, i.e., noises with bounded $\kappa$-th moment. Motivated by these real-world trends and challenges, we explore sign-based methods in this setup and demonstrate their effectiveness in comparison with other popular solutions like clipping or normalization. In theory, we prove the first-known high probability convergence bounds under $(L_0, L_1)$-smoothness and heavy-tailed noises with mild parameter dependencies. In the case of standard smoothness, these bounds are novel for sign-based methods as well. In particular, \algname{SignSGD} with batching achieves sample complexity $\tilde{O}\left(\left(\frac{\Delta L_0d}{\varepsilon^2} + \frac{\Delta L_1d^\frac{3}{2}}{\varepsilon}\right)\left[1 +  \left(\frac{\sigma}{\varepsilon}\right)^\frac{\kappa}{\kappa-1}\right]\right), \kappa \in (1,2]$. Under the assumption of symmetric noises, \algname{SignSGD} with Majority Voting can robustly work on the whole range of $\kappa \in (0,2]$ with complexity $\tilde{O}\left(\left(\frac{\Delta L_0d}{\varepsilon^2} + \frac{\Delta L_1d^\frac{3}{2}}{\varepsilon}\right)\left[\frac{1}{\kappa^2} +  \frac{\sigma^2}{\varepsilon^2}\right]\right)$. We also obtain results for parameter-agnostic setups, Polyak-Lojasiewicz functions and momentum-based methods (in expectation). Our theoretical findings are supported by the superior performance of sign-based methods in training Large Language Models compared to clipping and normalization.

\end{abstract}
\section{Introduction}
\subsection{Problem statement.} 
Consider a stochastic optimization problem of a smooth non-convex function $f:\R^d \to \R$:
\begin{eqnarray}
    \min\limits_{x \in \R^d} f(x) := \EE_{\xi} [f(x, \xi)],\label{eq: min problem}
\end{eqnarray}
where the random variable $\xi$ can only be sampled from an unknown distribution. The main goal is to find a point with the smallest gradient norm. To achieve this, we are able to sample an unbiased estimate $\nabla f (x, \xi) \in \R^d$. For example, in machine learning, $f(x, \xi)$ can be interpreted as a loss function on a sample $\xi$ \cite{shalev2014understanding}. The backbone of all popular stochastic first-order methods for solving \eqref{eq: min problem} is Stochastic Gradient Descent (\algname{SGD}) \cite{robbins1951stochastic}:
\begin{equation}
    x^{k+1} = x^k - \gamma_k \cdot  g^k, \quad g^k := \nabla f (x^k, \xi^k). \notag \label{eq: sgd intro}
\end{equation} Huge success of these methods in the rapidly developing neural networks field \cite{bottou2012stochastic, kingma2014adam} has sparked numerous works studying their convergence under various assumptions on noise corrupting true gradients. For \algname{SGD}, the optimal sample complexity bound $O(\varepsilon^{-4})$ in expectation \cite{arjevani2023lower} is obtained for sub-Gaussian noise \cite{nemirovski2009robust} and for noise with bounded variance (BV) \cite{ghadimi2013stochastic}. These results are derived under classic assumptions. However, motivated by real-world complex Machine Learning applications \cite{zhang2020gradient}, modern theoretical papers focus on relaxed assumptions and settings. Below, we give three important stories that are relevant to this paper.

\textbf{$(L_0, L_1)$-smoothness.} Usually, for the objective function $f$, standard $L_0$-smoothness is assumed, i.e., $\|\nabla f(x) - \nabla f(y)\|_2 \le L_0\|x - y\|_2, \forall x, y \in \R^d$. However, a new generalized $(L_0, L_1)$-smoothness assumption was recently proposed and motivated for Large Language Models (LLM) in \cite{zhang2020gradient}. This assumption describes objective functions with a linearly growing Hessian norm: $\|\nabla^2f(x)\|_2 \leq L_0 + L_1 \|\nabla f(x)\|_2, \forall x \in \R^d$. In ongoing research, other variants of this assumption were introduced: for only once differentiable functions \cite{chen2023generalized}, for symmetrically and asymmetrically growing powers of norms  \cite{chen2023generalized}, and for sub-quadratic polynomially growing norms \cite{li2023convex}. Generalized smoothness applications can be found not only in LLM training \cite{zhang2020gradient, liu2023preGenSmooth}, but also in distributionally robust optimization \cite{levy2020largeDRO, jin2021nonDRO}, multitask learning \cite{zhang2024mgda}, federated learning \cite{liu2022communication}, and bilevel optimization \cite{haobilevel, gongnearly}. The convergence of the most popular optimization algorithms \algname{Adam} \cite{kingma2014adam} and  \algname{SGD} was explored under various noise and generalized smoothness assumptions in works \cite{li2023convergenceadam, zhang2024convergenceadam, wang2024provableadam, wang2024convergenceadam} and \cite{li2023convex}, respectively. 

\textbf{High probability bounds.} Due to the expensive training of large deep learning models \cite{davis2021low}, \textit{high probability (HP)} bounds have gained even more attention than bounds in expectation describing the behavior of stochastic methods over several runs. HP bounds provide convergence guarantees that hold true with probability at least $1 - \delta, \delta \in (0,1).$ The bound in expectation can be reduced to the HP bound using Markov's inequality; however, it leads to a dominant $\nicefrac{1}{\delta}$ factor. Meanwhile, much milder $\log\nicefrac{1}{\delta}$ factors can be achieved. For \algname{SGD}, HP bound $O(\varepsilon^{-4}\log\nicefrac{1}{\delta})$ under sub-Gaussian noise is obtained in \cite{li2020high}. However, already under BV noise, vanilla \algname{SGD} has $\nicefrac{1}{\sqrt{\delta}}$ dependency under standard \cite{sadiev2023high} and $(L_0, L_1)-$smoothness \cite{li2023convex}.

\textbf{Heavy-tailed noise.} Moreover, it is shown that the BV assumption cannot describe noises in loss functions in modern deep learning problems. In Transformer models, stochasticity tends to have a rather \textit{heavy-tailed (HT)} distribution \cite{simsekli2019tail, zhang2020adaptivegood, gurbuzbalaban2021heavy}. This means that the noise has bounded $\kappa$-th moment for some $\kappa \in (1,2]$, that is, $\EE_\xi[\| \nabla f (x, \xi) - \nabla f(x)\|_2^\kappa] \leq \sigma^\kappa$.  The desire to obtain  better $\delta$-dependency in HP bounds and to consider HT noise motivated the development of more robust modifications of \algname{SGD}, e.g. \algname{SGD} with clipping or normalization of the input gradient estimates. In this work, we show that applying a simple sign operator to the gradient estimates is an effective and comparable solution to cope with heavy-tailed noise as well. 
 \vspace{-5pt}
\subsection{Related works}
 \vspace{-5pt}
\paragraph{Clipping.} The idea of clipping the norm of the gradient estimate to reduce heavy noise demonstrates significant empirical results \cite{pascanu2013difficulty, goodfellow2016deep} and helps achieve $\log \nicefrac{1}{\delta}$ dependency under BV noise and standard smoothness \cite{nazin2019algorithms, gorbunov2020stochastic}. The clipping operator is defined as $\clip(g^k, \lambda_k) := \min\{1, \nicefrac{\lambda_k}{\|g^k\|_2}\} \cdot g^k$ and \algname{SGD} with clipping is called \algname{ClipSGD}. Clipping can also be applied to convex optimization, variational inequalities \cite{sadiev2023high}, non-smooth optimization \cite{zhang2020adaptivegood}, zeroth-order optimization \cite{kornilov2024accelerated}, robust aggregation \cite{karimireddy2021learning}, distributed optimization \cite{liu2022communication, qin2025high} and ensuring differential privacy \cite{andrew2021differentially}. 

\textit{For standard smoothness}, let us list the latest results on the HP convergence of  \algname{ClipSGD} under HT noise. First, for non-convex functions, the authors of \cite{zhang2020adaptivegood} proved lower bounds $O(\varepsilon^{-\frac{3\kappa - 2}{\kappa - 1}})$ for sample complexity in expectation. As shown in \cite{sadiev2023high, nguyen2023improved}, \algname{ClipSGD} with proper clipping levels, stepsizes, and fixed horizon achieves the bound $\tilde{O}(\varepsilon^{-\frac{3\kappa - 2}{\kappa - 1}})$. \algname{ClipSGD} can also work with an infinite horizon and extra 
$\log\nicefrac{1}{\varepsilon}$ factors. In \cite{sadiev2023high}, the authors apply \algname{ClipSGD} to Polyak-Lojasiewicz functions and obtain faster convergence.  In a number of works \cite{chen2020understanding, puchkin2024breaking}, the authors work with symmetric HT noise to eliminate the dependency on $\kappa$, expand the range of feasible $\kappa$ for $\kappa \in (0,1]$  and break the actual lower bounds from \cite{zhang2020adaptivegood}. In \cite{chen2020understanding}, \algname{ClipSGD} attains the optimal rate $O(\varepsilon^{-4})$ without any requirements on $\kappa$. In \cite{puchkin2024breaking},  the authors use the coordinate-wise median operator  with clipping and prove HP rates for convex functions as if the noise is BV for $\kappa > 0$. 

For \textit{$(L_0, L_1)-$smoothness}, only BV noise is considered: \algname{ClipSGD} \cite{zhang2020gradient, koloskova2023revisiting, reisizadeh2025variance} achieve the complexity bound $O(\varepsilon^{-4})$ in expectation. \algname{ClipSGD} with momentum \algname{M-ClipSGD}\cite{zhang2020improved} achieves the same bound when the noise in the gradient estimate is bounded. The difference is that \algname{ClipSGD} necessarily  requires large batchsizes, and \algname{M-ClipSGD} allows for any size.

Despite the effectiveness of clipping, it requires careful tuning of clipping levels, which optimal values depend on the iteration and the characteristics of the optimization problem \citep[Theorem. $3.1$]{sadiev2023high}.

\paragraph{Normalization.}
A natural simplification of clipping with a profound level schedule is the permanent normalization of the gradient estimate, i.e., $\text{norm}(g^k) := \nicefrac{g^k}{\|g^k\|_2}.$ \algname{SGD} with normalization is called \algname{NSGD} \cite{hazan2015beyond} and \algname{NSGD} with momentum is called \algname{M-NSGD}\cite{cutkosky2020momentum}. Early work devoted to normalization focused on BV noise and bounds in expectation. For standard smoothness, $O(\varepsilon^{-4})$ bounds are derived for \algname{M-NSGD} \cite{jin2021nonDRO, cutkosky2020momentum, yang2024two} and \algname{NSGD} \cite{zhao2021convergence}. For $(L_0, L_1)-$smoothness, \algname{M-NSGD} achieves a bound $\tilde{O}(\varepsilon^{-4})$ with optimal parameters and arbitrary ones (with exponential dependency on $L_1$) over an infinite horizon. 

Next, we discuss normalization-based works considering HT noises and HP bounds, but only for standard smoothness.  In \cite{liu2023breaking, cutkosky2021high}, normalization is combined with clipping which helped to cope with HT noise and obtain suboptimal in $\varepsilon$ HP bounds $\tilde{O}(\varepsilon^{-\frac{3\kappa - 2}{\kappa - 1}})$. The HP convergence of vanilla \algname{NSGD} under HT noise is proved in \cite{hubler2024gradient}. The authors show that its complexity is $O(\varepsilon^{-\frac{3\kappa - 2}{\kappa - 1}}\log\nicefrac{1}{\delta})$ for optimal parameters and $O(\varepsilon^{-\frac{2\kappa}{\kappa - 1}}\log\nicefrac{1}{\delta})$  for parameter-agnostic tuning.
The same complexities hold for \algname{M-NSGD}, but only in expectation. In experiments with BV noise, the authors observe super-linear dependency on $\log \frac{1}{\delta}$ for momentum method.
\paragraph{Sign operator.} There is one more promising modification of \algname{SGD} which behavior under heavy-tailed noise has not yet been studied. Originally proposed in \cite{bernstein2018signsgd} for distributed optimization, \algname{SignSGD} takes only a sign of each coordinate of gradient estimate $ \sign(g^k).$ There is one peculiarity in bounds for sign-based methods: they are proved w.r.t. the $\ell_1$-norm instead of smaller $\ell_2$-norm. As a consequence, additional $d$ dependent factors appear.  

For \textit{standard smoothness}, \algname{SignSGD} achieves  sample complexity $O(d^2\varepsilon^{-4})$ in expectation under BV noise \cite{bernstein2018signsgd}. Similar to \algname{ClipSGD} and \algname{NSGD}, \algname{SignSGD} requires aggressive batching, which can be substituted by \algname{SignSGD} with momentum (\algname{M-SignSGD}) with bound $O(d^4\varepsilon^{-4})$\cite{sun2023momentum}. The alternative solution is to add error feedback mechanism that additionally fixes the biased nature of the sign operator and allows using convex functions \cite{seide20141, karimireddy2019error}.

For \textit{$(L_0, L_1)-$smoothness}, the authors of \cite{crawshaw2022robustness} propose generalized \algname{SignSGD} with Adam-like structure and prove bound $O(\varepsilon^{-4}\log(d/ \varepsilon))$ for \algname{M-SignSGD} under almost surely bounded noise. For the same bounded noise, in \cite{crawshaw2025complexitylowerboundsadaptive}, the authors study the behavior of various adaptive gradient algorithms and derive lower bounds for them with explicit parameter dependencies.


For all previously mentioned works, the results are obtained under BV noise. Bounds for HT noise are obtained only under symmetry assumption and standard smoothness. Works \cite{jakovetic2023nonlinear, armacki2023high, armacki2024large} analyze online non-linear \algname{SGD} without batching for convex and non-convex functions. It includes a wide range of non-linear transformations of gradient estimates such as \textit{clipping, normalization, and sign operator.} The authors of \cite{armacki2024large} propose a unified theoretical framework and prove bounds which are arbitrarily close to $O(\varepsilon^{-4})$ for all $\kappa > 0$. 
In works \cite{compagnoni2024adaptive, compagnoni2025unbiased}, the authors derive continuous SDE with Student's noise describing \algname{SignSGD} dynamics and obtain the $O(\varepsilon^{-4})$ HP bound from it.

\subsection{Contributions}
\paragraph{Theory.} Using sign-based methods, \textbf{we prove the first-known high probability bounds for non-convex $(L_0,L_1)$-smooth optimization under heavy-tailed noise}. These bounds are valid for all possible problem parameters, have mild dependencies on them, and match the optimal bounds in the case of standard smoothness. Moreover, the \textbf{HP results for sign-based methods in case of standard $L_0$-smoothness are novel as well}. For momentum-methods, our in expectation bounds are the first to consider together heavy-tailed noise and $(L_0, L_1)$-smoothness.  For all methods, we observe two-stage convergence: in the beginning when $\varepsilon \geq \nicefrac{8L_0}{\sqrt{d}L_1}$ methods work in the accelerated regime, and after passing the threshold, convergence speed drops to rates as if smoothness is standard. In addition, we consider special cases of Polyak-Lojasiewicz functions, symmetric noises, and parameter-agnostic settings. The summarized results and comparisons with related works are presented in Table \ref{tab: results summary}.
\paragraph{Experiments.} To validate our findings in real-world scenarios with heavy-tailed noise and generalized smoothness, in~\Cref{sec:experiments} we evaluate the sign-based methods on Transformer models, specifically on pre-training LLaMA\citep{llama} family models of sizes up to 1.3B on the C4 dataset~\citep{c4} and the Switch Transformer~\citep{fedus2022switch} Mixture of Experts (MoE) model on the FineWeb dataset~\citep{penedo2406fineweb}. 
Results demonstrate the effectiveness of sign-based methods compared to other commonly considered techniques to cope with heavy-tailed noise, namely, clipping and normalization. 
Surprisingly, our results also show that \algname{M-SignSGD} demonstrates competitive performance and slight improvements compared to \algname{AdamW}, which is the de facto optimizer for language model training.

\begin{table}[]
\tiny
\caption{Convergence guarantees for non-convex optimization.  The metrics are Avr. $\ell_1$ : $\frac{1}{T} \sum_{k=1}^{T}  \|\nabla f(x^k)\|_1 \leq \varepsilon$, Avr. $\ell_2$: $\frac{1}{T} \sum_{k=1}^{T}  \|\nabla f(x^k)\|_2 \leq \varepsilon$, Avr. $\ell_2^2$: $\frac{1}{T} \sum_{k=1}^{T}  \|\nabla f(x^k)\|^2_2 \leq \varepsilon^2$, Func. acc. :
$f(x^T) - f(x^*) \leq \varepsilon.$ HP stands for bounds with probability at least $1- \delta$, $\EE$ stands for in expectation bounds. 
}\label{tab: results summary}
\begin{tabular}{|c|c|c|c|c|}
\hline
\multicolumn{1}{|c|}{\textbf{Method}}                                       & \multicolumn{1}{c|}{\textbf{Complexity Bound}} & \multicolumn{1}{c|}{\textbf{Smoothness}} & \multicolumn{1}{c|}{\textbf{Noise type}} & \multicolumn{1}{c|}{\textbf{Metric}} \\ \hline \hline

\algname{NSGD} \cite{hubler2024gradient}                                   &            $O\left(\frac{\Delta L_0}{\varepsilon^2}\left[1 +  \left(\frac{\|\Vec{\sigma}\|_2}{\varepsilon}\right)^\frac{\kappa}{\kappa-1}\right]\right)$                         & \color{BrickRed} $L_0 $                        & \color{ForestGreen} $\kappa \in (1,2]$                    & \color{ForestGreen} HP Avr. $\ell_2$  \\ \hline
\algname{ClipSGD} \cite{nguyen2023improved}                                &  $O\left(\left(\frac{ ||\Vec{\sigma}||_2^\kappa \log \frac{1}{\delta}}{\sqrt{\Delta L_0}} \right)^\frac{3\kappa -2}{\kappa - 1} \left(\frac{ \sqrt{\Delta L_0}\log \frac{1}{\delta} }{\varepsilon^2}\right)^\frac{3\kappa - 2}{2\kappa - 2}\right)$  \label{eq: lol}                         & \color{BrickRed} $L_0$                               &           \color{ForestGreen} $\kappa \in (1,2]$                 &                    \color{ForestGreen} HP Avr. $\ell_2^2$       \\ \hline

\algname{ClipSGD} \cite{koloskova2023revisiting}                                   &            $O\left(\frac{\Delta L_1\|\Vec{\sigma}\|^4_2}{\varepsilon^5} + \frac{ \Delta(1+\|\Vec{\sigma}\|_2)(L_0   + L_1\varepsilon)}{\varepsilon^4} \right)$                         & \color{ForestGreen} $(L_0, L_1)$                        & \color{BrickRed} $\kappa = 2$                    & \color{BrickRed} $\EE$ Avr. $\ell_2$ \\ \hline
\algname{D-AdaGrad}  \cite{crawshaw2025complexitylowerboundsadaptive} & $\tilde{\Omega}\left(\frac{\Delta^2 L_0^2 \|\Vec{\sigma}\|^2_2}{\varepsilon^4} + \frac{\Delta^2 L_1^2 \|\Vec{\sigma}\|^2_2}{\varepsilon^2\log(1 + \Delta L_1^2/L_0)} \right), \varepsilon \leq \Delta L_1+ \|\Vec{\sigma}\|_2$  & \color{ForestGreen}  $(L_0,L_1)$ & \color{BrickRed}  Bounded & \color{ForestGreen} HP Avr. $\ell_2$  \\ \hline
\algname{AdaGrad-Norm}  \cite{wang2023convergenceadagrad} &
$\tilde{O}\left(\frac{\Delta^2 L_1^2 \|\Vec{\sigma}\|^2_2}{\varepsilon^4} + \frac{\Delta L_0 \|\Vec{\sigma}\|^2_2}{\varepsilon^4} + \frac{\|\Vec{\sigma}\|^6_2 }{\delta^4\varepsilon^4} \right)$   & \color{ForestGreen}  $(L_0,L_1)$ & \color{BrickRed}  $\underset{\text{affine}}{\kappa = 2}$ & \color{ForestGreen} HP Avr. $\ell_2$  \\ \hline
\algname{D-AdaGrad-Norm}  \cite{crawshaw2025complexitylowerboundsadaptive} &
$\tilde{\Omega}\left(\frac{\Delta^2 L_1^2 \|\Vec{\sigma}\|^2_2}{\varepsilon^4} + \frac{\Delta L_0 \|\Vec{\sigma}\|^2_2}{\varepsilon^4} + \frac{\Delta^2 L_1^2 }{\varepsilon^2} \right), \varepsilon \leq \sqrt{\Delta L_0}, \Delta L_1$   & \color{ForestGreen}  $(L_0,L_1)$ & \color{BrickRed}  Bounded & \color{ForestGreen} HP Avr. $\ell_2$  \\ \hline

 \pbox{20cm}{\algname{minibatch-SignSGD} \\ \textbf{(Theorem \ref{thm:minibatch SignSGD})}}   \cellcolor{bgcolor}
                                                                   &  \cellcolor{bgcolor}  $O\left(\left(\frac{\Delta L_0d}{\varepsilon^2} + \frac{\Delta L_1d^\frac32}{\varepsilon}\right)\left[1 +  \left(\frac{\|\Vec{\sigma}\|_1}{\varepsilon}\right)^\frac{\kappa}{\kappa-1}\right]\log \frac{1}{\delta}\right)$                      & \cellcolor{bgcolor} \color{ForestGreen}$(L_0,L_1)$                        & \cellcolor{bgcolor} \color{ForestGreen} $\kappa \in (1,2]$                    & \cellcolor{bgcolor} \color{ForestGreen} HP Avr. $\ell_1$

\\ \hline
\multicolumn{5}{|c|}{\textbf{Momentum methods}} \\ \hline
\algname{M-NSGD} \cite{hubler2024parameter}                                   &     $\tilde{O}\left(\frac{(\Delta L_1 

+ \|\Vec{\sigma}\|_2 + L_0/L_1)^4}{\varepsilon^4}\right)$                               & \color{ForestGreen} $(L_0, L_1)$                        & \color{BrickRed}$\kappa = 2$                    &\color{BrickRed} $\EE$ Avr. $\ell_2$ \\ \hline
\algname{M-NSGD} \cite{hubler2024gradient}                                   &     $O\left(\frac{\Delta L_0}{\varepsilon^2}\left[1 +  \left(\frac{\|\Vec{\sigma}\|_2}{\varepsilon}\right)^\frac{\kappa}{\kappa-1}\right]\right)$                               & \color{BrickRed} $L_0$                        & \color{ForestGreen}$\kappa \in (1,2]$                    &\color{BrickRed} $\EE$ Avr. $\ell_2$ \\ \hline
\cellcolor{bgcolor} \pbox{20cm}{\algname{M-SignSGD} \\ \textbf{(Theorem \ref{thm:momentum SignSGD})}}                                    &  \cellcolor{bgcolor}   $O\left(\left(\frac{\Delta L_0d}{\varepsilon^2} + \frac{\Delta L_1d}{\varepsilon}\right)\left[1 +  \left(\frac{\sqrt{d}\|\Vec{\sigma}\|_\kappa}{\varepsilon}\right)^\frac{\kappa}{\kappa-1}\right]\right)$                                         & \cellcolor{bgcolor}  \color{ForestGreen} $(L_0, L_1)$                        & \cellcolor{bgcolor} 
 \color{ForestGreen}$\kappa \in (1,2]$                    & \cellcolor{bgcolor}  \color{BrickRed} $\EE$ Avr. $\ell_1$\\ \hline

\multicolumn{5}{|c|}{\textbf{Polyak-Lojasiewicz  functions (Assumption~\ref{as: PL}, $\mu  > 0$)}} \\ \hline
 \multicolumn{1}{|c|}{\algname{ClipSGD}\cite{sadiev2023high}} & \multicolumn{1}{c|}{$\tilde{O}\left( \frac {  L_0}{\mu} \left[1 + \left(\frac{L_0\|\Vec{\sigma}\|^2_2}{\mu^2\varepsilon}\right)^\frac{\kappa}{2(\kappa-1)}\right]\right)$  }         & \color{BrickRed} $L_0$          & \color{ForestGreen} $\kappa \in (1,2]$   \   & \color{ForestGreen} HP Func. acc.  \color{ForestGreen}      \\ \hline
\pbox{20cm}{\algname{Restarted-SignSGD} \\ \textbf{(Theorem \ref{thm:restarted minibatch SignSGD})}} \cellcolor{bgcolor}                      &     $\tilde{O}\left( \left(\frac{L_0d }{\mu}  + \frac {  L_1 d^\frac32 \sqrt{\Delta}}{\sqrt{\mu}} \right)\left[1 + \left(\frac{\|\Vec{\sigma}\|^2_1}{\mu \varepsilon}\right)^\frac{\kappa}{2(\kappa-1)}\right]\right)$      \cellcolor{bgcolor}  & \color{ForestGreen} \cellcolor{bgcolor}  $(L_0,L_1) $                        & \color{ForestGreen} $\kappa \in (1,2]$   \cellcolor{bgcolor}                 & \color{ForestGreen} HP Func. acc.     \cellcolor{bgcolor}   \\ \hline

\multicolumn{5}{|c|}{\textbf{Symmetric and unimodal noise}} \\\hline
\pbox{20cm}{\algname{MajorityVote-SignSGD} \\ \cite{bernstein2018signsgd, bernstein2018majorityvote} }                    &    $O\left(\frac{\Delta L_0d }{\varepsilon^2}\left[1 +  \left(\frac{\|\Vec{\sigma}\|_1}{\varepsilon}\right)^2\right]\right) $                      &  \color{BrickRed} $L_0$                        &  \color{BrickRed} $\underset{\text{symmetric}}{\kappa  = 2}$                     &\color{BrickRed} $\EE$ Avr. $\ell_1$\\ \hline

\cellcolor{bgcolor} \pbox{20cm}{\algname{MajorityVote-SignSGD} \\ \textbf{(Theorem \ref{thm:com-sign conv})} }                   &   \cellcolor{bgcolor} $O\left(\left(\frac{\Delta L_0 d}{\varepsilon^2} + \frac{\Delta L_1d^\frac32}{\varepsilon}\right)\left[\frac{1}{\kappa^2} +  \left(\frac{\|\Vec{\sigma}\|_1}{\varepsilon}\right)^2\right]\log \frac{1}{\delta}\right)$                      & \cellcolor{bgcolor} \color{ForestGreen} $(L_0,L_1)$                        & \cellcolor{bgcolor} \color{ForestGreen} $\underset{\text{symmetric}}{\kappa \in (0,2]}$                     & \cellcolor{bgcolor} \color{ForestGreen} HP Avr. $\ell_1$ \\\hline  
\end{tabular}

\end{table}




\subsection{Notations} The notation $\overline{1,n}$ represents the set of natural numbers $\{1, 2, \dots, n\}$.  We define $\ell_p$-norm $p \in [1,+\infty]$ as $(\|x\|_p)^p := \sum_{i=1}^d |x_i|^p, x \in \R^d$.  The notation $\la x, y \ra := \sum_{i=1}^{d} x_i y_i$ denotes the standard scalar product for $x,y \in \R^d$. The sign operator $\sign(\cdot)$ returns the sign of a scalar input and can also be applied element-wise to a vector. The notation $\widetilde\cO$ omits the logarithmic factors.

\section{High probability bounds for sign-based methods under heavy-tailed noise and $(L_0, L_1)$-smoothness}\label{sec: sec_2}

In this section, we present our novel non-convex convergence guarantees with high probability for  \algname{SignSGD} with batching and majority voting.
We prove them for $(L_0, L_1)$-smooth objective functions with  heavy-tailed noise in gradient estimates. We provide  the best convergence rates and optimal parameters or rates under arbitrary parameters. All proofs are located in Appendix \ref{app: proofs}.

\subsection{Assumptions}
\begin{assumption}[Lower bound]\label{as: bounded}
    The objective function $f$ is lower bounded by $f^* > -\infty$.
\end{assumption}
We use the following formulation of $(L_0, L_1)$-smoothness from \cite{gorbunov2024methods}. 
\begin{assumption}[$(L_0,L_1)$-smoothness]\label{as: smooth}
    The objective function $f$ is differentiable and $(L_0, L_1)$-smooth, i.e., for the non-negative constants $(L_0, L_1)$ and $ x, y \in \R^d$, it holds
    $$
\|\nabla f(x) - \nabla f(y)\| \leq (L_0 + L_1  \sup_{u \in \left[x, y \right]}\|\nabla f(u)\|) \|x - y\|. 
\quad $$
\end{assumption}
For examples of $(L_0, L_1)$-smooth functions and their properties, we refer the reader to Appendix \ref{sunbsec: gen smoothness}.
\begin{assumption}[Heavy-tailed noise in gradient estimates]\label{as: pBCM}
    The unbiased estimate $\nabla f (x, \xi)$  has bounded $\kappa$-th moment $\kappa \in (1,2]$ for each coordinate, i.e., $\forall x \in \R^d$: 
    \begin{eqnarray}
        \EE_\xi [\nabla f (x, \xi)] = \nabla f(x), \quad \EE_\xi [|\nabla f (x, \xi)_i - \nabla f(x)_i|^\kappa] \leq \sigma_i^\kappa, i \in \overline{1,d},
    \end{eqnarray}
    where $\Vec{\sigma} = [\sigma_1, \dots, \sigma_d]$ are non-negative.
\end{assumption}

\subsection{HP convergence properties of the backbone \algname{SignSGD} method} 
We begin our analysis with the simplest of sign-based methods, namely \algname{SignSGD} (Alg. \ref{alg: signSGD}) and prove a general lemma on its convergence with high probability.
\begin{algorithm}[ht!]
\caption{\algname{SignSGD} }
\label{alg: signSGD}   
\begin{algorithmic}[1]
\REQUIRE Starting point $x^1 \in \R^d$, number of iterations $T$, stepsizes  $\{\gamma_k\}_{k=1}^{T}$.

\FOR{$k=1,\ldots, T$}
\STATE Sample $\xi^k$ and compute estimate $x^{k+1} = x^k - \gamma_k \cdot \sign(\nabla f(x^k, \xi^k))$;
\ENDFOR
\ENSURE uniformly random point from $\{x^1, \dots, x^T\}$ . 
\end{algorithmic}
\end{algorithm}
\vspace{-2mm}
\begin{lemma}[\textbf{\algname{SignSGD} Convergence Lemma}] \label{lem: signsgd T update}
Consider lower-bounded $(L_0, L_1)$-smooth function $f$ (As. \ref{as: bounded}, \ref{as: smooth}) and HT gradient estimates $\Vec{\sigma}_k$ (As. \ref{as: pBCM}). Then Alg. \ref{alg: signSGD} after $T$ iterations with non-increasing stepsizes $\gamma_k \leq 1/ (48L_1d^\frac32\log\frac1\delta)$ achieves with probability at least $1 - \delta$:
\begin{equation}
 \sum\limits_{k=1}^T \frac{\gamma_k}{16}\|\nabla f (x^k)\|_1 \leq \Delta + L_0d\sum_{k=1}^T\gamma_k^2 + 2\sum_{k=1}^T\gamma_k \|\Vec{\sigma}_k\|_1 
       + 6d(\gamma_1 \|\nabla f (x^1)\|_1  + 2C_TL_0) \log\frac{1}{\delta}, \label{eq: signsgd convergence lemma}
\end{equation}
where $C_T := \max\limits_{k \in \overline{1,T}} \gamma_k \cdot  \sum\limits_{\tau=1}^{k-1}\gamma_\tau$ and $\Delta = f(x^1) - f^*$.
\end{lemma}

The bound \eqref{eq: signsgd convergence lemma} resembles the convergence bound in expectation for \algname{SignSGD} for $\kappa=2$ \cite{bernstein2018signsgd}. The difference is the last term with linear dependence on $\log\frac{1}{\delta}$. Remarkably, $L_1$  constant affects only the upper bound for the largest possible stepsizes $\gamma_k \leq 1/ (48L_1d^\frac32\log\frac1\delta)$. In the case of small $L_0$, this is the only condition that prevents us from increasing the stepsize too much. We provide synthetic experiments verifying dependencies in bound \eqref{eq: signsgd convergence lemma} in Appendix \ref{sec: exps for theory}.

In order to upper bound an average accuracy norm  from \eqref{eq: signsgd convergence lemma} by $\varepsilon$, the noise $\|\Vec{\sigma}\|_1$ has not to exceed $\varepsilon$. The first way to lower the noise is to use batch averaging. 
\vspace{-2mm}
\subsection{New HP bounds for $(L_0,L_1)$-smoothness for \algname{SignSGD} with average batching}\label{sec:minibatch sign sgd}

\begin{algorithm}[ht!]
\caption{\algname{minibatch-SignSGD} }
\label{alg:minibatch-signSGD}   
\begin{algorithmic}[1]
\REQUIRE Starting point $x^1 \in \R^d$, number of iterations $T$, stepsizes  $\{\gamma_k\}_{k=1}^{T}$, batchsizes $\{B_k\}_{k=1}^{T}$.

\FOR{$k=1,\ldots, T$}
\STATE Sample $\{\xi^k_i\}_{i=1}^{B_k}$ and compute  $x^{k+1} = x^k - \gamma_k \cdot \sign(\sum_{i=1}^{B_k} \nicefrac{\nabla f(x^k, \xi^k_i)}{B_k})$;
\ENDFOR
\ENSURE uniformly random point from $\{x^1, \dots, x^{T}\}$ . 
\end{algorithmic}
\end{algorithm}

\begin{theorem}[\textbf{HP complexity for \algname{minibatch-SignSGD}}]\label{thm:minibatch SignSGD}
Consider lower-bounded $(L_0,L_1)$-smooth function $f$ (As. \ref{as: bounded}, \ref{as: smooth}) and HT gradient estimates (As. \ref{as: pBCM}). Then Alg. \ref{alg:minibatch-signSGD} requires the sample complexity $N$  to achieve $\frac{1}{T} \sum_{k=1}^{T}  \|\nabla f(x^k)\|_1 \leq \varepsilon$ with probability at least $1-\delta$ for:

\textbf{Optimal tuning:}  $T = O\left(\frac{\Delta L_1^\delta d^\frac{3}{2} }{\varepsilon}\right), \gamma_k \equiv \frac{1}{48 L_1^\delta d^\frac32} , B_k \equiv  \left(\frac{16\|\Vec{\sigma}\|_1}{\varepsilon}\right)^\frac{\kappa}{\kappa-1}$ for $\varepsilon \geq \frac{8L_0}{L_1\sqrt{d}}$ and $T = O\left(\frac{L_0^\delta d }{\varepsilon^2}\right), \gamma_k \equiv \sqrt{\frac{\Delta}{20 L_0^\delta dT}} , B_k \equiv  \left(\frac{16\|\Vec{\sigma}\|_1}{\varepsilon}\right)^\frac{\kappa}{\kappa-1} $ for $\varepsilon \leq \frac{8L_0}{L_1\sqrt{d}}$:
\begin{equation}
   N = O\left(\left(\frac{\Delta L_0 d }{\varepsilon^2} + \frac{\Delta L_1   d^\frac{3}{2}}{\varepsilon}\right)\left[1 +  \left(\frac{\|\Vec{\sigma}\|_1}{\varepsilon}\right)^\frac{\kappa}{\kappa-1}\right]\log \nicefrac{1}{\delta}\right), \label{eq: sign SGD optimal} 
\end{equation}
\textbf{Arbitrary tuning:\footnote{These bounds are proved for a metric $\min_{k \in \overline{1,T}}\|\nabla f(x_k)\|_1  \leq \varepsilon$.}}  Until plateau $\gamma_k = \gamma_0 \leq \frac{1}{48L^\delta_1d^\frac32}, B_k = B_0k^2$, after $\gamma_k = \frac{\gamma_0}{\sqrt{k}}, B_k = B_0k$:
\begin{eqnarray}
    \varepsilon \geq \frac{8L_0}{L_1\sqrt{d}} &\Rightarrow& N = \tilde{O}\left( B_0\left(\frac{\Delta}{\gamma_0\varepsilon} 
 \right)^3 + \frac{1}{B_0^2}\left(\frac{\|\Vec{\sigma}\|_1}{\varepsilon}\right)^\frac{3\kappa}{2(\kappa - 1)}\right), \notag \\
    \varepsilon \ll \frac{8L_0}{L_1\sqrt{d}} &\Rightarrow& N =  \tilde{O}\left(\frac{B_0(L_0^\delta \gamma_0 d + \Delta/\gamma_0)^4}{\varepsilon^4}  +  \frac{1}{B_0}\left(\frac{\|\Vec{\sigma}\|_1}{\varepsilon}\right)^\frac{2\kappa}{\kappa - 1} \right), \label{eq: minibatch signSD arb tuining}
\end{eqnarray} 
where $\Delta = f(x^1) - f^*, L_0^\delta = L_0 \log(\nicefrac{1}{\delta}), L_1^\delta = L_1 \log(\nicefrac{1}{\delta}).$
\end{theorem}

\vspace{-1mm}
Proof, optimal tuning for infinite horizon and arbitrary tuning for finite horizon are in Appendix \ref{subsec: minibatch signsgd proof}.

\subsubsection{Discussion}

\paragraph{Optimal tuning bounds.} From Theorem \ref{thm:minibatch SignSGD}, we can clearly distinguish two phases of algorithm convergence: fast initial phase with rate $\tilde{O}\left(\varepsilon^{-\nicefrac{2\kappa - 1}{\kappa - 1}}\right)$ before threshold $\varepsilon \geq \nicefrac{8L_0}{L_1\sqrt{d}}$ and slower one with rate $\tilde{O}\left(\varepsilon^{-\nicefrac{3\kappa - 2}{\kappa - 1}}\right)$ after. We provide synthetic experiments verifying two stage convergence and batching effectiveness in Appendix \ref{sec: exps for theory}.

In the case of $L_0 \approx 0$ (e.g. for logistic regression \cite{gorbunov2024methods} and deep neural networks \cite{zhang2020gradient}), \algname{minibatch-SignSGD} runs in the fast regime the whole time and can work with large constant stepsizes. Otherwise, under standard smoothness $L_1 = 0$, the bound \eqref{eq: sign SGD optimal} matches the lower in expectation bound $\Omega \left( \nicefrac{\Delta L_0}{\varepsilon^2} + \nicefrac{\Delta L_0}{\varepsilon^2}\left(\nicefrac{\|\Vec{\sigma}\|_2}{\varepsilon} \right)^{\frac{\kappa}{\kappa - 1} }\right)$ for first-order stochastic  optimization \cite{zhang2020adaptivegood}. We also wish to highlight the linear dependency of \eqref{eq: sign SGD optimal}  on $\Delta,  L_0, L_1, \log\nicefrac{1}{\delta}$ and the mild dependency on $\|\Vec{\sigma}\|_1$.

\paragraph{Extra $d$ factors.} In bound \eqref{eq: sign SGD optimal}, there are extra $d$ factors that are missing in bounds for $\ell_2$-norm (Table~\ref{tab: results summary}). Indeed, in sign-based methods, we usually work with larger $\ell_1$-norm instead of $\ell_2-$norm, i.e., $\|x\|_2 \leq \|x\|_1 \leq \sqrt{d}\|x\|_2, \forall x \in \R^d.$ Thus, in order to achieve $\varepsilon'$ accuracy in the $\ell_2$-norm, accuracy $\varepsilon$ in the $\ell_1$-norm has to be $\varepsilon \sim   \varepsilon'\cdot\sqrt{d}$, especially for dense gradient vectors. For dense corrupting noise, we similarly have   $\|\Vec{\sigma}\|_{1} \sim \sqrt{d}\|\Vec{\sigma}\|_{2}$. In \cite{bernstein2018signsgd}, the authors show that gradients and noise during DL model training via \algname{SignSGD} actually keep high density. Putting $\varepsilon \sim   \varepsilon'\cdot\sqrt{d}$ and $\|\Vec{\sigma}\|_{1} \sim \sqrt{d}\|\Vec{\sigma}\|_{2}$ into \eqref{eq: sign SGD optimal}, we see that the only remaining factor is $L_1d$.
 
\paragraph{Comparison under standard smoothness.} According to the HP analysis  of \algname{ClipSGD} from \cite{nguyen2023improved}, it achieves the rates from Table \ref{tab: results summary}. These rates have optimal $\tilde{O}\left(\varepsilon^{-\nicefrac{3\kappa - 2}{\kappa - 1}}\right)$ dependency on $\varepsilon$, however, dependencies on $\Delta, L_0, \Vec{\sigma}$ are much worse than ours. Moreover, \algname{ClipSGD} requires careful clipping level scheduling, and we are not aware of any works proposing arbitrary tuning for clipping methods. In \cite{hubler2024gradient}, the authors analyze HP convergence of \algname{minibatch-NSGD} and obtain the sample complexity w.r.t. to the $\ell_2$-norm (Table \ref{tab: results summary}), the only difference from \eqref{eq: sign SGD optimal} is the absence of $d$ factors. As explained before, this difference results from the distinct norms in the bounds. From a practical point of view, sign-based methods can be applied to distributed optimization (Appendix \ref{sec: distributed}) where normalization does not fit. Besides, one can use majority voting as a more powerful alternative to batching.  

\paragraph{Comparison under generalized smoothness.} Under BV noise, the results for \algname{CLipSGD} \cite{reisizadeh2025variance} in case $\varepsilon \leq \nicefrac{L_0}{L_1}$ match our bound \eqref{eq: sign SGD optimal}. However, as $L_0$ becomes smaller our bound has better dependencies on $\varepsilon$ and $\sigma$ than \cite{koloskova2023revisiting}. In comparison with adaptive methods \cite{wang2023convergenceadagrad, crawshaw2025complexitylowerboundsadaptive} from Table \ref{tab: results summary}, our results remain valid for all possible values of parameters $L_0, L_1, \|\Vec{\sigma}\|_2, \Delta \geq 0$ and accuracy $\varepsilon$. Moreover, our bound demonstrates a milder linear dependency on $\Delta, L_0, L_1$ and $\log$-dependency on probability $\delta$ instead of a polynomial dependency.

\paragraph{Arbitrary tuning bounds.} For our methods, we use dynamic arbitrary tuning. We start with the largest stepsize $\gamma_0$ for which the method converges and continue until it reaches the plateau $dL_0\gamma_0$ observed in the convergence bound \eqref{eq: signsgd convergence lemma}. After that, we begin to decrease the stepsizes as $1/\sqrt{k}$. Arbitrary tuning results in worse polynomial dependency on parameters $L_0,L_1, \Delta$ and slower (for $\kappa \neq 2$) non-optimal rates $\tilde{O}(\varepsilon^{-\frac{3\kappa}{2(\kappa - 1)}})$ and $\tilde{O}(\varepsilon^{-\frac{2\kappa}{\kappa - 1}})$  instead of $O(\varepsilon^{-\frac{2\kappa - 1}{\kappa - 1}})$ and $O(\varepsilon^{-\frac{3\kappa  - 2}{\kappa - 1}})$. 

In \cite{hubler2024parameter}, the authors apply simple arbitrary tuning to \algname{M-NSGD} and obtain in BV expectation bound $\tilde{O}((\Delta/\gamma_0 + \gamma_0 L_0)e^{(\gamma_0L_1)^2} + \|\Vec{\sigma}\|_2)^4/\varepsilon^4)$ without any restrictions on $\gamma_0$. If $\gamma_0 \leq 1/L_1$, their bounds match our HP rates \eqref{eq: minibatch signSD arb tuining} ($\kappa = 2$) during the second convergence phase, however, only we consider the first faster phase with $ \tilde{O}\left(\left(\nicefrac{\Delta }{\gamma_0\varepsilon} 
 \right)^3 + \left({\|\Vec{\sigma}\|_1}/{\varepsilon}\right)^3\right)$ rates. 


\paragraph{Practical heuristics justification.} Using arbitrary tuning bounds \eqref{eq: minibatch signSD arb tuining}, we can explain why popular practical heuristics for training neural networks such as grid search of hyperparmaters and decreasing stepsizes successfully work in real-world problems. 

First, one can grid search hyperparameters (e.g., initial stepsize $\gamma_0$) and choose the best ones according to the achieved final accuracy. Theorem \ref{thm:minibatch SignSGD} guarantees convergence to any accuracy for almost all $\gamma_0, B_0$, and the only difference is the convergence speed. Hence, trying various values reveals which values pair better with the unknown problem parameters $L_0, L_1, \Vec{\sigma}$ in bounds \eqref{eq: minibatch signSD arb tuining}.   

Second, one can start to decrease stepsizes at any moment before the oscillating plateau, even from the beginning of the training. In this case, the initial fast convergence phase is not fully utilized, and the slower speed $\tilde{O}(\varepsilon^{-\frac{2\kappa}{\kappa - 1}})$ from \eqref{eq: minibatch signSD arb tuining} comes earlier for all accuracies $\varepsilon > 0$.

\subsubsection{Polyak-Lojasiewicz functions.}  The \algname{minibathc-SignSGD} algorithm can be accelerated for the special class of generalized smooth functions that satisfy the Polyak-Lojasiewicz condition.  
\begin{assumption}[Polyak-Lojasiewicz (PL)]\label{as: PL}
    The objective function $f$ satisfies the PL condition, i.e., for the non-negative constant $\mu$ and $ x \in\R^d$, it holds
    $$
\|\nabla f(x)\|^2_2 \geq 2\mu(f(x) - f(x^*)). 
 $$
\end{assumption}
For example, $\mu$-strongly convex functions satisfy the PL condition. A similar behavior has also been observed in over-parameterized models \cite{liu2022loss}. For these functions, we use restarts on \algname{minibatch-SignSGD} to achieve the following HP function accuracy $f(x^T) - f(x^*) \leq \varepsilon$. The explicit algorithm and parameters are presented in Theorem \ref{thm:restarted minibatch SignSGD}, Appendix \ref{subsec:restarted}, and it achieves the bounds 
\begin{eqnarray}
    N = \tilde{O}\left( \left(\frac{L_0d }{\mu}  + \frac {  L_1 d^\frac32 \sqrt{\Delta}}{\sqrt{\mu}} \right)\left[1 + \left(\frac{\|\Vec{\sigma}\|^2_1}{\mu \varepsilon}\right)^\frac{\kappa}{2(\kappa-1)}\right]\right). \label{eq: restarts bounds main} 
\end{eqnarray}

For the standard smoothness $L_1 = 0$, we compare our method with the most related \algname{ClipSGD} with the complexity bound $\tilde{O}\left( \frac {  L_0}{\mu} \left[1 + \left(\frac{L_0\|\Vec{\sigma}\|^2_2}{\mu^2\varepsilon}\right)^\frac{\kappa}{2(\kappa-1)}\right]\right)$\cite{sadiev2023high}. In contrast, \algname{minibatch-SignSGD} does not require adjusting the clipping schedule and has better $L_0/\mu$ dependency. However, our bound has to be restarted and has extra $d$ factors and larger $\|\Vec{\sigma}\|^2_1$ variance, which remain even after considering lower $\ell_2$-norm accuracy $\varepsilon = \varepsilon'\sqrt{d}$.

\subsection{\algname{SignSGD} with majority voting for symmetric HT noise}\label{sec: signsggd with majority}

The second approach to noise reduction inherent to sign-based methods is majority voting.

\paragraph{Majority voting.} As mentioned above, the original motivation of \algname{SignSGD} is fast communication in distributed optimization \cite{bernstein2018majorityvote,jin2020stochastic}. 
In the literature, various types of communication were studied, but the most effective one turned out to be majority voting. For sign vectors $\sign(g^{k}_i), i \in \overline{1,M}$, the resulting update vector is the majority of the received signs $g^k = \sign(\sum_{i=1}^M \sign(g^{k}_i)).$

To be effective, majority voting must decrease the probability of failure of the resulting vector with the growth of $M.$ However, for very skewed or bimodal random variables, it might not be true. Choosing the most frequent value from the sign sequence $\{\sign(g^k_i)\}_{i=1}^M$ is actually $M$ Bernoulli trials. In these trials,  the probability of choosing a correct answer grows only if the probability of failure of a single worker is less than $\frac{1}{2}$, i.e.: $\mathbb{P}\left[\sign(\nabla f (x^k)) \neq \sign(g^k_i) \right] < \frac12, \forall i \in \overline{1,M}.$
For example, this condition is satisfied if the noise of the gradient estimate for each coordinate is \textit{unimodal and symmetric about its true value}. We use this assumption, but other assumptions \cite{safaryan2021stochastic} are valid as well.
\begin{algorithm}[ht!]
\caption{\algname{MajorityVote-SignSGD} }
\label{alg:majorityvotesignSGDsingle}   
\begin{algorithmic}[1]
\REQUIRE Starting point $x^0 \in \R^d$, number of iterations $T$, stepsizes  $\{\gamma_k\}_{k=1}^{T}$, batchsizes $\{M_k\}_{k=1}^{T}$.

\FOR{$k=1,\ldots, T$}
\STATE Sample $\{\xi^k_i\}_{i=1}^{B_k}$ and  compute  $x^{k+1} = x^k - \gamma_k \cdot\sign\left(\sum_{i=1}^{M_k} \sign(\nabla f(x^k, \xi^k_i))\right)$;
\ENDFOR
\ENSURE uniformly random point from $\{x^1, \dots, x^{T}\}$ . 
\end{algorithmic}
\end{algorithm}


\begin{theorem}[\textbf{HP complexity for \algname{MajorityVote-SignSGD}}]\label{thm:com-sign conv}
  Consider lower-bounded $(L_0, L_1)$-smooth function $f$ (As. \ref{as: bounded}, \ref{as: smooth}) and the gradient estimates corrupted by \textbf{unimodal and
symmetric HT noise with $\kappa > 0$} (As. \ref{as: pBCM}). Then Alg. \ref{alg:majorityvotesignSGDsingle} requires the sample complexity $N$ to achieve $\frac{1}{T} \sum_{k=1}^{T}  \|\nabla f(x^k)\|_1 \leq \varepsilon$ with probability at least $1-\delta$ for:

\textbf{Optimal tuning:}  $T = O\left(\frac{\Delta L^\delta_1 d^\frac{3}{2} }{\varepsilon}\right), \gamma_k \equiv \frac{1}{48 L_1^\delta d^\frac32} , M_k \equiv \max \left\{\frac{160}{\kappa^2}, \frac{2^{16}\|\Vec{\sigma}\|^2_1}{\varepsilon^2}\right\}$ for $\varepsilon \geq \frac{8L_0}{L_1\sqrt{d}}$ and $T = O\left(\frac{\Delta L_0^\delta d }{\varepsilon^2}\right), \gamma_k \equiv \sqrt{\frac{\Delta}{ 80L_0^\delta dT}} , M_k \equiv \max \left\{\frac{160}{\kappa^2}, \frac{2^{16}\|\Vec{\sigma}\|^2_1}{\varepsilon^2}\right\}$ for $\varepsilon \leq \frac{8L_0}{L_1\sqrt{d}}$:
\begin{equation}
   N = O\left(\left(\frac{\Delta L_0 d }{\varepsilon^2} + \frac{\Delta L_1d^\frac32}{\varepsilon}\right)\left[\frac{1}{\kappa^2} +  \left(\frac{\|\Vec{\sigma}\|_1}{\varepsilon}\right)^2\right]\log\nicefrac{1}{\delta}\right), \label{eq: majority sign SGD optimal} 
\end{equation}
\textbf{Arbitrary tuning:\footnote{These bounds are proved for a metric $\min_{k \in \overline{1,T}}\|\nabla f(x_k)\|_1  \leq \varepsilon$.}}  Until plateau $\gamma_k = \gamma_0 \leq \frac{1}{48L^\delta_1d^\frac32}, M_k =  M_0(\frac{k}{\kappa})^2$, after $\gamma_k = \frac{\gamma_0}{\sqrt{k}}, M_k = \frac{M_0k}{\kappa^2}$:
\begin{eqnarray}
    \varepsilon \geq \frac{8L_0}{L_1\sqrt{d}} &\Rightarrow& N = \tilde{O}\left( \frac{M_0(\Delta/\gamma_0)^3 + \|\Vec{\sigma}\|_1^3/M_0^2}{\kappa^2\varepsilon^3} \right), \notag \\
    \varepsilon \ll \frac{8L_0}{L_1\sqrt{d}} &\Rightarrow& N =  \tilde{O}\left(\frac{M_0(L_0^\delta \gamma_0 d + \Delta/\gamma_0)^4 + \|\Vec{\sigma}\|_1^4/M_0}{\kappa^2\varepsilon^4}  \right), \notag
\end{eqnarray} 
where $\Delta = f(x^1) - f^*, L_0^\delta = L_0 \log(\nicefrac{1}{\delta}), L_1^\delta = L_1 \log(\nicefrac{1}{\delta}).$

\end{theorem}

The proof of Theorem \ref{thm:com-sign conv}, optimal tuning for infinite horizon, and arbitrary tuning for finite horizon are located in Appendix \ref{subsec: majority vote sign proofs}. For PL functions, we use the restart technique and achieve bounds similar to \eqref{eq: restarts bounds main} as if $\kappa = 2$. The results are presented in Theorem \ref{thm:restarted majority SignSGD} in Appendix \ref{subsec:restarted}.

\paragraph{Discussion.} Similar to previous works dedicated to symmetric HT noise \cite{armacki2024large, puchkin2024breaking, compagnoni2025unbiased}, the severity of the corrupting noise, namely, the value of $\kappa$ has much milder effect on convergence of \algname{MajorityVote-SignSGD} in comparison with \algname{minibatch-SignSGD} and its bound \eqref{eq: sign SGD optimal}.  Moreover, in the case of arbitrary tuning, despite of chosen parameters, $O(\varepsilon^{-4})$ dependency remains intact, while deterioration happens only in $\Delta, L_0, L_1,  \Vec{\sigma}$-depending factors.

Under standard smoothness $L_1 = 0$, the bounds \eqref{eq: majority sign SGD optimal} with linear $\log\nicefrac{1}{\delta}$ dependency matches the optimal bound $\Omega\left(\nicefrac{\Delta L_0\|\Vec{\sigma}\|_2^2}{\varepsilon^{4}}\right)$ in expectation for first-order non-convex methods under \textit{BV noise} \cite{ arjevani2023lower}. 


\subsection{\algname{SignSGD} with momentum and small batchsizes}\label{sec: MsignSGD}

Both \algname{minibatch-SignSGD} and \algname{minibatch-SignSGD} methods require increasing batch sizes comparable to the number of iterations. In order to avoid large batchsizes, one can utilize the momentum technique, resulting in the same total sample complexity.  The proof is located in Appendix \ref{subsec: MSignSGD proof}. 

\begin{algorithm}[ht!]
\caption{\algname{M-SignSGD} }
\label{alg:SignSGD-M}   
\begin{algorithmic}[1]
\REQUIRE Starting point $x^1 \in \R^d$, number of iterations $K$, stepsizes  $\{\gamma_k\}_{k=1}^{T}$, momentums $\{\beta_k\}_{k=1}^{T}$.

\FOR{$k=1,\ldots, T$}
\STATE Sample $\xi^k$ and compute $m^k = \beta_k m^{k-1} + (1-\beta_k) \nabla f(x^k, \xi^k)$;  
\STATE Set $x^{k+1} = x^k - \gamma_k \cdot \text{sign}(m^k)$;
\ENDFOR
\ENSURE uniformly random point from $\{x^1, \dots, x^{T}\}$ . 
\end{algorithmic}
\end{algorithm}
\begin{theorem}[\textbf{Complexity for \algname{M-SignSGD} in expectation}]\label{thm:momentum SignSGD}
Consider a lower-bounded $(L_0,L_1)$-smooth function $f$ (As. \ref{as: bounded}, \ref{as: smooth}) and HT gradient estimates (As. \ref{as: pBCM}). Then, Alg. \ref{alg:SignSGD-M} requires $T$ iterations  to achieve  $\frac{1}{T} \sum_{k=1}^{T}  \EE \left[ \|\nabla f(x^k)\|_1 \right]  \leq \varepsilon$ starting with $\Delta = f(x^1) - f^*$:

\textbf{Optimal tuning:} $\beta_k \equiv 1 - \min\left\{1, \left(\frac{\Delta L_1 \sqrt{d}}{T \|\Vec{\sigma}\|_\kappa}\right)^\frac{\kappa}{2\kappa - 1}\right\}, \gamma_k  \equiv \frac{1 - \beta_k}{8} \frac{1}{L_1d}$  for $\varepsilon \geq \frac{3L_0}{L_1}$ and $1  - \beta_k \equiv   1 - \min\left\{1, \left(\frac{\Delta L_0}{T \|\Vec{\sigma}\|_\kappa^2}\right)^\frac{\kappa}{3\kappa - 2}  \right\}, \gamma_k  \equiv \sqrt{\frac{\Delta (1 - \beta_k)}{T L_0 d}}$ for $\varepsilon \leq \frac{3L_0}{L_1}$:
\begin{equation}
     T = O\left(\left(\frac{\Delta L_0d}{\varepsilon^2 } +\frac{\Delta L_1d}{\varepsilon }\right)\left(1 + \left(\frac{\sqrt{d}\|\Vec{\sigma}\|_\kappa}{\varepsilon}\right)^\frac{\kappa }{\kappa  -1}\right)\right). \label{eq: MSignSGD optimal}
\end{equation}
\end{theorem}
\paragraph{Comparison with \algname{minibatch-SignSGD}.} In comparison with \algname{minibatch-SignSGD}, \algname{M-SignSGD} allows to use small constant batchsizes, which are more common in practice, especially in LLMs training. However, in theory, its bound \eqref{eq: MSignSGD optimal} in expectation  has large factor $\sqrt{d}\|\Vec{\sigma}\|_\kappa$  instead of $\|\Vec{\sigma}\|_1$ from the HP bound \eqref{eq: sign SGD optimal}. Nevertheless, for $\kappa$ close to $2$ and dense gradient noise $\Vec{\sigma}$, $\|\Vec{\sigma}\|_\kappa$ is exactly $\sqrt{d}$ times lower than $\|\Vec{\sigma}\|_1$ and the difference is fading. Also, the convergence during the first phase is $\sqrt{d}$ times faster, but the threshold $\varepsilon \geq 3L_0/{L_1}$ is larger in return.

\paragraph{Related works.}  In \cite{hubler2024parameter}, the authors work with \algname{M-NSGD} under BV noise and derive  $\tilde{O}\left({(\Delta L_1  + \|\Vec{\sigma}\|_2 + L_0/L_1)^4}/{\varepsilon^4}\right)$ bound in expectation. Our bound \eqref{eq: MSignSGD optimal} with $\kappa = 2$ has a milder dependency on parameters $\Delta,L_0,L_1$ and does not degenerate as $L_1 \to 0.$ Under standard smoothness, there exists an optimal bound for \algname{M-NSGD} with $\kappa \in (1,2]$ that matches our bound \eqref{eq: MSignSGD optimal} for $L_1 = 0$ up to $d$ factors (see Table \ref{tab: results summary}).  Our bounds are the first to combine HT noise and $(L_0, L_1)$-smoothness.

\vspace{-8pt}
\section{Experiments}\label{sec:experiments}

\begin{wrapfigure}[14]{r}{8cm}
\vspace{-11mm}
\begin{minipage}{8cm}
\begin{table}[H]
    \caption{Comparison of validation perplexity for various optimization methods across LLaMA model scales trained on C4}
    \label{tab:pre-training}
 
    \begin{center}
    \begin{tabular}{c|ccc}
    \toprule
    \textbf{Method} & \multicolumn{3}{c}{\textbf{Perplexity $\downarrow$}} \\
    \midrule
    Model size & 130M & 350M & 1.3B \\
    \midrule
    \algname{M-SignSGD} & \textbf{18.37}$_{\pm .01}$ & \textbf{13.73} & \textbf{11.56}\\
    \midrule
    \algname{M-NSGD} &  19.28$_{\pm .03}$ & 14.60 & 12.62 \\
    \algname{M-ClippedSGD} & 18.95$_{\pm .03}$ & 14.30 & 12.30 \\
    \algname{AdamW} & 18.67$_{\pm .00}$ & 13.78 & 11.57 \\
    \midrule
    Training tokens  & 10B & 30B & 30B \\
    Number of iterations  & 100k & 300k & 300k \\
    \bottomrule
    \end{tabular}
    \end{center}
    \vspace{-1.em}
\end{table}
\end{minipage}
\vspace{-5pt}
\end{wrapfigure}

In this section, we present experimental results for sign-based methods described in~\Cref{sec: sec_2}. 
To demonstrate the effectiveness of sign-based methods, we focus on language model training tasks. 
This choice is motivated by two factors: first, these tasks are known to exhibit heavy-tailed noise~\cite{zhang2020adaptivegood} and generalized smoothness~\cite{zhang2020gradient, liu2023preGenSmooth} characteristics, and second, they represent an important real-world application domain.

To evaluate the performance of \algname{M-SignSGD}~(\Cref{alg:SignSGD-M}) we adopt the established experimental setup from~\cite{relora}, training LLaMA-like models~\citep{llama} of various sizes --- up to 1.3B parameters --- on the Colossal Clean Crawled Corpus (C4) dataset~\citep{c4}. 
The C4 dataset represents a colossal, cleaned version of Common Crawl's web corpus, specifically designed for pre-training language models and word representations.

For our comparison, we focus on two key techniques for handling heavy-tailed noise: gradient clipping with momentum and gradient normalization with momentum. As representative methods, we choose \algname{M-ClippedSGD} \cite{zhang2020improved} and \algname{M-NSGD} \cite{cutkosky2020momentum}, respectively.
We also compare to \algname{AdamW}~\cite{loshchilov2017decoupled}, as a de-facto method for the first-order optimization algorithm for deep learning.
To ensure a fair comparison, we conduct an extensive grid search over key hyperparameters, including learning rate, weight decay, and clipping level. 
Detailed information on the final hyperparameter values and complete experimental setup is provided in~\Cref{app:pre-training}.

\Cref{tab:pre-training} presents final validation perplexity 
for each method. 
\algname{M-SignSGD} demonstrates superior performance over other heavy-tail mitigating baselines, aligning with our theoretical results.
Furthermore, to our surprise, we discovered that \algname{M-SignSGD} outperforms the strong AdamW baseline, despite careful hyperparameter tuning of the latter. 
These findings highlight the significant potential of \algname{M-SignSGD} for language model pretraining. 
Given these promising results on smaller LLaMA models, we invite the research community to further explore sign-based optimization methods for large-scale LLM training and other computationally demanding applications of practical importance.

To ensure the generalizability of our findings, we complemented our experiments with a new setup --- new architecture and data. 
We have switched model to the Switch Transformer MoE architecture~\citep{fedus2022switch}, and data to the FineWeb dataset~\citep{penedo2406fineweb}, a popular corpus for LLM pre-training.
Further details on experimental setup and results can be found in~\Cref{sec: moe exps}.

\bibliography{refs}
\bibliographystyle{plain}

\newpage

\appendix

\section{Proofs}\label{app: proofs}
\subsection{$(L_0, L_1)$-smoothness} \label{sunbsec: gen smoothness}

Standard $L$-smoothness assumes that the gradient of a function is globally Lipschitz continuous. However, this condition can be too restrictive in practice. Many functions arising in optimization, especially in Machine Learning and statistics, either do not satisfy $L$-smoothness or satisfy it with a very large constant $L_0$, leading to overly pessimistic theoretical guarantees.
$(L_0, L_1)$-smoothness (Assumption \ref{as: smooth}) is weaker than $L$-smoothness and allows finer control over the smoothness behavior of functions with rapidly growing curvature in regions where the gradient is large.

Importantly, many functions satisfy $(L_0, L_1)$-smoothness with \emph{significantly smaller constants} $L_0$ and $L_1$ compared to the $L$ required for global Lipschitz smoothness. As a result, optimization algorithms tailored for $(L_0, L_1)$-smooth functions can achieve better convergence guarantees, especially in settings involving large gradients or heavy-tailed noise. The examples of practically used $(L_0, L_1)$-smooth functions include:

\begin{example}[Power of Norm]
    Let $f(x) = \|x\|^{2n}$, where $n$ is a positive integer. Then, $f(x)$ is convex and $(2n,2n-1)$-smooth. Moreover, $f(x)$ is not $L$-smooth for $n\geq 2$ and any $L \geq 0$.
\end{example}

\begin{example}[Exponent of the Inner Product]
    Function $f(x) = \exp(a^\top x)$ for some $a\in \R^d$ is convex, $(0,\|a\|)$-smooth, but not $L$-smooth for $a\neq 0$ and any $L \geq 0$.
\end{example}

\begin{example}[Logistic Function] \label{exp: log func L1}
    Consider logistic function: $f(x) = \log\left(1 + \exp(-a^\top x)\right)$, where $a \in \R^d$ is some vector. It is known that this function is $L$-smooth and convex with $L = \|a\|^2$. However, one can show that $f$ is also $(L_0, L_1)$-smooth with $L_0 = 0$ and $L_1 = \|a\|$. For $\|a\| \gg 1$, both $L_0$ and $L_1$ are much smaller than $L$.
\end{example}

\begin{example}[Quadratic Function with Linear Term.]
    Let $f(x) = \frac{1}{2}x^\top A x + b^\top x$, where $A \in \R^{d \times d}$ is symmetric positive semi-definite, and $b \in \R^d$. Then $f$ is convex and $(L_0, 0)$-smooth with $L_0 = \|A\|$. This function is also $L$-smooth with the same $L$, but here $(L_1 = 0)$ shows the gradient is Lipschitz regardless of gradient size.
\end{example}

The condition of $(L_0,L_1)$-smoothness from Assumption \ref{as: smooth} can be formulated in terms of inequalities without $\sup$ operator, similar to the case of standard smoothness.  
\begin{lemma}($(L_0,L_1)$-smoothness properties \cite{gorbunov2024methods})  
\label{lem: L_0,L_1 smoothness}
    For $(L_0, L_1)$-smooth function $f$ (As. \ref{as: smooth}) and $ x,y \in \R^d$, it holds 
    \begin{eqnarray}
        \|\nabla f(x) - \nabla f(y)\|_2 \leq (L_0 + L_1\|\nabla f(y)\|_2)\exp(L_1\|x-y\|_2)\|x-y\|_2, \notag \\
        f(y) - f(x) - \la \nabla f(x), y -x \ra \leq \frac{L_0 + L_1\|\nabla f(x)\|_2 }{2}\exp(L_1\|x-y\|_2)\|x-y\|_2^2.
    \end{eqnarray}
\end{lemma} 

\subsection{Technical lemmas and propositions}

We use the following facts from the linear algebra and convex analysis \cite{boyd2004convex}:
\begin{proposition}[Norm Relation]
    For two norms $ \ell_p$ and $\ell_q$ with $\ 1 \leq p \leq q \leq 2$, the following relation holds true:
    \begin{eqnarray}
        \|x\|_q \leq \|x\|_p \leq d^{\frac{1}{p} - \frac{1}{q}}\|x\|_q, \quad \forall x \in \R^d. \label{eq: norm relation}
    \end{eqnarray}
\end{proposition}
\begin{proposition}[Jensen's Inequality]
    For scalar random variable $\xi$ with bounded $\kappa$-th moment $\kappa \in (1,2]$,  the following inequality holds true:
    \begin{equation}
        \EE [|\xi|] \leq \left(\EE[|\xi|^\kappa] \right)^\frac{1}{\kappa}. \label{eq: Jensen}
    \end{equation}
\end{proposition}
\begin{proposition}[Markov's Inequality]
    For scalar random variable $\xi$ with bounded first moment, the following inequality holds true for any $a > 0$:
    \begin{equation}
        \mathbb{P}(|\xi - \EE[\xi]]| \geq a) \leq \frac{\EE[|\xi|]}{a}. \label{eq: Markov}
    \end{equation}
\end{proposition}
To prove the HP bounds with the logarithmic dependency, we use the following measure concentration result (see, for example, \citep[Lemma $1$]{li2020high}.
\begin{lemma}[Measure Concentration Lemma]\label{lem: bernstein ineq}
    Let  $\{D_k\}_{k = 1}^T$ be a martingale difference sequence (MDS), i.e., $\EE[D_k| D_{k-1}, \dots, D_1] = 0$ for all $k \in \overline{1,T}$. Furthermore, for each $k \in \overline{1,T}$, there exists positive $\sigma_k \in \R$, s.t. $\EE\left[\exp\left(\frac{D_k^2}{\sigma_k^2}\right)| k\right] \leq e.$ Then the following probability bound holds true:
    \begin{equation} \label{eq: bernstein ineq}
        \forall \lambda > 0, \delta \in (0,1): \quad \mathbb{P}\left(\sum\limits_{k=1}^T D_k \leq \frac{3}{4}\lambda \sum \limits_{k=1}^T \sigma_k^2  + \frac1\lambda \log(\nicefrac{1}{\delta})\right) \geq 1  - \delta.
    \end{equation}
\end{lemma}
To control error reduction during batching, we use the following batching lemma for HT variables. Its modern proof for $d = 1$ was proposed in \citep[Lemma $4.2$]{cherapanamjeri2022optimal} and then generalized for the multidimensional case  in \cite{kornilov2024accelerated, hubler2024gradient}.
\begin{lemma}[HT Batching Lemma]\label{lem: batching p}
    Let $\kappa \in (1,2]$, and $X_1, \dots, X_B \in \R^d$ be a martingale  difference sequence (MDS), i.e., $\EE[X_i| X_{i-1}, \dots, X_1] = 0$ for all $i \in \overline{1,B}$.  If all variables $X_i$ have bounded $\kappa-$th moment, i.e., $\EE[\|X_i\|_2^\kappa] < +\infty,$ then the following bound holds true
    \begin{eqnarray}
        \EE\left[\left|\left|\frac{1}{B}\sum_{i=1}^B X_i \right|\right|_2^\kappa\right]\leq \frac{2}{B^\kappa} \sum_{i=1}^B \EE[ \|X_i\|^\kappa_2]. \label{eq: HT batching}
    \end{eqnarray}
\end{lemma}
We generalize the following lemma about changes after one update step of sign-based momentum methods from \citep[Lemma $1$]{sun2023momentum}.

\begin{lemma}[Sign Update Step Lemma] \label{lem: single update step} Let $x, m \in \R^d$ be arbitrary vectors, $A = \text{diag}(a_1, \dots, a_d)$ be diagonal matrix and $f$ be $(L_0,  L_1)$-smooth function (As.~\ref{as: smooth}). Then for the update step  
$$x' = x - \gamma \cdot  A \cdot  \sign(m)$$
with $\epsilon := m - \nabla f(x)$, the following inequality holds true
\begin{equation}
    f(x') - f(x) \leq - \gamma \|A\nabla f(x)\|_1 + 2\gamma\|A\|_F\|\epsilon\|_2 + \frac{L_0 + L_1\|A  \nabla f(x^k)\|_2 }{2} \exp{(\gamma L_1\|A\|_F)}\gamma^2\|A\|^2_F . \label{eq: single sign update}
\end{equation}
\end{lemma}

\begin{proof} 
Using $(L_0, L_1)$-smoothness of $f$ (Lemma \ref{lem: L_0,L_1 smoothness}) between points $x$ and  $x'$, we have
\[
f(x') \leq f(x) + \langle \nabla f(x), x' - x \rangle + \frac{L_0 + L_1\|x' - x\|}{2} \|x' - x\|^2 \exp(L_1\|x' - x\|).
\]
Substitute $x' - x = -\gamma A \sign(m)$ gives us:
\[
\langle \nabla f(x), x' - x \rangle = -\gamma \langle \nabla f(x), A \sign(m) \rangle.
\]
Next, we decompose the inner product:
\[
\langle \nabla f(x), A \sign(m) \rangle = \langle \nabla f(x), A \sign(\nabla f(x)) \rangle + \langle \nabla f(x), A(\sign(m) - \sign(\nabla f(x))) \rangle.
\]
We use the identity:
\[
\langle \nabla f(x), A \sign(\nabla f(x)) \rangle = \|A \nabla f(x)\|_1,
\]
and define $[\nabla f(x)]_i =: g_i$, then the second term becomes
\[
\sum_{i=1}^d a_i g_i \left( \sign(m_i) - \sign(g_i) \right).
\]
Now we analyze two cases for each $i$:
\begin{itemize}
    \item If $\sign(m_i) = \sign(g_i)$, then the term is equal to zero.
    \item Otherwise, $g_i \cdot m_i \leq 0$, hence $|g_i - m_i| \geq |g_i|$, and we have the following with $\epsilon_i := m_i - g_i$:
    \[
    a_i g_i \left( \sign(m_i) - \sign(g_i) \right) \leq 2a_i |g_i| \leq 2a_i |\epsilon_i|.
    \]
\end{itemize}
In total, we have:
\[
\langle \nabla f(x), A \sign(\nabla f(x)) - A \sign(m) \rangle \leq  2 \sum_{i=1}^d a_i |\epsilon_i| \leq  2\|A\|_F\|\epsilon\|_2,
\]
\[
\langle \nabla f(x), x' - x \rangle \leq -\gamma \|A \nabla f(x)\|_1 + 2\gamma \|A\|_F \|\epsilon\|_2.
\]
Finally, we observe that
\[
\|x' - x\| = \gamma \|A \sign(m)\|_2 \leq \gamma \|A\|_F,
\]
and derive the upper bound:
\[
f(x') - f(x) \leq -\gamma \|A \nabla f(x)\|_1 + 2\gamma \|A\|_F \|\epsilon\|_2 + \frac{L_0 + L_1\|A \nabla f(x)\|_2}{2} \exp(\gamma L_1 \|A\|_F) \gamma^2 \|A\|_F^2.
\]
\end{proof}

\subsection{Proof of \algname{SignSGD} General Convergence Lemma \ref{lem: signsgd T update}} \label{subsec: signsgd conv}
For beginning, we prove general lemma about \algname{SignSGD} convergence with HT gradient estimates $g^k$ with $\Vec{\sigma}, \kappa \in (1,2].$ This proof considerably relies on proof techniques for \algname{NSGD} from \cite{hubler2024gradient}.  
\begin{proof}
    Consider the $k$-th step of \algname{SignSGD}. We use $(L_0, L_1)$ smoothness of function $f$ (Lemma \ref{lem: L_0,L_1 smoothness}) to estimate:
    \begin{eqnarray}
        f(x^{k+1}) - f(x^k) &\leq& \la \nabla f (x^k), x^{k+1} - x^k \ra + \frac{L_0 + L_1\|\nabla f(x^k)\|_2 }{2}\exp(L_1\|x^{k+1} - x^k\|_2)\|x^{k+1} - x^k\|_2^2 \notag \\
        &=& - \gamma_k  \frac{\la \nabla f (x^k), \sign(g^k) \ra}{\|\nabla f (x^k)\|_1} \cdot \|\nabla f (x^k)\|_1 + \frac{L_0d \gamma_k^2}{2}\exp(L_1\sqrt{d}\gamma_k) \notag \\
        &+& \frac{ L_1d\gamma_k\exp(L_1\sqrt{d}\gamma_k) }{2}\cdot \gamma_k \|\nabla f(x^k)\|_2 \notag \\
         &\leq& - \gamma_k  \frac{\la \nabla f (x^k), \sign(g^k) \ra}{\|\nabla f (x^k)\|_1} \cdot \|\nabla f (x^k)\|_1 + \frac{L_0d \gamma_k^2}{2}\exp(L_1\sqrt{d}\gamma_k) \notag \\
         &+& \frac{ L_1d\gamma_k\exp(L_1\sqrt{d}\gamma_k) }{2}\cdot \gamma_k \|\nabla f(x^k)\|_1.\notag
    \end{eqnarray}

    Let us choose $\gamma_k \leq \frac{1}{4L_1d}$, then we have 
    $L_1d\gamma_k\exp(L_1\sqrt{d}\gamma_k) \leq \frac14$ and
    \begin{eqnarray}
        f(x^{k+1}) - f(x^k) &\leq&  - \gamma_k  \frac{\la \nabla f (x^k), \sign(g^k) \ra}{\|\nabla f (x^k)\|_1} \cdot \|\nabla f (x^k)\|_1 + L_0d \gamma_k^2 + \frac{\gamma_k}{4}  \|\nabla f(x^k)\|_1. \notag
    \end{eqnarray}

    Consequently, after summing all $T$ steps, we obtain:
    \begin{eqnarray}
        \sum \limits_{k=1}^{T} \gamma_k \left[ \frac{\la \nabla f (x^k), \sign(g^k) \ra}{\|\nabla f (x^k)\|_1} - \frac14\right] \cdot \|\nabla f (x^k)\|_1  \leq \underset{ = \Delta}{\underbrace{f(x^1) - f(x^*)}} + L_0d\sum \limits_{k=1}^T \gamma_k^2.
    \end{eqnarray}
    We introduce the following terms  $\phi_k := \frac{\la \nabla f (x^k), \sign(g^k) \ra}{\|\nabla f (x^k)\|_1} \in [-1,1]$, $\psi_k := \EE[\phi_k| x^{k }]$ and $D_k := - \gamma_k (\phi_k - \psi_k)\|\nabla f (x^k)\|_1$. We note that $D_k$ is a martingale difference sequence ($\EE[D_k|D_{k-1}, \dots, D_k] = 0$) and satisfies 
    $$\exp\left( \frac{D_k^2}{4\gamma_k^2\|\nabla f (x^k)\|_1^2}\right) = \exp \left(\frac{(\phi_k - \psi_k)^2}{4}\right) \leq e.$$
    Applying Measure Concentration Lemma \ref{lem: bernstein ineq} to MSD $D_k$ with $\sigma^2_k = 4 \gamma_k^2 \|\nabla f (x^k)\|_1^2$, we derive the bound for all $\lambda > 0$ with probability at least $1 - \delta$:
    $$\sum\limits_{k=1}^{T} \gamma_k(\psi_k - 3 \lambda \gamma_k \|\nabla f (x^k)\|_1 - \nicefrac{1}{4} ) \|\nabla f (x^k)\|_1 \leq \Delta + L_0d\sum\limits_{k=0}^{T-1} \gamma_k^2 + \frac{1}{\lambda} \log(\nicefrac{1}{\delta}).$$
     We use norm relation \eqref{eq: norm relation} and $(L_0,L_1)$-smoothness to estimate maximum gradient norm for all $k \in \overline{2,T+1}:$
    \begin{eqnarray}
        \|\nabla f (x^k)\|_1/\sqrt{d} &\leq& \|\nabla f (x^k)\|_2 \leq \|\nabla f (x^k) - \nabla f (x^{k-1}) + \nabla f (x^{k-1}) \|_2  \notag \\
    &\leq& \|\nabla f (x^k) - \nabla f (x^{k-1})\|_2 +  \|\nabla f (x^{k-1}) \|_2 \notag \\
    &\leq&  (L_0 + L_1\|\nabla f(x^{k-1})\|_2)\exp(L_1\|x^{k} - x^{k-1}\|_2)\|x^{k} - x^{k-1}\|_2 + \|\nabla f (x^{k-1})\|_2  \notag\\
    &\leq&  (L_0 + L_1\|\nabla f(x^{k-1})\|_2)\exp(L_1\sqrt{d}\gamma_k )\sqrt{d}\gamma_k + \|\nabla f (x^{k-1})\|_2.  \notag
    \end{eqnarray}
    At this point, we take $\gamma_k \leq \frac{1}{48L_1d\log\frac1\delta \sqrt{d}}$ to obtain
    \begin{eqnarray}
        \|\nabla f (x^k)\|_1/\sqrt{d} &\leq& 2L_0\sqrt{d}\gamma_k  + \frac{\|\nabla f (x^{k-1})\|_2}{48d \log \frac1\delta}+ \|\nabla f (x^{k-1})\|_2 \notag \\
    &\leq&   2L_0\sqrt{d}\sum_{\tau=1}^{k-1}\gamma_\tau + \sum_{\tau=1}^{k-1}\frac{\|\nabla f (x^{\tau})\|_2}{48d \log \frac1\delta} + \|\nabla f (x^1)\|_2 \notag \\
    &\leq&   2L_0\sqrt{d}\sum_{\tau=1}^{k-1}\gamma_\tau + \sum_{\tau=1}^{k-1}\frac{\|\nabla f (x^{\tau})\|_1}{48d \log \frac1\delta} + \|\nabla f (x^1)\|_1, \notag \\
    \gamma_k \|\nabla f (x^k)\|_1&\leq&   2L_0d\cdot \gamma_k\sum_{\tau=1}^{k-1}\gamma_\tau + \gamma_k\sum_{\tau=1}^{k-1}\frac{\|\nabla f (x^{\tau})\|_1}{48\sqrt{d} \log \frac1\delta} + \sqrt{d}\gamma_k\|\nabla f (x^1)\|_1. \notag
    \end{eqnarray}
    Since stepsizes $\gamma_k$ are non-increasing, we have
    $$\gamma_k\sum_{\tau=1}^{k-1}\frac{\|\nabla f (x^{\tau})\|_1}{48\sqrt{d} \log \frac1\delta} \leq \sum_{\tau=1}^{k-1}\frac{\gamma_\tau\|\nabla f (x^{\tau})\|_1}{48\sqrt{d} \log \frac1\delta},$$
     $$\gamma_k \|\nabla f (x^k)\|_1\leq   2L_0d\cdot \gamma_k\sum_{\tau=1}^{k-1}\gamma_\tau + \sum_{\tau=1}^{k-1}\frac{\gamma_\tau\|\nabla f (x^{\tau})\|_1}{48\sqrt{d} \log \frac1\delta} + \sqrt{d}\gamma_k\|\nabla f (x^1)\|_1.$$
    Hence, the choice $\lambda := \frac{1}{6d(\gamma^{max} \|\nabla f (x^1)\|_1  + \sum_{k=1}^{T}\frac{\gamma_k\|\nabla f (x^{k})\|_1}{48d\log \frac1\delta} + 2C_TL_0)}$ where $C_T := \max\limits_{k \in \overline{1,T}} \gamma_k \cdot  \sum\limits_{\tau=1}^{k-1}\gamma_\tau$ and $\gamma^{max} := \max\limits_{k \in \overline{1,T}} \gamma_k  $ yields with probability at least $1 - \delta$:
    \begin{eqnarray}
        \sum\limits_{k=1}^T \gamma_k\left(\psi_k - \frac{1}{2} - \frac14 \right)\|\nabla f (x^k)\|_1 &\leq& \Delta + L_0d\sum_{k=1}^T\gamma_k^2 + 6\sqrt{d}(\gamma^{max} \|\nabla f (x^1)\|_1  + 2C_TL_0) \log(\nicefrac{1}{\delta}) \notag \\
        &+& \frac{6}{48}\sum_{k=1}^{T}\gamma_k\|\nabla f (x^{k})\|_1, \notag \\ \sum\limits_{k=1}^T \gamma_k\left(\psi_k - \frac{1}{2} - \frac14  - \frac18\right)\|\nabla f (x^k)\|_1 &\leq& \Delta + L_0d\sum_{k=1}^T\gamma_k^2 + 6\sqrt{d}(\gamma^{max} \|\nabla f (x^1)\|_1  + 2C_TL_0) \log(\nicefrac{1}{\delta}), \notag \label{eq: sign sgd before prob}
    \end{eqnarray}
    Next, we estimate each term $\psi_k \|\nabla f (x^k)\|_1$ in the previous sum:
\begin{eqnarray}
\psi_k \|\nabla f (x^k)\|_1 &=& \EE \left[ \la \nabla f (x^k), \sign(g^k) \ra| x^k \right] \notag \\
&=& \|\nabla f (x^k)\|_1 - \sum_{i=1}^d 2 |\nabla f (x^k)|_i \cdot \mathbb{P}(\sign(\nabla f (x^k))_i\neq \sign(g^k)_i | x^k).\label{eq: line with prob sign}
\end{eqnarray}
For each coordinate, we have a bound derived from Markov's inequality \eqref{eq: Markov}  followed by Jensen’s inequality \eqref{eq: Jensen}:
\begin{eqnarray}
    \mathbb{P}(\sign(\nabla f (x^k))_i \neq \sign(g^k)_i | x^k) &\leq& \mathbb{P}(|\nabla f (x^k)_i -  g^k_i| \geq |\nabla f (x^k)_i| |  x^k)  
    \leq \frac{\EE_{\xi^k}[|\nabla f (x^k)_i - g^k_i| ]}{|\nabla f (x^k)_i|} \notag \\ &\leq& \frac{(\EE_{\xi^k}[|\nabla f (x^k)_i - g^k_i|^\kappa ])^\frac{1}{\kappa}}{|\nabla f (x^k)_i|}  \leq \frac{\sigma_{k,i} }{|\nabla f (x^k)_i|}. \label{eq:Prob sign not eq simple}
\end{eqnarray}
Hence, the whole sum can be bounded as 
\begin{eqnarray}
    \sum_{i=1}^d 2 |\nabla f (x^k)|_i \cdot \mathbb{P}(\sign(\nabla f (x^k))_i\neq \sign(g^k)_i | x^k)
    &\leq& 2 \|\Vec{\sigma}_k\|_1. \notag \notag 
\end{eqnarray}
Finally, we put this bound in \eqref{eq: sign sgd before prob} and obtain:
   \begin{eqnarray}
       \frac{1}{16} \sum\limits_{k=1}^T \gamma_k\|\nabla f (x^k)\|_1 &\leq& \Delta + L_0d\sum_{k=1}^T\gamma_k^2 + 2\sum_{k=1}^T\gamma_k \|\Vec{\sigma}_k\|_1 \notag \\
       &+& 6d(\gamma^{max} \|\nabla f (x^1)\|_1  + 2C_TL_0) \log(\nicefrac{1}{\delta}). \label{eq: proof signsgd conv main eq}
   \end{eqnarray}
\end{proof}

\subsection{Proof of \algname{minibatch-SignSGD} Complexity Theorem \ref{thm:minibatch SignSGD}} \label{subsec: minibatch signsgd proof}
The proof of Theorem \ref{thm:minibatch SignSGD} is divided into two parts: for finite horizon with optimal tuning (Theorem \ref{thm: minibatch signsgd finite}) and for infinite horizon with arbitrary tuning (Theorem \ref{thm: minibatch signsgd inifinite}).
\begin{theorem}[\textbf{HP complexity for \algname{minibatch-SignSGD}, finite horizon, full version}] \label{thm: minibatch signsgd finite}
Consider lower-bounded $(L_0,L_1)$-smooth function $f$ (As. \ref{as: bounded}, \ref{as: smooth}) and HT gradient estimates (As. \ref{as: pBCM}). Then Alg. \ref{alg:minibatch-signSGD} requires the sample complexity $N$  to achieve $\frac{1}{T} \sum_{k=1}^{T}  \|\nabla f(x^k)\|_1 \leq \varepsilon$ with probability at least $1-\delta$ for:

\textbf{Optimal tuning for $\varepsilon \geq \frac{8L_0}{L_1\sqrt{d}}$:}  $T = O\left(\frac{\Delta  L^\delta_1 d^\frac{3}{2} }{\varepsilon}\right), \gamma_k \equiv \frac{1}{48 L_1^\delta d^\frac32} , B_k \equiv  \left(\frac{16\|\Vec{\sigma}\|_1}{\varepsilon}\right)^\frac{\kappa}{\kappa-1}:$
\begin{equation}
    N = O\left(\frac{\Delta L_1^\delta   d^\frac{3}{2}}{\varepsilon}\left[1 +  \left(\frac{\|\Vec{\sigma}\|_1}{\varepsilon}\right)^\frac{\kappa}{\kappa-1}\right]\right), \notag 
\end{equation}
\textbf{Optimal tuning for $\varepsilon \leq \frac{8L_0}{L_1\sqrt{d}}$:} $T = O\left(\frac{\Delta  L_0^\delta d }{\varepsilon^2}\right), \gamma_k \equiv \sqrt{\frac{\Delta}{20 L_0^\delta dT}} , B_k \equiv  \left(\frac{16\|\Vec{\sigma}\|_1}{\varepsilon}\right)^\frac{\kappa}{\kappa-1}:$ 
\begin{equation}
   N = O\left(\frac{\Delta  L_0^\delta d }{\varepsilon^2}\left[1 +  \left(\frac{\|\Vec{\sigma}\|_1}{\varepsilon}\right)^\frac{\kappa}{\kappa-1}\right]\right),  \notag 
\end{equation}
\textbf{Arbitrary tuning for $\varepsilon \geq \frac{8L_0}{L_1\sqrt{d}}$:}  $T, \gamma_k \equiv \gamma_0 \leq \frac{1}{48L_1^\delta d^\frac32 },  B_k \equiv \max \{1, B_0T^2\}$:
\begin{equation}
    N = O\left(B_0\left(\frac{\Delta} {\varepsilon\gamma_0} \right)^3 + \frac{1}{B_0^2}\left(\frac{\|\Vec{\sigma}\|_1}{\varepsilon}\right)^\frac{3\kappa}{2(\kappa-1)}\right), \notag 
\end{equation}
\textbf{Arbitrary tuning for $\varepsilon \leq \frac{8L_0}{L_1\sqrt{d}}$:}  $T, \gamma_k \equiv \frac{\gamma_0}{\sqrt{T}}, B_k \equiv \max \{1, B_0T\}$:
\begin{equation}
    N = O\left(\frac{B_0(\nicefrac{\Delta}{\gamma_0} + L_0^\delta d \gamma_0)^4} {\varepsilon^4} + \frac{1}{B_0}\left(\frac{\|\Vec{\sigma}\|_1}{\varepsilon}\right)^\frac{2\kappa}{\kappa-1}\right), \notag 
\end{equation} 
where $\Delta = f(x^1) - f^*, L_0^\delta = L_0 \log(\nicefrac{1}{\delta}), L_1^\delta = L_1 \log(\nicefrac{1}{\delta}).$
\end{theorem}

\begin{proof}
    Plugging in constant stepsizes $\gamma_k \equiv \gamma \leq \frac{1}{48L_1d\log\frac1\delta \sqrt{d}}$ in \eqref{eq: proof signsgd conv main eq} implies $C_T = T\gamma^2, \gamma^{max} = \gamma$:
$$\frac1T \sum\limits_{k=1}^{T} \|\nabla f (x^k)\|_1 \leq \frac{4\Delta}{T\gamma} + 80 L_0 d\gamma \log(\nicefrac{1}{\delta})  + 8 \|\Vec{\sigma}_k\|_1 + 24\frac{d\|\nabla f (x^1)\|_1}{T}  \log(\nicefrac{1}{\delta}).$$
Due to Batching Lemma \ref{lem: batching p}, we can estimate the $\kappa-$th moment of the batched estimate for constant batchsizes $B_k  \equiv B$ as
   $\|\Vec{\sigma}_k\|_1 \leq \frac{2\|\Vec{\sigma}\|_1}{B^{\frac{\kappa-1}{\kappa}}}$  and  derive:
\begin{equation}\label{eq: minibatch signsgd bound before params}
    \frac1T \sum\limits_{k=1}^{T} \|\nabla f (x^k)\|_1 \leq \frac{4\Delta}{T\gamma} + 80 L_0d\gamma \log(\nicefrac{1}{\delta})  + 8 \frac{\|\Vec{\sigma}\|_1}{B^{\frac{\kappa-1}{\kappa}}} + 24\frac{d\|\nabla f (x^1)\|_1}{T}  \log(\nicefrac{1}{\delta}). \notag 
\end{equation}
We can omit the last term since its dependency on $T$ has the largest power.

\textbf{Case $\varepsilon \geq  \frac{8L_0}{L_1\sqrt{d}}$, arbitrary tuning:} We use parameters $T, \gamma_k = \gamma_0, B_k = \max \{1, B_0T^2\}$ to get:
$$\frac1T \sum\limits_{k=1}^{T} \|\nabla f (x^k)\|_1 \leq \frac{4\Delta }{T\gamma_0}  + \varepsilon +  8 \frac{\|\Vec{\sigma}\|_1}{B_0^{\frac{\kappa-1}{\kappa}} T^{\frac{2(\kappa-1)}{\kappa}}} + 24\frac{d\|\nabla f (x^1)\|_1}{T}  \log(\nicefrac{1}{\delta}).$$
 Setting such $T$ that the first two terms become less than $\varepsilon$, we obtain the final complexity $N = T \cdot B_0T^2.$

\textbf{Case $\varepsilon \geq  \frac{8L_0}{L_1\sqrt{d}}$, optimal tuning:} We use stepsize $\gamma = \frac{1}{48L_1d\log\frac1\delta \sqrt{d}} \Rightarrow 80 L_0 d\gamma \log(\nicefrac{1}{\delta}) \leq \varepsilon/2$ and  batchsize $8\frac{\|\Vec{\sigma}\|_1}{B^{\frac{\kappa-1}{\kappa}}} \leq \varepsilon/2 \Rightarrow B_k \equiv \max \left\{1,  \left(\frac{16\|\Vec{\sigma}\|_1}{\varepsilon}\right)^\frac{\kappa}{\kappa-1}\right\}$. The number of iterations  $T$ is chosen to bound the first term: 
$$\frac{4\Delta}{T\gamma}  = \frac{192\Delta L_1\log \frac1\delta d^\frac32}{T} \leq \frac{\varepsilon}{2} \Rightarrow T = O\left(\frac{\Delta L_1  \log \frac1\delta d^\frac{3}{2}}{\varepsilon}\right).$$
The total number of oracle calls is:
\begin{eqnarray}
    \varepsilon &\geq&   \frac{8L_0}{L_1\sqrt{d}} \quad \Rightarrow \quad N = O\left(\frac{\Delta L_1  \log(\nicefrac{1}{\delta} ) d^\frac{3}{2}}{\varepsilon}\left[1 +  \left(\frac{\|\Vec{\sigma}\|_1}{\varepsilon}\right)^\frac{\kappa}{\kappa-1}\right]\right). \notag 
\end{eqnarray}

\textbf{Case $\varepsilon <  \frac{8L_0}{L_1\sqrt{d}}$, arbitrary tuning:} We use parameters $T, \gamma_k = \frac{\gamma_0}{\sqrt{T}}, B_k = \max \{1, B_0T\}$ to get:
$$\frac1T \sum\limits_{k=1}^{T} \|\nabla f (x^k)\|_1 \leq \frac{4\Delta}{\sqrt{T}\gamma_0} + 80 \frac{L_0d\gamma_0}{\sqrt{T}} \log(\nicefrac{1}{\delta})  + 8 \frac{\|\Vec{\sigma}\|_1}{B_0^{\frac{\kappa-1}{\kappa}} T^{\frac{\kappa-1}{\kappa}}} + 24\frac{d\|\nabla f (x^1)\|_1}{T}  \log(\nicefrac{1}{\delta}).$$
 Setting such $T$ that the first two terms become less than $\varepsilon$, we obtain the final complexity $N = T \cdot B_0T.$

\textbf{Case $\varepsilon < \frac{8L_0}{L_1\sqrt{d}}$, optimal tuning:} We set the same batchsize $8\frac{\|\Vec{\sigma}\|_1}{B^{\frac{\kappa-1}{\kappa}}} \leq \varepsilon/2 \Rightarrow B_k \equiv \max \left\{1,  \left(\frac{16\|\Vec{\sigma}\|_1}{\varepsilon}\right)^\frac{\kappa}{\kappa-1}\right\}$. The stepsize $\gamma$ is set to minimize the sum:
$$\min_\gamma \left[ \frac{4\Delta}{T\gamma} + 80 L_0 d\gamma \log(\nicefrac{1}{\delta}) \right] = 2\sqrt{\frac{320\Delta L_0 d \log(\nicefrac{1}{\delta})}{T}} ,$$
it means that the stepsize $\gamma = \sqrt{\frac{4\Delta}{80 TL_0\log(\nicefrac{1}{\delta})d}}$. The number of iterations $T$ is chosen to satisfy $$2\sqrt{\frac{320\Delta L_0  \log(\nicefrac{1}{\delta})d}{T}} \leq \frac{\varepsilon}{2} \Rightarrow T = O \left(\frac{\Delta L_0  \log(\nicefrac{1}{\delta})d}{\varepsilon^2}\right).$$
We only need to check whether condition $\gamma \leq \frac{1}{48L_1d\log\frac1\delta \sqrt{d}}$ holds:
\begin{eqnarray}
    \gamma &=& \sqrt{\frac{4\Delta}{80 TL_0\log(\nicefrac{1}{\delta})d}} = \sqrt{\frac{4\Delta}{ T} \frac{1}{80L_0\log(\nicefrac{1}{\delta})d}} \notag \\
    &\leq& \frac{\varepsilon}{4} \frac{1}{80L_0\log(\nicefrac{1}{\delta})d} \leq \frac{8L_0}{4L_1 \sqrt{d}} \frac{1}{80L_0\log(\nicefrac{1}{\delta})d} \notag \\
    &\leq& \frac{1}{48L_1d\log\frac1\delta \sqrt{d}}. \notag
\end{eqnarray}
Hence, we have the following bound for sample complexity
\begin{eqnarray}
\varepsilon &<& \frac{8L_0}{L_1\sqrt{d}} \quad \Rightarrow \quad N = O\left(\frac{\Delta  L_0\log(\nicefrac{1}{\delta} ) d }{\varepsilon^2}\left[1 +  \left(\frac{\|\Vec{\sigma}\|_1}{\varepsilon}\right)^\frac{\kappa}{\kappa-1}\right]\right). 
\end{eqnarray}
\end{proof}

\begin{theorem}[\textbf{HP complexity for \algname{minibatch-SignSGD}, infinite horizon}] \label{thm: minibatch signsgd inifinite}
Consider lower-bounded $(L_0,L_1)$-smooth function $f$ (As. \ref{as: bounded}, \ref{as: smooth}) and HT gradient estimates (As. \ref{as: pBCM}). Then Alg. \ref{alg:minibatch-signSGD} requires the sample complexity $N$  to achieve $\min \limits_{k \in \overline{1,T}} \|\nabla f(x^k)\|_1 \leq \varepsilon$ with probability at least $1-\delta$ for:

\textbf{Arbitrary tuning:}  Until plateau $\gamma_k = \gamma_0 \leq \frac{1}{48L^\delta_1d^\frac32}, B_k = B_0k^2$, after $\gamma_k = \frac{\gamma_0}{\sqrt{k}}, B_k = B_0k$:
\begin{eqnarray}
    \varepsilon \geq \frac{8L_0}{L_1\sqrt{d}} &\Rightarrow& N = \tilde{O}\left( B_0\left(\frac{\Delta}{\gamma_0\varepsilon} 
 \right)^3 + \frac{1}{B_0^2}\left(\frac{\|\Vec{\sigma}\|_1\ln\frac1\varepsilon}{\varepsilon}\right)^\frac{3\kappa}{2(\kappa - 1)}\right), \notag \\
    \varepsilon \ll \frac{8L_0}{L_1\sqrt{d}} &\Rightarrow& N =  \tilde{O}\left(\frac{B_0(L_0^\delta \gamma_0 d + \Delta/\gamma_0)^4}{\varepsilon^4}  +  \frac{1}{B_0}\left(\frac{\|\Vec{\sigma}\|_1}{\varepsilon}\right)^\frac{2\kappa}{\kappa - 1} \right), \notag
\end{eqnarray} 
\textbf{Optimal tuning:} First $\frac{64\Delta L_1^\delta L_1 d^2}{L_0}$ steps,  $\gamma_k = \frac{1}{48L_1^\delta d^\frac32}, B_k = \left(16  k \right)^\frac{\kappa}{\kappa - 1}$, after $\gamma_k = \sqrt{\frac{\Delta}{20 d L_0^\delta k}}, B_k = \left(16  k \right)^\frac{\kappa}{2(\kappa - 1)}:$
\begin{eqnarray}
    \varepsilon \geq \frac{8L_0}{L_1\sqrt{d}} &\Rightarrow& N  = \tilde{O}\left(\left(\frac{\Delta L_1^\delta d^\frac32 + \|\Vec{\sigma}\|_1}{\varepsilon}\right)^\frac{2\kappa -1}{\kappa - 1}\right), \notag \\
\varepsilon \ll \frac{8L_0}{L_1\sqrt{d}} &\Rightarrow& N  = \tilde{O}\left(\left(\frac{\Delta L_0^\delta d + \|\Vec{\sigma}\|^2_1}{\varepsilon^2}\right)^\frac{3\kappa -2}{2(\kappa - 1)}\right). \notag
\end{eqnarray}
where $\Delta = f(x^1) - f^*, L_0^\delta = L_0 \log(\nicefrac{1}{\delta}), L_1^\delta = L_1 \log(\nicefrac{1}{\delta}).$
\end{theorem}
\begin{proof}
First, we derive upper bound for new $\min$ metric with non-constant  parameters:
    \begin{eqnarray}
    \min \limits_{k \in \overline{1,T}} \|\nabla f(x^k)\|_1 &\leq& \frac{\sum\limits_{k=1}^T \gamma_k\|\nabla f (x^k)\|_1}{\sum\limits_{k=1}^T \gamma_k} =  \frac{\Delta}{\sum\limits_{k=1}^T \gamma_k} + L_0d\frac{\sum\limits_{k=1}^T\gamma_k^2}{\sum\limits_{k=1}^T \gamma_k} + \frac{2\sum\limits_{k=1}^T\gamma_k \|\Vec{\sigma}\|_1/B_k^\frac{\kappa - 1}{\kappa} } {\sum\limits_{k=1}^T \gamma_k}\notag \\
       &+& 6d(\gamma^{max} \|\nabla f (x^1)\|_1  + 2C_TL_0) \frac{\log(\nicefrac{1}{\delta})}{\sum\limits_{k=1}^T \gamma_k}. \notag
\end{eqnarray}
\textbf{Case $\varepsilon \geq \frac{8L_0}{L_1 \sqrt{d}}$, optimal tuning.} If we consider only first $T \leq \frac{64\Delta L_1^\delta L_1 d^2}{L_0}$ steps with constant stepsizes $\gamma_k = \frac{1}{48L_1^\delta d^\frac32}$ and increasing batchsizes $B_k = \left(16  k \right)^\frac{\kappa}{\kappa - 1}$, we get
\begin{eqnarray}
    \sum\limits_{k=1}^T \gamma_k &=& \frac{T}{48L_1^\delta d^\frac32}, \sum\limits_{k=1}^T \gamma^2_k = \frac{T}{(48L_1^\delta d^\frac32)^2}, \gamma^{max} = \frac{1}{48L_1^\delta d^\frac32}, C_T = \frac{T}{(48L_1^\delta d^\frac32)^2},\notag \\
    \sum\limits_{k=1}^T \frac{1}{B_k^\frac{\kappa - 1}{\kappa}} &=&  \sum\limits_{k=1}^T \frac{1}{16k } \leq  \frac{\ln T}{16}, \notag \\
    \min \limits_{k \in \overline{1,T}} \|\nabla f(x^k)\|_1 &\leq&  \frac{48\Delta L_1^\delta d^\frac32}{T} + \frac{24 L_0^\delta d}{48L_1^\delta d^\frac32} + \frac{2 \|\Vec{\sigma}\|_1} {T}  \frac{\ln T}{16} \leq \varepsilon. \notag 
\end{eqnarray}
The term $\frac{24 L_0^\delta d}{48L_1^\delta d^\frac32} \leq \frac{\varepsilon}{16}$ is bounded by condition, and the number of iterations $T = \tilde{O}\left(\frac{(\Delta  L_1^\delta d^\frac32 + \|\Vec{\sigma}\|_1)}{\varepsilon}\right)$ is enough to bound the other terms. The total sample complexity is 
\begin{eqnarray}
    \sum\limits_{k=1}^T B_k = \sum\limits_{k=1}^T \left(16  k \right)^\frac{\kappa}{\kappa - 1} \leq (16T)^\frac{2\kappa -1}{\kappa - 1} = \tilde{O}\left(\left(\frac{(\Delta  L_1^\delta d^\frac32 + \|\Vec{\sigma}\|_1)}{\varepsilon}\right)^\frac{2\kappa -1}{\kappa - 1}\right).
\end{eqnarray}

\textbf{Case $\varepsilon \ll \frac{8L_0}{L_1 \sqrt{d}}$, optimal tuning.} In this case, the first $\frac{64L_1 L_1^\delta d^2}{L_0}$ steps can be neglected, and we use decreasing stepsizes $\gamma_k = \sqrt{\frac{\Delta}{20 d L_0^\delta k}}$ and increasing batchsizes $B_k = \left(16  k \right)^\frac{\kappa}{2(\kappa - 1)}$ to get
\begin{eqnarray}
    \sum\limits_{k=1}^T \gamma_k &=& 2\sqrt{\frac{\Delta T}{20L_0^\delta d}}, \sum\limits_{k=1}^T \gamma^2_k = \frac{\Delta \ln T}{20L_0^\delta d}, \gamma^{max} = \frac{1}{48L_1^\delta d^\frac32}, C_T = \frac{\Delta}{20 L_0^\delta d},\notag \\
    \sum\limits_{k=1}^T \frac{\gamma_k}{B_k^\frac{\kappa - 1}{\kappa}} &=&  \sqrt{\frac{\Delta}{20 L_0^\delta d}} \sum\limits_{k=1}^T \frac{1}{\sqrt{k}}\frac{1}{4\sqrt{k} } \leq  \sqrt{\frac{\Delta}{256 L_0^\delta d}} \ln T, \notag \\
    \min \limits_{k \in \overline{1,T}} \|\nabla f(x^k)\|_1 &\leq&  \sqrt{\frac{80L_0^\delta d}{T}} + \sqrt{\frac{\Delta L_0^\delta d}{20  T}} \ln T + 2 \frac{\|\Vec{\sigma}\|_1 \ln T}{\sqrt{T}} + \sqrt{\frac{\Delta L_0^\delta d}{20  T}} \leq \varepsilon. \notag 
\end{eqnarray}
Hence, the number of iterations $T = \tilde{O}\left(\frac{(\Delta L_0^\delta d + \|\Vec{\sigma}\|^2_1)}{\varepsilon^2}\right)$ is enough to bound the sum. The total sample complexity is 
\begin{eqnarray}
    \sum\limits_{k=1}^T B_k = \sum\limits_{k=1}^T \left(16  k \right)^\frac{\kappa}{2(\kappa - 1)} \leq (16T)^\frac{3\kappa -2}{2(\kappa - 1)} = \tilde{O}\left(\left(\frac{(\Delta L_0^\delta d + \|\Vec{\sigma}\|^2_1)}{\varepsilon^2}\right)^\frac{3\kappa -2}{2(\kappa - 1)}\right).
\end{eqnarray}

\textbf{Case $\varepsilon \geq \frac{8L_0}{L_1 \sqrt{d}}$, arbitrary tuning.} If we consider only first $T$ steps until plateau $\frac{8L_0}{L_1 \sqrt{d}}$, we use constant stepsizes $\gamma_k = \gamma_0 \leq  \frac{1}{48 \delta_1 d^\frac32 }$ and increasing batchsizes $B_k =B_0k^2$ to get
\begin{eqnarray}
    \sum\limits_{k=1}^T \gamma_k &=& T\gamma_0, \sum\limits_{k=1}^T \gamma^2_k = T\gamma_0^2, \gamma^{max} = \gamma_0, C_T = T\gamma_0,\notag \\
    \sum\limits_{k=1}^T \frac{1}{B_k^\frac{\kappa - 1}{\kappa}} &=&  \sum\limits_{k=1}^T \frac{1}{(\sqrt{B_0}k)^\frac{2(\kappa - 1)}{\kappa} } \leq  \frac{T^\frac{2 -\kappa}{\kappa}\ln T}{B_0^\frac{\kappa - 1}{\kappa}}, \notag \\
    \min \limits_{k \in \overline{1,T}} \|\nabla f(x^k)\|_1 &\leq&  \frac{\Delta }{\gamma_0 T} + \frac{24 L_0^\delta d}{\gamma_0} + \frac{2 \|\Vec{\sigma}\|_1} {(T\sqrt{B_0})^\frac{2(\kappa - 1)}{\kappa}} \ln T \leq \varepsilon. \notag 
\end{eqnarray}
The term $\frac{24 L_0^\delta d}{\gamma_0} \leq \frac{\varepsilon}{16}$ is bounded by condition, and the number of iterations $T = \tilde{O}\left( \left(\frac{\Delta }{\gamma_0\varepsilon} 
 \right) + \frac{1}{B_0}\left(\frac{\|\Vec{\sigma}\|_1}{\varepsilon}\right)^\frac{\kappa}{2(\kappa - 1)}\right)$ is enough to bound the other terms. The total sample complexity is 
\begin{eqnarray}
    \sum\limits_{k=1}^T B_k = \sum\limits_{k=1}^T B_0 k^2 \leq B_0T^3 = \tilde{O}\left( B_0\left(\frac{\Delta}{\gamma_0\varepsilon} 
 \right)^3 + \frac{1}{B_0^2}\left(\frac{\|\Vec{\sigma}\|_1}{\varepsilon}\right)^\frac{3\kappa}{2(\kappa - 1)}\right). \notag
\end{eqnarray}

\textbf{Case $\varepsilon \ll \frac{8L_0}{L_1 \sqrt{d}}$, arbitrary tuning.} In this case, the first $\frac{64\Delta L_1 L_1^\delta d^2}{L_0}$ steps can be neglected, and we use decreasing stepsizes $\gamma_k = \frac{\gamma_0}{\sqrt{k}}$ and increasing batchsizes $B_k = B_0k$ to get
\begin{eqnarray}
    \sum\limits_{k=1}^T \gamma_k &=& \gamma_0\sqrt{T}, \sum\limits_{k=1}^T \gamma^2_k = \gamma_0^2\ln T, \gamma^{max} = \gamma_0, C_T = \gamma_0^2,\notag \\
    \sum\limits_{k=1}^T \frac{\gamma_k}{B_k^\frac{\kappa - 1}{\kappa}} &=&  \frac{\gamma_0}{B_0^\frac{\kappa - 1}{\kappa}} \sum\limits_{k=1}^T \frac{1}{k^\frac{3\kappa  -2}{2\kappa}} \leq  \frac{\gamma_0}{B_0^\frac{\kappa - 1}{\kappa}} T^\frac{2 - \kappa}{2\kappa} \ln T, \notag \\
    \min \limits_{k \in \overline{1,T}} \|\nabla f(x^k)\|_1 &\leq&  \frac{\Delta}{\gamma_0 \sqrt{T}} + L^\delta_0 d \gamma_0 \frac{\ln T}{\sqrt{T}} + \frac{\|\Vec{\sigma}\|_1 \ln T}{B_0^\frac{\kappa - 1}{\kappa} T^\frac{\kappa - 1}{\kappa} } \leq \varepsilon. \notag 
\end{eqnarray}
Hence, the number of iterations $T = \tilde{O}\left(\frac{(L_0^\delta \gamma_0 d + \Delta/\gamma_0)^2}{\varepsilon^2}  +  \frac{1}{B_0}\left(\frac{\|\Vec{\sigma}\|_1}{\varepsilon}\right)^\frac{\kappa}{\kappa - 1} \right)$ is enough to bound the sum. The total sample complexity is 
\begin{eqnarray}
    \sum\limits_{k=1}^T B_k = \sum\limits_{k=1}^T B_0k \leq B_0T^2 = \tilde{O}\left(\frac{B_0(L_0^\delta \gamma_0 d + \Delta/\gamma_0)^4}{\varepsilon^4}  +  \frac{1}{B_0}\left(\frac{\|\Vec{\sigma}\|_1}{\varepsilon}\right)^\frac{2\kappa}{\kappa - 1} \right).
\end{eqnarray}

\end{proof}

\subsection{Proof of \algname{MajorityVote-SignSGD} Complexity Theorem \ref{thm:com-sign conv}}\label{subsec: majority vote sign proofs}
We start this section with a general lemma on convergence of \algname{MajorityVote-SignSGD}. The proof of Theorem \ref{thm:com-sign conv} is located after the lemma and divided into two parts: for finite horizon with optimal tuning (Theorem \ref{thm: signsgd majority finite}) and for infinite horizon with arbitrary tuning (Theorem \ref{thm: majority signsgd inifinite}).
\begin{lemma}[\textbf{\algname{MajorityVote-SignSGD} Convergence Lemma}] \label{lem: majority signsgd T update}
Consider lower-bounded $(L_0, L_1)$-smooth function $f$ (As. \ref{as: bounded}, \ref{as: smooth}) and HT \textbf{unimodal and symmetric} gradient estimates $\kappa > 0$ (As. \ref{as: pBCM}). Then Alg. \ref{alg:majorityvotesignSGDsingle} after $T$ iterations with non-increasing stepsizes $\gamma_k \leq 1/ (48L_1d^\frac32\log\frac1\delta)$ and batchsizes $M_k \geq 160/\kappa^2$ achieves with probability at least $1 - \delta$:
\begin{equation}
 \sum\limits_{k=1}^T \frac{\gamma_k}{16}\|\nabla f (x^k)\|_1 \leq \Delta + L_0d\sum_{k=1}^T\gamma_k^2 + 2\sum_{k=1}^T\gamma_k \frac{\|\Vec{\sigma}\|_1}{\sqrt{M_k}} 
       + 6d(\gamma_1 \|\nabla f (x^1)\|_1  + 2C_TL_0) \log\frac{1}{\delta}, \label{eq: majority signsgd convergence lemma}
\end{equation}
where $C_T := \max\limits_{k \in \overline{1,T}} \gamma_k \cdot  \sum\limits_{\tau=1}^{k-1}\gamma_\tau$ and $\Delta = f(x^1) - f^*$.
\end{lemma}
\begin{proof}
The beginning of this proof exactly copies the proof of \algname{SignSGD} Convergence Lemma (Appendix \ref{subsec: signsgd conv}) until equality \eqref{eq: line with prob sign}. We have to estimate the probability of failure of majority voting for each coordinate $j$ conditioned on $x^k$, namely,
   \begin{eqnarray}
     \mathbb{P}\left(\sign(\nabla f (x^k))_j\neq \sign\left[\sum_{i=1}^{M_k}\sign(g^k_i) \right]_j\right), \quad g^k_i = \nabla f(x^k, \xi^k_i). \notag 
\end{eqnarray}
We use the generalized Gauss's Inequality about distribution of unimodal symmetric random variables \citep[Theorem 1]{dharmadhikari1986gauss}. 
\begin{lemma}[Gauss's Inequality]\label{lem: gauss ineq}
Let a random variable $\xi$ be unimodal symmetric with mode $\nu$ and bounded $\kappa$-th moment, $\kappa > 0$. Then the following bounds hold:
$$\mathbb{P}\left[|\xi - \nu| \geq \tau\right] \leq \begin{cases} \left(\frac{\kappa}{\kappa+1}\right)^\kappa \frac{\EE[|\xi - \nu|]^\kappa}{\tau^\kappa}, & \quad \tau^\kappa \geq \frac{\kappa^\kappa}{(\kappa+1)^{\kappa-1}} \cdot \EE[|\xi - \nu|^\kappa], \\
1 - \left[\frac{\tau^\kappa}{(\kappa+1) \EE[|\xi - \nu|]^\kappa}\right]^\frac1\kappa, &\quad \tau^\kappa \leq \frac{\kappa^\kappa}{(\kappa+1)^{\kappa-1}} \cdot \EE[|\xi - \nu|^\kappa].
\end{cases}$$    
\end{lemma}
We use Gauss's Inequality for each variable $g^k_{i,j} = \nabla f(x^k, \xi^k_i)_j$ satisfying the symmetry requirement from the theorem's statement. We denote  $S_j := \frac{|\nabla f (x^k)_j|}{\sigma_j}$ and bound
\begin{eqnarray}
    \mathbb{P}\left[\sign(\nabla f (x^k)_j) \neq \sign(g^k_{i,j})\right] &=& \mathbb{P}\left[g^k_{i,j} - \nabla f (x^k)_j \geq |\nabla f (x^k)_j|\right] \notag \\
    &=& \frac12 \mathbb{P}\left[|g^k_{i,j} - \nabla f (x^k)_j| \geq |\nabla f (x^k)_j|\right] \notag \\
    &\leq&  \begin{cases} \frac12\left(\frac{\kappa}{\kappa+1}\right)^\kappa \frac{\sigma_j^\kappa}{|\nabla f (x^k)_j|^\kappa}, & \quad |\nabla f (x^k)_j|^\kappa \geq \frac{\kappa^\kappa}{(\kappa+1)^{\kappa-1}} \cdot \sigma_j^\kappa, \\
\frac12 - \frac12\left[\frac{|\nabla f (x^k)_j|^\kappa}{(\kappa+1) \sigma_j^\kappa}\right]^\frac1\kappa, &\quad |\nabla f (x^k)_j|^\kappa \leq \frac{\kappa^\kappa}{(\kappa+1)^{\kappa-1}} \cdot \sigma_j^\kappa,
\end{cases} \notag \\
    &\leq&  \begin{cases} \frac12\left(\frac{\kappa}{\kappa+1}\right)^\kappa \frac{1}{S_j^\kappa}, & \quad S_j^\kappa \geq \frac{\kappa^\kappa}{(\kappa+1)^{\kappa-1}} , \\
\frac12 - \frac12\frac{S_j}{(\kappa+1)^\frac{1}{\kappa} }, &\quad S_j^\kappa \leq \frac{\kappa^\kappa}{(\kappa+1)^{\kappa-1}},
\end{cases}\notag
\end{eqnarray}

We denote probability of failure of a single estimate by 
\begin{eqnarray}
q_j &:=& \mathbb{P}\left[\sign(\nabla f (x^k)_j) \neq \sign(g^k_{i,j})\right] \notag \\
&\leq&  \begin{cases} \frac12\left(\frac{\kappa}{\kappa+1}\right)^\kappa \frac{1}{S_j^\kappa}, & \quad S_j^\kappa \geq \frac{\kappa^\kappa}{(\kappa+1)^{\kappa-1}} , \\
\frac12 - \frac12\frac{S_j}{(\kappa+1)^\frac{1}{\kappa} }, &\quad S_j^\kappa \leq \frac{\kappa^\kappa}{(\kappa+1)^{\kappa-1}},\notag 
\end{cases} \\
&=:& \tilde{q}_j(S_j). \label{eq: gauss inequality}
\end{eqnarray}
Moreover, this probability  $q_j \leq \tilde{q}_j(S_j) < \frac12$, and the deviation of $q_j$ from $\frac12$ can be bounded by 
$$\varepsilon_j := \frac{1}{2} - q_j \leq \frac{1}{2} - \tilde{q}_j(S_j) =: \tilde{\varepsilon}_j(S_j).$$
The probability of getting the wrong sign can be restated as the probability of failing half out of $M_k$  Bernoulli trials  with fail probability $q_j$:
\begin{eqnarray}
\mathbb{P}\left[\sign(\nabla f (x^k)_j) \neq \sign\left[\sum\limits_{i=1}^{M_k}\sign(g^k_{i,j})\right]\right]  \leq \frac{1}{1 + \frac{M_k}{\frac{1}{4\varepsilon_j^2} - 1}} < \frac{1}{1 + \frac{M_k}{\frac{1}{4\tilde{\varepsilon}_j^2(S_j)} - 1}}. \label{eq: prob not equal after bernoulli}
\end{eqnarray}
\begin{itemize}
    \item First, we consider the case $S_j \geq \frac{\kappa}{(\kappa+1)^{\frac{\kappa-1}{\kappa}}}$:
\begin{eqnarray}
\tilde{\varepsilon}_j^2(S_j) = \left(\frac12  - \frac12\left(\frac{\kappa}{\kappa+1}\right)^\kappa \frac{1}{S_j^\kappa} \right)^2 \geq \frac{1}{4} \frac{\kappa^2}{(\kappa + 1)^2} \notag,
\end{eqnarray}
\begin{eqnarray}
    \frac{1}{4\tilde{\varepsilon}_j^2(S_j)} - 1 &\leq& \frac{(\kappa + 1)^2}{\kappa^2} - 1 \leq \frac{5}{\kappa^2}. \notag 
\end{eqnarray}
 If we set $M_k \geq \frac{160}{\kappa^2}$, then the fail probability is upper bounded by
\begin{equation}
\mathbb{P}\left[\sign(\nabla f (x^k)_j) \neq \sign\left[\sum\limits_{i=1}^{M_k}\sign(g^k_{i,j})\right]\right]  < \frac{1}{1 + \frac{M_k}{\frac{1}{4\tilde{\varepsilon}_j^2(S_j)} - 1}} \leq \frac{1}{32}. \label{eq: S geq bound}
\end{equation}
\item For the case $S_j < \frac{\kappa}{(\kappa+1)^{\frac{\kappa-1}{\kappa}}}$, we derive the bound:
\begin{eqnarray}
    \frac{1}{4\tilde{\varepsilon}_j^2(S_j)} - 1 &=& \frac{(\kappa+1)^\frac2\kappa}{S_j^2} - 1 \leq \frac{4}{S_j^2}.
\end{eqnarray}
And we use the inequality $\frac{1}{1 + x^2} \leq \frac{1}{2x}, x > 0$ on \eqref{eq: prob not equal after bernoulli}:
\begin{eqnarray}
\eqref{eq: prob not equal after bernoulli} &\leq& \frac{\sqrt{\frac{1}{4\tilde{\varepsilon}_j^2(S_j)} - 1}}{2\sqrt{M_k}} \leq \frac{1}{\sqrt{M_k}} \cdot \frac{1}{S_j}.\label{eq: S leq bound}  
\end{eqnarray}
\end{itemize}
Combining \eqref{eq: S geq bound} and \eqref{eq: S leq bound} together, we obtain the bound for each coordinate:
 \begin{eqnarray}
 \mathbb{P}\left[\sign(\nabla f (x^k)_j) \neq \sign\left[\sum\limits_{i=1}^{M_k}\sign(g^k_{i,j})\right]\right]  \leq \frac{1}{32} + \frac{1}{\sqrt{M_k}} \cdot \frac{1}{S_j} = \frac{1}{32} +  \frac{1}{\sqrt{M_k}} \frac{\sigma_j}{|\nabla f (x^k)_j|}. \label{eq: majority proof P_M bound}
 \end{eqnarray}
The rest of this proof is copying the proof of \algname{SignSGD} Convergence Lemma (Appendix \ref{subsec: signsgd conv}) until the equality \eqref{eq: line with prob sign}. There we replace probability of single estimate with the majority voting and obtain:
$$ \sum_{j=1}^d  |\nabla f (x^k)|_j \cdot \mathbb{P}\left[\sign(\nabla f (x^k)_j) \neq \sign\left[\sum\limits_{i=1}^{M_k}\sign(g^k_{i,j})\right]\right]
   \leq  \frac{\|\nabla f(x^k)\|_1}{32} +  \frac{\|\Vec{\sigma}\|_1}{\sqrt{M_k}} $$ instead of
$$ \sum_{j=1}^d  |\nabla f (x^k)|_j \cdot \mathbb{P}(\sign([\nabla f (x^k))]_j \neq [\sign(g^k)]_j  )
   \leq  \frac{\|\Vec{\sigma}\|_1}{B_k^\frac{\kappa-1}{\kappa}}. $$
Hence, the final bound on the sum of $\ell_1$-norm of gradients with probability at least $1 - \delta$ is 
\begin{eqnarray}
    \frac{1}{16} \sum\limits_{k=1}^T \gamma_k\|\nabla f (x^k)\|_1 &\leq& \Delta + L_0d\sum_{k=1}^T\gamma_k^2 + 2\sum_{k=1}^T\gamma_k \frac{\|\Vec{\sigma}\|_1}{\sqrt{M_k}} \notag \\
       &+& 6d(\gamma^{max} \|\nabla f (x^1)\|_1  + 2C_TL_0) \log(\nicefrac{1}{\delta}), \quad M_k \geq \frac{160}{\kappa^2}. \notag
\end{eqnarray}
\end{proof}
\begin{theorem}[\textbf{HP complexity for \algname{MajorityVote-SignSGD}, finite horizon}] \label{thm: signsgd majority finite}
  Consider lower-bounded $(L_0, L_1)$-smooth function $f$ (As. \ref{as: bounded}, \ref{as: smooth}) and the gradient estimates corrupted by \textbf{unimodal and
symmetric HT noise with $\kappa > 0$} (As. \ref{as: pBCM}). Then Alg. \ref{alg:majorityvotesignSGDsingle} requires the sample complexity $N$ to achieve $\frac{1}{T} \sum_{k=1}^{T}  \|\nabla f(x^k)\|_1 \leq \varepsilon$ with probability at least $1-\delta$ for:

\textbf{Optimal tuning for $\varepsilon > \frac{8L_0}{L_1\sqrt{d}}$:}  $T = O\left(\frac{\Delta L^\delta_1 d^\frac{3}{2} }{\varepsilon}\right), \gamma_k \equiv \frac{1}{48 L_1^\delta d^\frac32} , M_k \equiv \max \left\{\frac{160}{\kappa^2}, \frac{2^{16}\|\Vec{\sigma}\|^2_1}{\varepsilon^2}\right\}:$
\begin{equation}
    N = O\left(\frac{\Delta L_1^\delta   d^\frac{3}{2}}{\varepsilon}\left[\frac{1}{\kappa^2} +  \left(\frac{\|\Vec{\sigma}\|_1}{\varepsilon}\right)^2\right]\right),   \notag 
\end{equation}
\textbf{Optimal tuning for $\varepsilon \leq \frac{8L_0}{L_1\sqrt{d}}$:} $T = O\left(\frac{\Delta L_0^\delta d }{\varepsilon^2}\right), \gamma_k \equiv \sqrt{\frac{\Delta}{ 80L_0^\delta dT}} , M_k \equiv \max \left\{\frac{160}{\kappa^2}, \frac{2^{16}\|\Vec{\sigma}\|^2_1}{\varepsilon^2}\right\}:$ 
\begin{equation}
   N = O\left(\frac{\Delta L_0^\delta d }{\varepsilon^2}\left[\frac{1}{\kappa^2} +  \left(\frac{\|\Vec{\sigma}\|_1}{\varepsilon}\right)^2\right]\right), \notag 
\end{equation}
\textbf{Arbitrary tuning for $\varepsilon \leq \frac{8L_0}{L_1\sqrt{d}}$:}  $T, \gamma_k \equiv \frac{\gamma_0}{\sqrt{T}}, M_k \equiv \max \{160/\kappa^2, M_0T\}$:
\begin{equation}
    N = O\left(\frac{M_0(\nicefrac{\Delta}{\gamma_0} + L_0^\delta d \gamma_0)^4} {\varepsilon^4} + \frac{1}{M_0}\left(\frac{\|\Vec{\sigma}\|_1}{\varepsilon}\right)^4\right),  \notag 
\end{equation} 
\textbf{Arbitrary tuning for $\varepsilon \geq \frac{8L_0}{L_1\sqrt{d}}$:}  $T, \gamma_k \equiv \gamma_0 \leq \frac{1}{48L^\delta_1 d^\frac32},  M_k \equiv \max \{160/\kappa^2, M_0T^2\}$:
\begin{equation}
    N = O\left(M_0\left(\frac{\Delta } {\varepsilon\gamma_0} \right)^3 + \frac{1}{M_0^2}\left(\frac{\|\Vec{\sigma}\|_1}{\varepsilon}\right)^3\right),  \notag  
\end{equation} 
where $\Delta = f(x^1) - f^*, L_0^\delta = L_0 \log(\nicefrac{1}{\delta}), L_1^\delta = L_1 \log(\nicefrac{1}{\delta}).$
\end{theorem}
\begin{proof}
    
Plugging in constant stepsizes $\gamma_k \equiv \gamma$ implies $C_T = T\gamma^2, \gamma^{max} = \gamma$ into the bound \eqref{eq: majority signsgd convergence lemma} from Convergence Lemma \ref{lem: majority signsgd T update}, we have :
\begin{equation}
    \frac1T \sum\limits_{k=1}^{T} \|\nabla f (x^k)\|_1 \leq \frac{16\Delta}{T\gamma} + 192 L_0d\gamma \log(\nicefrac{1}{\delta})  + \frac{32 \|\Vec{\sigma}\|_1}{\sqrt{M_k}} + 96\frac{d\|\nabla f (x^1)\|_1}{T}  \log(\nicefrac{1}{\delta}). \label{eq: majorvote convergence}
\end{equation}
\textbf{Case $\varepsilon >  \frac{8L_0}{L_1\sqrt{d}}$, arbitrary tuning:} We use parameters $T, \gamma_k = \gamma_0, M_k = \max \{160/\kappa^2, M_0T^2\}$ to get:
$$\frac1T \sum\limits_{k=1}^{T} \|\nabla f (x^k)\|_1 \leq \frac{16\Delta  }{T\gamma_0} + 32 \frac{\|\Vec{\sigma}\|_1}{\sqrt{M_0} T} + 96\frac{d\|\nabla f (x^1)\|_1}{T}  \log(\nicefrac{1}{\delta}).$$
 Setting such $T$ that the first two terms become less than $\varepsilon$, we obtain the final complexity $N = T \cdot M_0T^2.$

\textbf{Case $\varepsilon \geq  \frac{8L_0}{L_1\sqrt{d}}$, optimal tuning:} We use stepsize $\gamma = \frac{1}{400L_1d\log\frac1\delta \sqrt{d}} \Rightarrow 192 L_0 d\gamma \log(\nicefrac{1}{\delta}) \leq \varepsilon/2$ and  batchsize $32\frac{\|\Vec{\sigma}\|_1}{\sqrt{M_k}} \leq \varepsilon/4 \Rightarrow M_k \equiv \max \left\{\frac{160}{\kappa^2},  \left(\frac{128\|\Vec{\sigma}\|_1}{\varepsilon}\right)^2\right\}$. The number of iterations  $T$ is chosen to bound the first term: 
$$\frac{16\Delta}{T\gamma}  = \frac{2560\Delta L_1\log \frac1\delta d^\frac32}{T} \leq \frac{\varepsilon}{4} \Rightarrow T = O\left(\frac{\Delta L_1  \log \frac1\delta d^\frac{3}{2}}{\varepsilon}\right).$$
The total number of oracle calls is:
\begin{eqnarray}
     N = O\left(\frac{\Delta L_1  \log(\nicefrac{1}{\delta} ) d^\frac{3}{2}}{\varepsilon}\left[\frac{1}{\kappa^2} +  \left(\frac{\|\Vec{\sigma}\|_1}{\varepsilon}\right)^2\right]\right). \notag 
\end{eqnarray}

\textbf{Case $\varepsilon <  \frac{8L_0}{L_1\sqrt{d}}$, arbitrary tuning:} We use parameters $T, \gamma_k = \frac{\gamma_0}{\sqrt{T}}, M_k = \max \{160/\kappa^2, M_0T\}$ to get:
$$\frac1T \sum\limits_{k=1}^{T} \|\nabla f (x^k)\|_1 \leq \frac{16\Delta}{\sqrt{T}\gamma_0} + 192 \frac{L_0d\gamma_0}{\sqrt{T}} \log(\nicefrac{1}{\delta})  + 32 \frac{\|\Vec{\sigma}\|_1}{\sqrt{M_0T}} + 96\frac{d\|\nabla f (x^1)\|_1}{T}  \log(\nicefrac{1}{\delta}).$$
 Setting such $T$ that the first two terms become less than $\varepsilon$, we obtain the final complexity $N = T \cdot M_0T.$

\textbf{Case $\varepsilon < \frac{8L_0}{L_1\sqrt{d}}$, optimal tuning:} We set the same batchsize $32\frac{\|\Vec{\sigma}\|_1}{\sqrt{M_k}} \leq \varepsilon/4 \Rightarrow M_k \equiv \max \left\{\frac{160}{\kappa^2},  \left(\frac{128\|\Vec{\sigma}\|_1}{\varepsilon}\right)^2\right\}$. The stepsize $\gamma$ is set to minimize the sum:
$$\min_\gamma \left[ \frac{16\Delta}{T\gamma} + 192 L_0 d\gamma \log(\nicefrac{1}{\delta}) \right] = 2\sqrt{\frac{3200\Delta L_0 d \log(\nicefrac{1}{\delta})}{T}} ,$$
it means that the stepsize $\gamma = \sqrt{\frac{4\Delta}{80 TL_0\log(\nicefrac{1}{\delta})d}}$. The number of iterations $T$ is chosen to satisfy $$2\sqrt{\frac{3200\Delta L_0  \log(\nicefrac{1}{\delta})d}{T}} \leq \frac{\varepsilon}{2} \Rightarrow T = O \left(\frac{\Delta L_0  \log(\nicefrac{1}{\delta})d}{\varepsilon^2}\right).$$
We only need to check whether condition $\gamma \leq \frac{1}{48L_1d\log\frac1\delta \sqrt{d}}$ holds:
\begin{eqnarray}
    \gamma &=& \sqrt{\frac{4\Delta}{80 TL_0\log(\nicefrac{1}{\delta})d}} = \sqrt{\frac{4\Delta}{ T} \frac{1}{80L_0\log(\nicefrac{1}{\delta})d}} \notag \\
    &\leq& \frac{\varepsilon}{4} \frac{1}{80L_0\log(\nicefrac{1}{\delta})d} \leq \frac{8L_0}{4L_1 \sqrt{d}} \frac{1}{80L_0\log(\nicefrac{1}{\delta})d} \notag \\
    &\leq& \frac{1}{48L_1d\log\frac1\delta \sqrt{d}}. \notag
\end{eqnarray}
Hence, we have the following bound for sample complexity
\begin{eqnarray}
N = O\left(\frac{\Delta L_0\log(\nicefrac{1}{\delta} ) d }{\varepsilon^2}\left[\frac{1}{\kappa^2} +  \left(\frac{\|\Vec{\sigma}\|_1}{\varepsilon}\right)^2\right]\right). 
\end{eqnarray}
\end{proof}

\begin{theorem}[\textbf{HP complexity for \algname{MajorityVote-SignSGD}, infinite horizon}] \label{thm: majority signsgd inifinite}
Consider lower-bounded $(L_0,L_1)$-smooth function $f$ (As. \ref{as: bounded}, \ref{as: smooth}) and HT gradient estimates corrupted by \textbf{unimodal and
symmetric HT noise with $\kappa > 0$} (As. \ref{as: pBCM}). Then Alg. \ref{alg:majorityvotesignSGDsingle} requires the sample complexity $N$  to achieve $\min \limits_{k \in \overline{1,T}} \|\nabla f(x^k)\|_1 \leq \varepsilon$ with probability at least $1-\delta$ for:

\textbf{Arbitrary tuning:}  Until plateau $\gamma_k = \gamma_0 \leq \frac{1}{48L^\delta_1d^\frac32}, M_k =  M_0k^2/\kappa^2$, after $\gamma_k = \frac{\gamma_0}{\sqrt{k}}, M_k = M_0k/\kappa^2$:
\begin{eqnarray}
    \varepsilon \geq \frac{8L_0}{L_1\sqrt{d}} &\Rightarrow& N = \tilde{O}\left( \frac{M_0(\Delta/\gamma_0)^3 + \|\Vec{\sigma}\|_1^3/M_0^2}{\kappa^2\varepsilon^3} \right), \notag \\
    \varepsilon \ll \frac{8L_0}{L_1\sqrt{d}} &\Rightarrow& N =  \tilde{O}\left(\frac{M_0(L_0^\delta \gamma_0 d + \Delta/\gamma_0)^4 + \|\Vec{\sigma}\|_1^4/M_0}{\kappa^2\varepsilon^4}  \right). \notag
\end{eqnarray} 
\textbf{Optimal tuning:} First $\frac{64\Delta L_1^\delta L_1 d^2}{L_0}$ steps,  $\gamma_k = \frac{1}{48L_1^\delta d^\frac32}, M_k = \max \left\{160/\kappa^2, \left(16  k \right)^2\right\}$, after $\gamma_k = \sqrt{\frac{\Delta}{20 d L_0^\delta k}}, M_k = \max \left\{160/\kappa^2, 16  k \right\}:$
\begin{eqnarray}
    \varepsilon \geq \frac{8L_0}{L_1\sqrt{d}} &\Rightarrow& N  = \tilde{O}\left(\frac{(\Delta L_1^\delta d^\frac32 + \|\Vec{\sigma}\|_1)^3}{\kappa^2\varepsilon^3}\right), \notag \\
\varepsilon \ll \frac{8L_0}{L_1\sqrt{d}} &\Rightarrow& N  = \tilde{O}\left(\frac{(\Delta L_0^\delta d + \|\Vec{\sigma}\|^2_1)^2}{\kappa^2 \varepsilon^4}\right).
\end{eqnarray}
where $\Delta = f(x^1) - f^*, L_0^\delta = L_0 \log(\nicefrac{1}{\delta}), L_1^\delta = L_1 \log(\nicefrac{1}{\delta}).$
\end{theorem}
\begin{proof}
    The proof is similar to the proof of Theorem \ref{thm: minibatch signsgd inifinite} with  $\kappa = 2$ and additional condition $M_k \geq 160/\kappa^2$.
\end{proof}

\subsection{Proof of \algname{M-SignSGD} Complexity Theorem \ref{thm:momentum SignSGD}} \label{subsec: MSignSGD proof}
\begin{theorem}[\textbf{Complexity for \algname{M-SignSGD} in expectation, full version}]
Consider lower-bounded $(L_0,L_1)$-smooth function $f$ (As. \ref{as: bounded}, \ref{as: smooth}) and HT gradient estimates (As. \ref{as: pBCM}). Then Alg. \ref{alg:SignSGD-M} requires $T$ iterations  to achieve  $\frac{1}{T} \sum_{k=1}^{T}  \EE \left[ \|\nabla f(x^k)\|_1 \right]  \leq \varepsilon$ starting with $\Delta = f(x^1) - f^*$:

\textbf{Optimal tuning for $\varepsilon \geq \frac{3L_0}{L_1}$:} $\beta_k \equiv 1 - \min\left\{1, \left(\frac{\Delta L_1 \sqrt{d}}{T \|\Vec{\sigma}\|_\kappa}\right)^\frac{\kappa}{2\kappa - 1}\right\}, \gamma_k  \equiv \frac{1 - \beta_k}{8} \frac{1}{L_1d}$ 
\begin{equation}
     T = O\left(\frac{\Delta L_1d}{\varepsilon } \left(1 + \left(\frac{\sqrt{d}\|\Vec{\sigma}\|_\kappa}{\varepsilon}\right)^\frac{\kappa }{\kappa  -1}\right)\right), \notag
\end{equation}
\textbf{Optimal tuning for $\varepsilon < \frac{3L_0}{L_1}$:} $1  - \beta_k \equiv   1 - \min\left\{1, \left(\frac{\Delta L_0}{T \|\Vec{\sigma}\|_\kappa^2}\right)^\frac{\kappa}{3\kappa - 2}  \right\}, \gamma_k  \equiv \sqrt{\frac{\Delta (1 - \beta_k)}{T L_0 d}}$ 
\begin{equation}
     T = O\left(\frac{\Delta L_0d}{\varepsilon^2 } \left(1 + \left(\frac{\sqrt{d}\|\Vec{\sigma}\|_\kappa}{\varepsilon}\right)^\frac{\kappa }{\kappa  -1}\right)\right), \notag
\end{equation}
\textbf{Arbitrary tuning for $\varepsilon \geq \frac{3L_0}{L_1}$:}  $T, \beta_k \equiv 1 - \nicefrac{1}{T^\frac{2}{3}}, \gamma_k \equiv \gamma_0 (1 - \beta_k), \gamma_0 \leq \nicefrac{1}{8dL_1}$:
$$T = O\left(\left(\frac{\Delta }{\gamma_0 \varepsilon }\right)^3 +  \left(\frac{\sqrt{d}\|\Vec{\sigma}\|_\kappa}{\varepsilon}\right)^\frac{3\kappa }{2(\kappa  -1)}\right),$$
\textbf{Arbitrary tuning for $\varepsilon < \frac{3L_0}{L_1}$:}  $T, \beta_k \equiv 1 - \nicefrac{1}{\sqrt{T}}, \gamma_k \equiv \gamma_0  T^{-\frac34}$:
$$ T = O\left(\frac{(\nicefrac{\Delta}{\gamma_0} + L_0d \gamma_0)^4} {\varepsilon^4} + \left(\frac{\sqrt{d}\|\Vec{\sigma}\|_\kappa}{\varepsilon}\right)^\frac{2\kappa}{\kappa-1}\right).$$

\end{theorem}

In this proof, we generalize the proof of Theorem $1$ from \cite{sun2023momentum} for HT noise. 
\begin{proof}

    Consider the $k$-th step of \algname{M-SignSGD}. We use $(L_0, L_1)$ step update Lemma \ref{lem: single update step} to estimate:
    \begin{eqnarray}
        f(x^{k+1}) - f(x^k) &\leq& \la \nabla f (x^k), x^{k+1} - x^k \ra + \frac{L_0 + L_1\|\nabla f(x^k)\|_2 }{2}\exp(L_1\|x^{k+1} - x^k\|_2)\|x^{k+1} - x^k\|_2^2 \notag \\ &\leq&  - \gamma_k\| \nabla f (x^k)\|_1 + 2\gamma_k \sqrt{d} \| \epsilon^k\|_2 + \frac{L_0d \gamma_k^2}{2}\exp(L_1\sqrt{d}\gamma_k) \notag \\ &+&  \frac{ L_1d \gamma_k\exp(L_1\sqrt{d}\gamma_k) }{2}\cdot \gamma_k \|\nabla f(x^k)\|_1.\label{eq: m-sign l0 l1 step before E} 
    \end{eqnarray}

Since we set constant steps sizes and momentum, we denote them as $\gamma \equiv \gamma_k$ and $\beta \equiv \beta_k$, respectively. We use notations $\epsilon^k := m^k - \nabla f(x^k)$ and $\theta^k := g^k - \nabla f(x^k)$. Therefore, we have at $k$-th step values:
\begin{eqnarray}
    m^k &=& \beta m^{k-1} + (1-\beta) g^k= \gamma (\epsilon^{k-1} + \nabla f(x^{k-1})) + (1-\gamma)(\theta^k + \nabla f(x^k)),\notag \\
    \epsilon^k &=& m^k - \nabla f(x^k) = \beta \epsilon^{k-1} + \beta(\underset{=:s^k}{\underbrace{\nabla f (x^{k-1}) - \nabla f (x^k)}} ) + (1- \beta)\theta^k, \notag\\
    \epsilon^k &=& m^k - \nabla f(x^k) = \beta \epsilon^{k-1} + \beta s^k + (1- \beta)\theta^k.\notag 
\end{eqnarray}

Unrolling the recursion, we obtain an explicit formula (upper index of $\beta$ is its power):
\begin{eqnarray}
\epsilon^{k} &=& \beta^{k-1}\epsilon^1 + \sum_{i=2}^{k} \beta^{k-i + 1} s^i + (1-\beta) \sum_{i=2}^{k} \beta^{k-i} \theta^i. \label{eq: unrolling m-signsgd}
\end{eqnarray}

From $(L_0, L_1)-$smoothness of the function $f$ (Lemma \ref{lem: L_0,L_1 smoothness}) follows the bound:
$$\|s^k\|_2 \leq (L_0 + L_1\|\nabla f(x^k)\|_2)\exp(L_1\|x^k-x^{x + 1}\|_2)\|x^k-x^{k + 1}\|_2 = (L_0 + L_1\|\nabla f(x^k)\|_2)\exp(L_1\gamma_k\sqrt{d})\gamma_k\sqrt{d} $$
Denote $\lambda := \exp(L_1\gamma_k\sqrt{d})\gamma_k\sqrt{d}$. Hence, the norm of \eqref{eq: unrolling m-signsgd} can be bounded as:
$$\|\epsilon^k\|_2 \leq \beta^{k-1}\|\epsilon^1\|_2 + L_0\lambda \sum_{i=2}^{k} \beta^{k-i + 1} + L_1\lambda \sum_{i=2}^{k}\beta^{k-i + 1}  \|\nabla f(x^k)\|_2+ (1-\beta) \|\sum_{i=2}^k \beta^{k-i} \theta^i\|_2.$$

We notice that variables $\{\theta_i\}$ are martingale difference sequence from Lemma \ref{lem: batching p} which we plan to use. Due to the formal definition of $\theta^i = g^i - \nabla f(x^i) = \nabla f(x^i, \xi_i) - \nabla f(x^i)$ and \algname{M-SingSGD} step, the conditioning on $ \theta^{i-1}, \dots, \theta^1$  with randomness $\xi_1, \dots, \xi_{i-1} $ is equivalent to the conditioning on point s $x^{i},\dots ,x^{2}$. Hence, we show by definition of martingale difference sequence that  $$\EE[\theta^i| \theta^{i-1}, \dots, \theta^1 ] = \EE[\theta^i|x^{i},\dots ,x^{2}] = \EE[\nabla f(x^i, \xi_i) - \nabla f(x^i)|x^{i},\dots ,x^{2}] = 0.$$

To take math expectation from both sides, we first take it from the term
\begin{eqnarray}
    \EE \left[ \|\sum_{i=2}^k \beta^{k-i} \theta^i\|_2 \right]\leq \left(\EE \left[\|\sum_{i=2}^k \beta^{k-i} \theta^i\|_2^\kappa\right]\right)^\frac1\kappa \overset{\text{Lem. }\ref{lem: batching p} }{\leq} \left(\sum_{i=2}^k 2 \EE \left[\| \beta^{(k-i)}\theta^i\|_2^\kappa\right]\right)^\frac1\kappa \leq  \left(\sum_{i=2}^k 2 \beta^{\kappa(k-i)}\EE \left[\| \theta^i\|_2^\kappa\right]\right)^\frac1\kappa. \notag \label{eq: m-signsgd 1} 
\end{eqnarray}
For each $i \in \overline{2,T}$, we estimate $\EE \left[\| \theta^i\|_2^\kappa\right]$ as 
\begin{eqnarray}
    \EE \left[\| \theta^i\|_2^\kappa\right] \overset{\eqref{eq: norm relation}}{\leq } \EE \left[\| \theta^i\|_\kappa^\kappa\right] = \EE \left[\sum_{j=1}^d| g^k_j - \nabla f(x^k)_j|^\kappa\right] \overset{As. \ref{as: pBCM}}{\leq }  \sum_{j=1}^d \sigma^\kappa_j = \|\Vec{\sigma}\|_\kappa^\kappa. 
\end{eqnarray}
We continue bounding \eqref{eq: m-signsgd 1} with
\begin{eqnarray}
    \eqref{eq: m-signsgd 1} \leq \left(\sum_{i=2}^k 2 \beta^{\kappa(k-i)}\|\Vec{\sigma}\|_\kappa^\kappa\right)^\frac1\kappa \leq \frac{2\|\Vec{\sigma}\|_\kappa}{(1 - \beta^\kappa)^\frac1\kappa}. \notag
\end{eqnarray}
Therefore, the final math expectation can be calculated as:
\begin{eqnarray}
     \EE\|\epsilon^k\|_2 &\leq& \beta^{k-1} \EE \|\epsilon^1\|_2 + \frac{L\sqrt{d}\gamma}{1 - \beta}   + \frac{2(1-\beta) \|\Vec{\sigma}\|_\kappa}{(1 - \beta^\kappa)^\frac1\kappa}.
\end{eqnarray}


Then, we take math expectation from \eqref{eq: m-sign l0 l1 step before E}:
\begin{eqnarray}
   \EE[f(x^{k+1})] - \EE[f(x^k)] &\leq& - \gamma \EE[\|\nabla f(x^k)\|_1] + 2\gamma \sqrt{d} \beta^{k-1} \EE \|\epsilon^1\|_2 \notag\\ &+& \notag L_0\lambda \frac{2\gamma \sqrt{d}}{1 - \beta} + L_1\lambda 2\gamma \sqrt{d} \sum_{i=2}^{k}\beta^{k-i + 1} \EE \|\nabla f(x^k)\|_1 + \frac{4\gamma \sqrt{d}(1-\beta) \|\Vec{\sigma}\|_\kappa}{(1 - \beta^\kappa)^\frac1\kappa} \notag \\&+& \frac{L_0\sqrt{d} \gamma}{2}\lambda \notag  + \frac{ L_1 \sqrt{d} \gamma }{2}\lambda \EE \|\nabla f(x^k)\|_1. \notag 
\end{eqnarray}
Summing it over $k$, we derive
\begin{eqnarray}
   f^* - f(x^1)&\leq& - \gamma \sum_{k=1}^T\EE\|\nabla f(x^k)\|_1 + 2\gamma T \sqrt{d} \beta^{k-1} \EE \|\epsilon^1\|_2\notag \\ &+&  L_1\lambda 2\gamma \sqrt{d}  \sum_{k=1}^T\sum_{i=2}^{k}\beta^{k-i + 1} \EE \|\nabla f(x^i)\|_1 + \frac{4\gamma T \sqrt{d}(1-\beta) \|\Vec{\sigma}\|_\kappa}{(1 - \beta^\kappa)^\frac1\kappa} \notag \\ &+& \frac{L_0 T \sqrt{d} \gamma}{2}\lambda   + \frac{ L_1 \sqrt{d} \gamma }{2}\lambda \sum_{k=1}^T \EE \|\nabla f(x^k)\|_1. \label{eq:m-signsgd prefinal}
\end{eqnarray}
Changing the order of summation in the right part of \eqref{eq:m-signsgd prefinal}, we obtain:
\begin{eqnarray} 
2\gamma L_1\lambda \sqrt{d} \sum_{k=1}^T  
 \left( \sum_{i=2}^{k}\beta^{k-i + 1} \EE \|\nabla f(x^i)\|_1 \right) \notag \notag  &=&   2\gamma L_1\lambda \sqrt{d}  \sum_{i=2}^T\left( \sum_{k = i}^{T}\beta^{k-i + 1} \EE \|\nabla f(x^i)\|_1 \right) 
 \notag \\ &=&   2\gamma L_1\lambda \sqrt{d}  \sum_{i=2}^T \beta^{-i} \left( \sum_{k = i}^{T}\beta^{k + 1} \right) \EE \|\nabla f(x^i)\|_1 \notag \\ &=&  2\gamma L_1\lambda \sqrt{d}  \sum_{i=2}^T \beta^{-i + 1} \beta^i \left(\frac{1 - \beta^{T - i}}{1 - \beta} \right) \EE \|\nabla f(x^i)\|_1 \notag  \\ &\leq&   2\gamma L_1\lambda \sqrt{d}  \sum_{i=2}^T  \beta \left(\frac{1}{1 - \beta} \right) \EE \|\nabla f(x^i)\|_1. \notag  
    \label{eq: m-signsgd 2} 
\end{eqnarray}
Finally, we have the bound 
\begin{eqnarray}
   f^* - f(x^1)  &\leq&  - \gamma \sum_{k=1}^T\EE\|\nabla f(x^k)\|_1 + \frac{2\gamma \sqrt{d} \EE \|\epsilon^1\|_2}{1 - \beta}  \notag\\ &+& 2\gamma L_1\lambda \sqrt{d} \cdot \frac{\beta}{1 - \beta} \sum_{k = 1}^T \EE \|\nabla f(x^k)\|_1 \notag + \frac{4\gamma T \sqrt{d}(1-\beta)\|\Vec{\sigma}\|_\kappa}{(1 - \beta^\kappa)^{1/\kappa}} \\ &+& \frac{L_0 T \sqrt{d} \gamma}{2 (1 - \beta)}\lambda + \frac{L_1 \sqrt{d} \gamma}{2}\lambda \sum_{k=1}^T \EE \|\nabla f(x^k)\|_1 \notag 
 \\ &\leq&  \left(- \gamma +  \frac{2\gamma L_1\lambda \sqrt{d}\beta}{1 - \beta} +\frac{L_1 \sqrt{d} \gamma}{2}\lambda  \right) \sum_{k=1}^T\EE\|\nabla f(x^k)\|_1 \notag \\ &+&   \frac{2\gamma \sqrt{d} \EE \|\epsilon^1\|_2}{1 - \beta} + \frac{4\gamma T \sqrt{d}(1-\beta)\|\Vec{\sigma}\|_\kappa}{(1 - \beta^\kappa)^{1/\kappa}} + \frac{L_0 T \sqrt{d} \gamma}{2 (1 - \beta)}\lambda.
 \label{eq: m-signsgd final} 
\end{eqnarray}
Let us set stepsize $\gamma$ such that 
$$\frac{2\gamma L_1\lambda \sqrt{d}\beta}{1 - \beta} +\frac{L_1 \sqrt{d} \gamma}{2}\lambda \leq \frac{3\gamma^2 L_1 d \exp(L_1 d \gamma)}{1 - \beta}  \leq \gamma/2 \Rightarrow\gamma \leq \frac{1 - \beta}{8} \frac{1}{L_1d}. $$
Thus, we obtain
\begin{eqnarray}
    f^* - f(x^1)  &\leq&  -\frac{\gamma}{2} \sum_{k=1}^T\EE\|\nabla f(x^k)\|_1 + \frac{2\gamma \sqrt{d} \EE \|\epsilon^1\|_2}{1 - \beta} + 4\gamma T \sqrt{d}(1-\beta)^\frac{\kappa - 1}{\kappa}\|\Vec{\sigma}\|_\kappa + \frac{L_0 T d \gamma^2}{(1 - \beta)}, \notag \\
    \frac{1}{T} \sum_{k=1}^T\EE\|\nabla f(x^k)\|_1 &\leq& \frac{2(f^* - f(x^1))}{\gamma T} + \frac{4 \sqrt{d}  \EE \|\epsilon^1\|_2}{T(1 - \beta)} + 8 \sqrt{d}(1-\beta)^\frac{\kappa - 1}{\kappa}\|\Vec{\sigma}\|_\kappa + \frac{2L_0 d \gamma}{(1 - \beta)}.
\end{eqnarray}
\textbf{Case $\varepsilon \geq \frac{3L_0}{L_1}$, arbitrary tuning:} We set $1 - \beta = \frac{1}{T^\frac23}, \gamma = \frac{\gamma_0(1-\beta)}{8d^\frac32}$ and obtain
$$\frac{1}{T} \sum_{k=1}^T \EE\|\nabla f(x^k)\|_1 \leq \frac{16\Delta d^\frac32}{\gamma_0 T^\frac13} +  \frac{8\sqrt{d} \|\Vec{\sigma}\|_\kappa}{T^\frac{2(\kappa - 1)}{3\kappa}} + \frac{\varepsilon}{4}.$$
Next, we choose $T$ to limit $\frac{16\Delta d^\frac32}{\gamma_0 T^\frac13}\leq \frac{\varepsilon}{2}$ and $\frac{8\sqrt{d} \|\Vec{\sigma}\|_\kappa}{T^\frac{2(\kappa - 1)}{3\kappa}} \leq \frac{\varepsilon}{4}:$
$$T = O\left(\left(\frac{\Delta d^\frac32}{\gamma_0 \varepsilon }\right)^3 +  \left(\frac{\sqrt{d}\|\Vec{\sigma}\|_\kappa}{\varepsilon}\right)^\frac{3\kappa }{2(\kappa  -1)}\right).$$

\textbf{Case $\varepsilon \geq \frac{3L_0}{L_1}$, optimal tuning:} 
We choose the stepsize $\gamma  = \frac{1 - \beta}{8} \frac{1}{L_1d} \leq \frac{1 - \beta}{8} \frac{1}{L_1d}$ and get:
\begin{eqnarray}
    \frac{1}{T} \sum_{k=1}^T\EE\|\nabla f(x^k)\|_1 &\leq& \frac{16\Delta L_1 d}{T(1 - \beta)} + \frac{4 \sqrt{d}  \EE \|\epsilon^1\|_2}{T(1 - \beta)} + 8 \sqrt{d}(1-\beta)^\frac{\kappa - 1}{\kappa}\|\Vec{\sigma}\|_\kappa + \frac{4L_0}{L_1} \notag \\
    &\leq& \frac{16(\Delta L_1 + \EE \|\epsilon^1\|_2) d}{T(1 - \beta)}+ 8 \sqrt{d}(1-\beta)^\frac{\kappa - 1}{\kappa}\|\Vec{\sigma}\|_\kappa + \frac{4\varepsilon}{3}. \notag
\end{eqnarray}
Then, we  choose $1 - \beta = \min\left\{1, \left(\frac{\Delta L_1 \sqrt{d}}{T \|\Vec{\sigma}\|_\kappa}\right)^\frac{\kappa}{2\kappa - 1}\right\}$ to obtain
\begin{eqnarray}
    \min_{\beta \in [0,1)} \left[ \frac{16\Delta L_1 d}{T(1 - \beta)}+ 8 \sqrt{d}(1-\beta)^\frac{\kappa - 1}{\kappa}\|\Vec{\sigma}\|_\kappa \right] \leq 24\sqrt{d} \left( \frac{\Delta L_1 \sqrt{d}}{T}\right)^\frac{\kappa - 1}{2\kappa - 1}\|\Vec{\sigma}\|_\kappa^\frac{\kappa}{2\kappa - 1} + \frac{24\Delta L_1 d}{T}.
\end{eqnarray}
Finally, we choose number of iterations $T$ to get:
$$24\sqrt{d} \left( \frac{\Delta L_1 \sqrt{d}}{T}\right)^\frac{\kappa - 1}{2\kappa - 1}\|\Vec{\sigma}\|_\kappa^\frac{\kappa}{2\kappa - 1} + \frac{24\Delta L_1 d}{T} \leq \varepsilon  \Rightarrow T = O\left(\frac{\Delta L_1d}{\varepsilon } \left(1 + \left(\frac{\sqrt{d}\|\Vec{\sigma}\|_\kappa}{\varepsilon}\right)^\frac{\kappa }{\kappa  -1}\right)\right).$$

\textbf{Case $\varepsilon \leq \frac{3L_0}{L_1}$, arbitrary tuning:} We set $1 - \beta = \frac{1}{\sqrt{T}}, \gamma = \gamma_0 T^{-\frac{3}{4}}$ and obtain
$$\frac{1}{T} \sum_{k=1}^T \EE\|\nabla f(x^k)\|_1 \leq \frac{2\Delta}{\gamma_0 T^\frac14} + \frac{2Ld\gamma_0}{T^\frac14} + \frac{8\sqrt{d} \|\Vec{\sigma}\|_\kappa}{T^\frac{\kappa - 1}{2\kappa}} + \frac{4\sqrt{d} \|\epsilon^0\|_2}{T^\frac12}.$$
Next, we choose $T$ to limit $\frac{2\Delta/\gamma_0 + 2Ld\gamma_0}{ T^\frac14} \leq \frac{\varepsilon}{2}$ and $\frac{8\sqrt{d} \|\Vec{\sigma}\|_\kappa}{T^\frac{\kappa - 1}{2\kappa}} \leq \frac{\varepsilon}{2}.$

\textbf{Case $\varepsilon \leq \frac{3L_0}{L_1}$, optimal tuning:} We choose stepsize $\gamma = \sqrt{\frac{\Delta (1 - \beta)}{T L_0 d}}$ to minimize the sum
\begin{eqnarray}
    \min_\gamma \left[\frac{2(f^* - f(x^1))}{\gamma T} + \frac{2L_0 d \gamma}{(1 - \beta)} \right] = 4\sqrt{\frac{\Delta L_0 d}{T(1-\beta)}}, \notag
\end{eqnarray}
\begin{eqnarray}
    \frac{1}{T} \sum_{k=1}^T\EE\|\nabla f(x^k)\|_1 &\leq&  \frac{4 \sqrt{d}  \EE \|\epsilon^1\|_2}{T(1 - \beta)} + 4\sqrt{\frac{\Delta L_0 d}{T(1-\beta)}} + 8 \sqrt{d}(1-\beta)^\frac{\kappa - 1}{\kappa}\|\Vec{\sigma}\|_\kappa.
\end{eqnarray}
The first term is much smaller than the second one, hence we omit it. Next, we choose $1  - \beta =   \min\left\{1, \left(\frac{\Delta L_0}{T \|\Vec{\sigma}\|_\kappa^2}\right)^\frac{\kappa}{3\kappa - 2}  \right\}$ to minimize the last two terms:
$$\min_{\beta \in [0,1)} \left[ 4\sqrt{\frac{\Delta L_0 d}{T(1-\beta)}} + 8 \sqrt{d}(1-\beta)^\frac{\kappa - 1}{\kappa}\|\Vec{\sigma}\|_\kappa \right] \leq 12 \sqrt{d} \left(\frac{\Delta L_0}{T}\right)^\frac{\kappa - 1}{3\kappa - 2} \|\Vec{\sigma}\|_\kappa^\frac{\kappa}{3\kappa  - 2} + 12 \sqrt{\frac{\Delta L_0 d}{T}}.$$
Finally, we choose number of iterations $T$ to satisfy:
\begin{eqnarray}
    12 \sqrt{d} \left(\frac{\Delta L_0}{T}\right)^\frac{\kappa - 1}{3\kappa - 2} \|\Vec{\sigma}\|_\kappa^\frac{\kappa}{3\kappa  - 2}  + 12 \sqrt{\frac{\Delta L_0 d}{T}}\leq \frac{\varepsilon}{2} \Rightarrow T = O\left(\frac{\Delta L_0d}{\varepsilon^2 } \left( 1 + \left(\frac{\sqrt{d}\|\Vec{\sigma}\|_\kappa}{\varepsilon}\right)^\frac{\kappa }{\kappa  -1}\right)\right).
\end{eqnarray}
We only need to check that
$$\gamma = \sqrt{\frac{\Delta (1 - \beta)}{T L_0 d}} = \sqrt{\frac{\Delta L_0 d}{T (1-\beta)}} \frac{(1-\beta)}{L_0 d}  \leq \frac{\varepsilon}{2 \cdot 12} \frac{(1-\beta)}{L_0 d} \overset{\varepsilon \leq \frac{3L_0}{L_1}}{\leq} \frac{(1-\beta)}{L_1 d}.$$
\end{proof}

\section{Restarted \algname{minibatch-SignSGD} and \algname{MajorityVote-SignSGD}} \label{subsec:restarted}

For PL functions (As. \ref{as: PL}), we can apply restart technique to \algname{minibatch-SignSGD} and \algname{MajorityVote-SignSGD}.  At each round, we run a base algorithm with certain parameters and then aggregate the output point. This output point is used as an initial point for the next round.  

\begin{algorithm}[ht!]
\caption{\algname{Restarted-$\mathcal{A}$} }
\label{alg:restarted}   
\begin{algorithmic}[1]
\REQUIRE Starting point $x^0 \in \R^d$, number of restarts $\tau$, base algorithm $\mathcal{A}$, parameters $\{\theta_n\}_{n=1}^\tau$.

\FOR{$n=1,\ldots, \tau$}

\STATE Run Algorithm $\mathcal{A}$ with parameters $\theta_n$ and initial point $x^{n-1}$;

\STATE Set $x^n$ as the aggregated output point from the previous round: the point with the minimal $\ell_2$ gradient norm;
\ENDFOR
\ENSURE $x^\tau$
\end{algorithmic}
\end{algorithm}

\begin{theorem}[\textbf{HP complexity for Restarted \algname{minibatch-SignSGD}}]\label{thm:restarted minibatch SignSGD}
Consider lower-bounded $(L_0,L_1)$-smooth, $\mu$-PL function $f$ (As. \ref{as: bounded}, \ref{as: smooth}, \ref{as: PL}) and HT gradient estimates (As. \ref{as: pBCM}). Then restarted \algname{minibatch-SignSGD}  requires the sample complexity $N$  to achieve $f(x^\tau) - f(x^*) \leq \varepsilon$ with probability at least $1-\delta$ for:

\textbf{Optimal tuning for $\varepsilon \geq  (\frac{8L_0}{L_1\sqrt{d}})^2$:} $\tau  = \log(\Delta/\varepsilon)$, iterations $ T_n = O\left(\frac {  L_1^\delta d^\frac32 \sqrt{\Delta}}{2^{n/2}\sqrt{\mu}} \right), $ constant batchsizes $B_n \equiv \max \left\{1,  \left(\frac{1024\|\Vec{\sigma}\|^2_1}{\mu\varepsilon}\right)^\frac{\kappa}{2(\kappa-1)}\right\}$, constant stepsizes $\gamma_n \equiv  \frac{1}{48L^\delta_1d\sqrt{d}}:$ 
\begin{equation}
   N = O\left( \frac {  L_1^\delta d^\frac32 \sqrt{\Delta}}{\sqrt{\mu}} \left[1 + \left(\frac{\|\Vec{\sigma}\|^2_1}{\mu \varepsilon}\right)^\frac{\kappa}{2(\kappa-1)}\right]\right)\notag 
\end{equation}

\textbf{Optimal tuning for $\varepsilon < (\frac{8L_0}{L_1\sqrt{d}})^2$:}  $\tau  = \log(\Delta/\varepsilon)$, iterations $ T_n = O\left(\frac {  L_0^\delta d}{\mu} \right), $ constant batchsizes $B_n \equiv \max \left\{1,  \left(\frac{1024\|\Vec{\sigma}\|^2_1}{\mu\varepsilon}\right)^\frac{\kappa}{2(\kappa-1)}\right\}$, constant stepsizes $\gamma_n \equiv  \sqrt{\frac{\Delta}{2^{n+4}T_nL_0^\delta d}}:$
\begin{equation}
    N = O\left( \frac {  L_0^\delta d \log\frac{\Delta}{\varepsilon}}{\mu} \left[1 + \left(\frac{\|\Vec{\sigma}\|^2_1}{\mu\varepsilon}\right)^\frac{\kappa}{2(\kappa-1)}\right]\right),  \notag
\end{equation}
where $\Delta = f(x^1) - f^*, L_0^\delta = L_0 \log(\frac{\log\frac{\Delta}{\varepsilon}}{\delta }), L_1^\delta = L_1 \log(\frac{\log\frac{\Delta}{\varepsilon}}{\delta }).$
\end{theorem}

\begin{theorem}[\textbf{HP complexity for Restarted \algname{MajorityVote-SignSGD}}]\label{thm:restarted majority SignSGD}
Consider lower-bounded $(L_0,L_1)$-smooth, $\mu$-PL function $f$ (As. \ref{as: bounded}, \ref{as: smooth}, \ref{as: PL}) and HT gradient estimates corrupted by \textbf{unimodal and
symmetric HT noise with $\kappa > 0$} 
 (As. \ref{as: pBCM}). Then restarted \algname{MajorityVote-SignSGD}  requires the sample complexity $N$  to achieve $f(x^\tau) - f(x^*) \leq \varepsilon$ with probability at least $1-\delta$ for:

\textbf{Optimal tuning for $\varepsilon \geq  (\frac{8L_0}{L_1\sqrt{d}})^2$:} $\tau  = \log(\Delta/\varepsilon)$, iterations $ T_n = O\left(\frac {  L_1^\delta d^\frac32 \sqrt{\Delta}}{2^{n/2}\sqrt{\mu}} \right), $ constant batchsizes $M_n \equiv \max \left\{\frac{160}{\kappa^2},  \frac{1024\|\Vec{\sigma}\|^2_1}{\mu\varepsilon}\right\}$, constant stepsizes $\gamma_n \equiv  \frac{1}{48L^\delta_1d\sqrt{d}}:$ 
\begin{equation}
   N = O\left( \frac {  L_1^\delta d^\frac32 \sqrt{\Delta}}{\sqrt{\mu}} \left[\frac{1}{\kappa^2} + \frac{\|\Vec{\sigma}\|^2_1}{\mu \varepsilon}\right]\right)\notag
\end{equation}

\textbf{Optimal tuning for $\varepsilon < (\frac{8L_0}{L_1\sqrt{d}})^2$:}  $\tau  = \log(\Delta/\varepsilon)$, iterations $ T_n = O\left(\frac {  L_0^\delta d}{\mu} \right), $ constant batchsizes $M_n \equiv \max \left\{\frac{160}{\kappa^2},  \frac{1024\|\Vec{\sigma}\|^2_1}{\mu\varepsilon}\right\}$, constant stepsizes $\gamma_n \equiv  \sqrt{\frac{\Delta}{2^{n+4}T_nL_0^\delta d}}:$
\begin{equation}
    N = O\left( \frac {  L_0^\delta d \log\frac{\Delta}{\varepsilon}}{\mu} \left[\frac{1}{\kappa^2} + \frac{\|\Vec{\sigma}\|^2_1}{\mu\varepsilon}\right]\right),  \notag
\end{equation}
where $\Delta = f(x^1) - f^*, L_0^\delta = L_0 \log(\frac{\log\frac{\Delta}{\varepsilon}}{\delta }), L_1^\delta = L_1 \log(\frac{\log\frac{\Delta}{\varepsilon}}{\delta }).$
\end{theorem}

\begin{proof}
Here we prove only Theorem \ref{thm:restarted minibatch SignSGD}. The proof of Theorem \ref{thm:restarted majority SignSGD} is similar. 

Consider one round of restarted algorithm with the initial condition $\Delta$ which will be transformed into $\Delta_2 \leq \Delta/2$. In total, we will have $\log(\Delta/\varepsilon)$ rounds. 
Instead of the initial failure probability $\delta$ we use decreased probability  $\delta /\log \frac{\Delta}{\varepsilon}$, since the probability of holding bounds  $\log(\Delta/\varepsilon)$ times for all restarts equals to $(1 - \frac{\delta}{\log(\Delta/\varepsilon)})^{\log(\Delta/\varepsilon)} \geq (1 - \frac{\delta}{\log(\Delta/\varepsilon)} \log(\Delta/\varepsilon)) = (1 - \delta).$

Plugging in constant stepsizes $\gamma_k \equiv \gamma \leq \frac{1}{48L^\delta_1d\sqrt{d}}$ in \eqref{eq: proof signsgd conv main eq} implies $C_T = T\gamma^2, \gamma^{max} = \gamma$:
$$\frac1T \sum\limits_{k=1}^{T} \|\nabla f (x^k)\|_1 \leq \frac{16\Delta}{T\gamma} + 256 L^\delta_0 d\gamma   + 32 \|\Vec{\sigma}_k\|_1.$$
Due to Batching Lemma \ref{lem: batching p}, we can estimate the $\kappa-$th moment of the batched estimate for constant batchsizes $B_k  \equiv B$ as
   $\|\Vec{\sigma}_k\|_1 \leq \frac{2\|\Vec{\sigma}\|_1}{B^{\frac{\kappa-1}{\kappa}}}$  and  derive:
\begin{eqnarray} \notag
      \min_{k\in \overline{1,T}} \|\nabla f (x^k)\|_2 \leq \frac1T \sum\limits_{k=1}^{T} \|\nabla f (x^k)\|_2 \leq \frac1T \sum\limits_{k=1}^{T} \|\nabla f (x^k)\|_1 \leq \frac{16\Delta}{T\gamma} + 256 L^\delta_0 d\gamma + 32 \frac{\|\Vec{\sigma}_k\|_1}{B^{\frac{\kappa-1}{\kappa}}}. \notag 
\end{eqnarray}
Next, we square the inequality and apply PL condition:
\begin{eqnarray}
    \|\nabla f (x^T_{min})\|^2_2 &\leq &8\left(\frac{16\Delta}{T\gamma}\right)^2 + 8(256 L^\delta_0 d\gamma)^2  + 8\left(32 \frac{\|\Vec{\sigma}_k\|_1}{B^{\frac{\kappa-1}{\kappa}}}\right)^2, \notag \\
    f(x^T_{min}) - f(x^*) &\leq&  \|\nabla f (x^T_{min})\|^2_2 \leq  8\left(\frac{16\Delta}{T\gamma}\right)^2 + 8(256 L^\delta_0 d\gamma)^2  + 8\left(32 \frac{\|\Vec{\sigma}_k\|_1}{B^{\frac{\kappa-1}{\kappa}}}\right)^2, \notag \\
   \Delta_2 = f(x^T_{min}) - f(x^*) &\leq&  \frac{4}{\mu} \left[\left(\frac{16\Delta}{T\gamma}\right)^2 + (256 L^\delta_0 d\gamma)^2  + \left(32 \frac{\|\Vec{\sigma}_k\|_1}{B^{\frac{\kappa-1}{\kappa}}}\right)^2\right], \notag
\end{eqnarray}
where $x^T_{min}  = \arg\min \limits_{k \in \overline{1,T}} f(x^k)$.

\textbf{Case $\varepsilon \geq  (\frac{8L_0}{L_1\sqrt{d}})^2$, optimal tuning:} We use stepsizes $\gamma = \frac{1}{48L^\delta_1d \sqrt{d}} \Rightarrow (256 L^\delta_0 d\gamma)^2 \leq \varepsilon/2$ and batchsizes $32\frac{\|\Vec{\sigma}\|_1}{B^{\frac{\kappa-1}{\kappa}}} \leq \sqrt{\mu\varepsilon/8} \Rightarrow B_k \equiv \max \left\{1,  \left(\frac{1024\|\Vec{\sigma}\|^2_1}{\mu\varepsilon}\right)^\frac{\kappa}{2(\kappa-1)}\right\}$. The number of iterations  $T$ is chosen to decrease the term $\Delta$ by half: 
$$\Delta_2   \leq \Delta^2 \left(\frac{1028*48^2  (L_1^\delta)^2 d^3} {\mu T^2}\right) \leq \frac{\Delta}{2} \Rightarrow T = O\left(\frac {  L_1^\delta d^\frac32 \sqrt{\Delta}}{\sqrt{\mu}} \right).$$
At each restart, the initial condition $\Delta_n$ becomes $\Delta_{n+1} \leq \frac{\Delta_n}{2} \leq \frac{\Delta}{2^{n-1}}$, and the total number of iterations is
\begin{eqnarray}
T_{total} = \sum_{n=1}^{\log(\Delta/\varepsilon)} O\left(\frac {  L_1^\delta d^\frac32 \sqrt{\Delta}}{\sqrt{\mu} 2^\frac{n-1}{2}}  \right) = O\left(\frac {  L_1^\delta d^\frac32 \sqrt{\Delta}}{\sqrt{\mu}}  \right),
\end{eqnarray}
with the total number of oracle calls:
\begin{eqnarray}
    N = T_{total} * B_k = O\left( \frac {  L_1^\delta d^\frac32 \sqrt{\Delta}}{\sqrt{\mu}} \left[1 + \left(\frac{\|\Vec{\sigma}\|^2_1}{\mu \varepsilon}\right)^\frac{\kappa}{2(\kappa-1)}\right]\right).
\end{eqnarray}

\textbf{Case $\varepsilon \leq  (\frac{8L_0}{L_1\sqrt{d}})^2$, optimal tuning:} We use batchsizes $32\frac{\|\Vec{\sigma}\|_1}{B^{\frac{\kappa-1}{\kappa}}} \leq \sqrt{\mu\varepsilon/8} \Rightarrow B_k \equiv \max \left\{1,  \left(\frac{1024\|\Vec{\sigma}\|^2_1}{\mu\varepsilon}\right)^\frac{\kappa}{2(\kappa-1)}\right\}$ and stepsizes $\gamma = \sqrt{\frac{\Delta}{16TL_0^\delta d}}$ to have
\begin{eqnarray}
    \Delta_2  &\leq& \frac{8}{\mu} \left[  \frac{1024L_0^\delta d}{T} \Delta+  \varepsilon\right]. \notag
\end{eqnarray}
The number of iterations  $T$ is chosen to decrease the term $\Delta$ by half: 
$$\Delta_2   \leq \frac{8}{\mu}  \frac{1024L_0^\delta d}{T} \Delta \leq \frac{\Delta}{2} \Rightarrow T = O\left(\frac {  L_0^\delta d }{\mu} \right).$$
At each restart, the initial condition $\Delta_n$ becomes $\Delta_{n+1} \leq \frac{\Delta_n}{2} \leq \frac{\Delta}{2^{n-1}}$, and the total number of iterations is
\begin{eqnarray}
T_{total} = \sum_{n=1}^{\log(\Delta/\varepsilon)} O\left(\frac {  L_0^\delta d }{\mu} \right) = \left(\frac {  L_0^\delta d \log\frac{\Delta}{\varepsilon}}{\mu} \right), \notag
\end{eqnarray}
with the total number of oracle calls:
\begin{eqnarray}
    N = T_{total} * B_k = O\left( \frac {  L_0^\delta d \log\frac{\Delta}{\varepsilon}}{\mu} \left[1 + \left(\frac{\|\Vec{\sigma}\|^2_1}{\mu\varepsilon}\right)^\frac{\kappa}{2(\kappa-1)}\right]\right).
\end{eqnarray}
\end{proof}

\section{Experimental validation of the theoretical convergence bounds}\label{sec: exps for theory}
In this section, we run experiments to verify the following convergence bound from Lemma \ref{lem: signsgd T update} for the backbone \algname{SignSGD} method: 
$$\sum\limits_{k=1}^T \frac{\gamma_k}{16}\|\nabla f (x^k)\|_1 \leq \Delta + L_0d\sum_{k=1}^T\gamma_k^2 + 2\sum_{k=1}^T\gamma_k \|\Vec{\sigma}_k\|_1 
       + 6d(\gamma_1 \|\nabla f (x^1)\|_1  + 2C_TL_0) \log\frac{1}{\delta}, $$
where $C_T := \max\limits_{k \in \overline{1,T}} \gamma_k \cdot  \sum\limits_{\tau=1}^{k-1}\gamma_\tau, \gamma_k \leq 1/ (48L_1d^\frac32\log\frac1\delta)$ and $\Delta = f(x^1) - f^*$. In case of constant stepsizes $\gamma_k \equiv \gamma$, the bound transforms into
\begin{equation}\label{eq: exps bound}
    \frac1T \sum\limits_{k=1}^{T} \|\nabla f (x^k)\|_1 \leq \frac{4\Delta}{T\gamma} + 80 L_0d\gamma \log(\nicefrac{1}{\delta})  + 8{\|\Vec{\sigma}\|_1} + 24\frac{d\|\nabla f (x^1)\|_1}{T}  \log(\nicefrac{1}{\delta}), \quad \gamma \leq \frac{1}{ (48L_1d^\frac32\log\frac1\delta)}.
\end{equation}

\paragraph{Objective function and noise.} We optimize a non-convex neural network for classification task with features $X$ and one-dimensional labels $y$. The network $NN_\theta$ with parameter vector $\theta$ consists of two fully connected layers, ReLU activation, batch normalization and dropout. The objective function is the following logistic regression with $L_0$ and $L_1$ regularizations with coefficients $\lambda_{L_0}$ and $\lambda_{L_1}$, respectively:
$$f(\theta) = \log(1 + \exp(-\la y, NN_\theta(X)\ra)) + \frac{\lambda_{L_0}}{2} \cdot \|\theta\|_2^2 + \exp(\lambda_{L_1} \cdot \la \Vec{1}, \theta\ra).$$
The regularization coefficients $\lambda_{L_0}, \lambda_{L_1}$ are the smoothness constants of the corresponding regularization terms (see Appendix~\ref{sunbsec: gen smoothness}). If their value are changed by some amount then the actual $L_0, L_1$ smoothness constants of the objective function $f$ are changed by the exactly same amount.

To model the noise, we compute the whole gradient $\nabla f(\theta)$ and artificially add noise vector with independent components sampled from $\alpha$-stable Levy distribution with scale $\sigma$ ($\alpha$ is the $\kappa$ parameter). 

As training data, we consider the dataset \texttt{mushrooms} from LibSVM \cite{chang2011libsvm}. The matrix $X$ has shape $(6499, 112)$, hence, we set the NN layers sizes $(112, 32,1)$ and dropout rate $0.1$.

\paragraph{Noise dependencies.} First, we verify the linear dependency of the achieved accuracy \eqref{eq: exps bound} on noise $\sigma$. We set small regularization coefficients $\lambda_{L_0}  = 0.01, \lambda_{L_1} = 0.001$ and constant stepsize $\gamma = 3\cdot 10^{-4}$ for all experiments in this paragraph. 

Next, we vary $\sigma \in [0.1, 0.01, 0.001]$ and $\kappa \in [2, 1.5, 1]$. The results over $10$ runs with standard deviation bars are depicted in the left graph of Figure \ref{fig:sigma-kappa vary}. 
\begin{figure}
    \centering
    \includegraphics[width=0.3\linewidth]{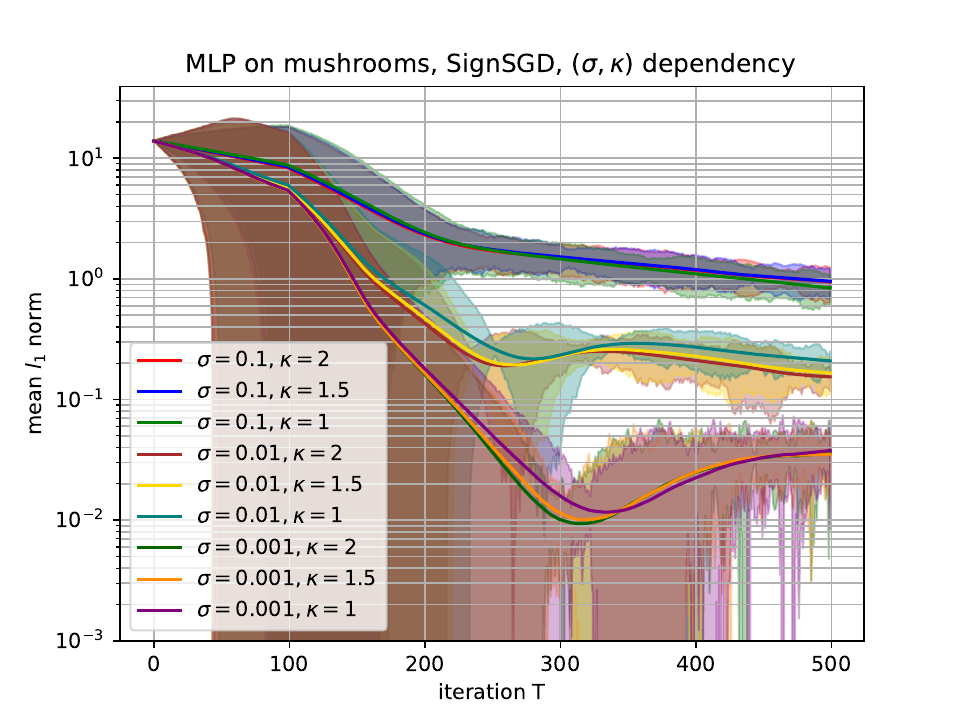}
    \includegraphics[width=0.3\linewidth]{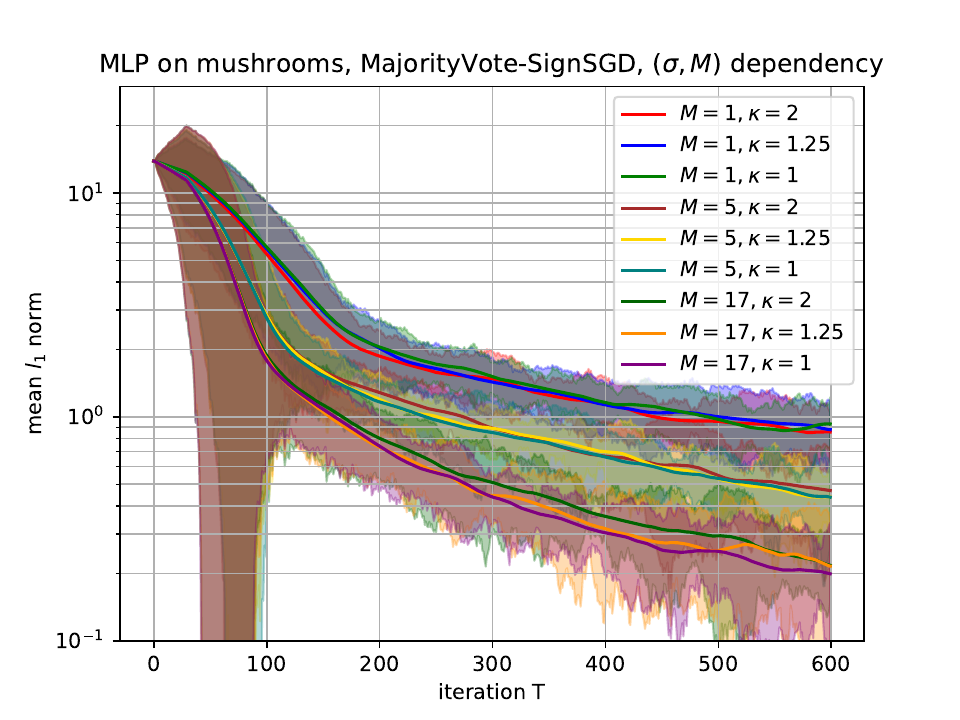}
    \includegraphics[width=0.3\linewidth]{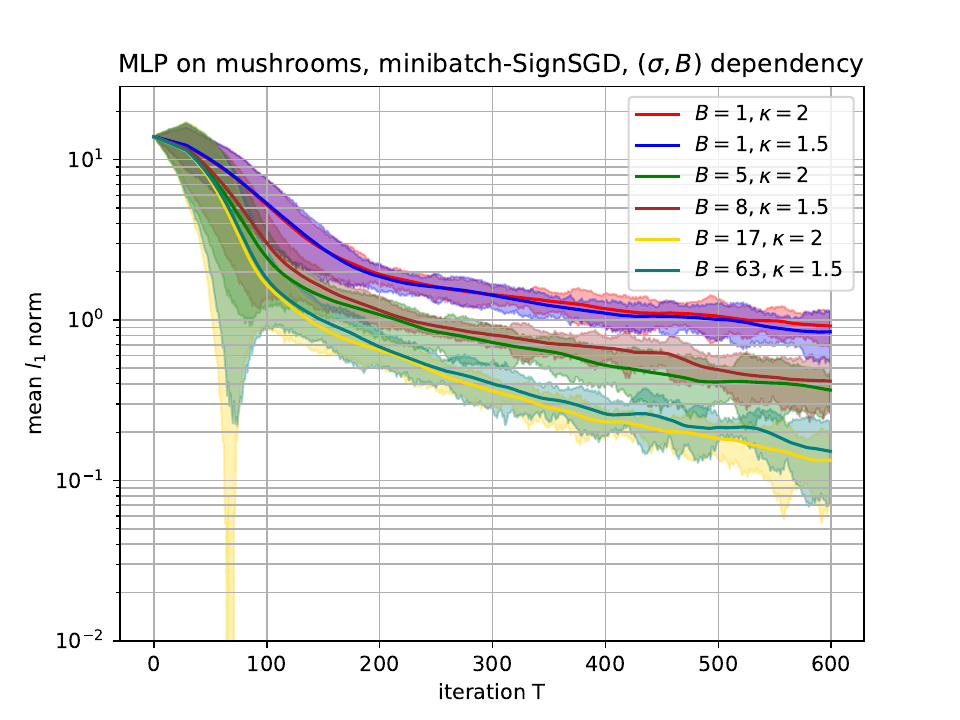}
    \caption{Experimental noise dependencies for $(L_0, L_1)$-smoooth problems.}
    \label{fig:sigma-kappa vary}
\end{figure}

In practice, the achieved accuracy does not depend on $\kappa$, only linearly on $\sigma$. We also wish to highlight the small size of error bars especially in the end of training which corresponds to mild $\log \frac{1}{\delta}$ dependency.     

In addition, we check how well batching (Alg. \ref{alg:minibatch-signSGD}) and majority voting (Alg. \ref{alg:majorityvotesignSGDsingle}) reduce the noise, .i.e, whether reduction laws $\sigma/B^\frac{\kappa - 1}{\kappa }$ and $\sigma/\sqrt{M}$ from Theorems \ref{thm:minibatch SignSGD} and \ref{thm:com-sign conv} hold true.  To reduce the noise by $2$ and $4$ times for majority voting, we use batchsizes $M = 1, 4, 16$ for all $\kappa \in [2, 1.25, 1]$. The results are shown in the middle graph of Figure \ref{fig:sigma-kappa vary}. To reduce the noise by $2$ and $4$ times for batching, we use batchsizes $B = 1, 4, 16$ for $\kappa  =2$ and $B = 1, 8, 64$ for $\kappa = 1.5$. The results are shown in the right graph of Figure \ref{fig:sigma-kappa vary}. In practice, both methods actually reduce the noise according to the theoretical laws. 



\paragraph{Two phase convergence.} Here, we demonstrate the convergence speed slowdown after reaching the accuracy $\frac{8L_0}{L_1\sqrt{d}}$ as it stated in Theorem \ref{thm:minibatch SignSGD}. We also test the arbitrary tuning strategy proposed for reaching this behavior.

We slightly change the setup to better control constants $L_0, L_1$. We replace fully-connected neural network $NN_\theta(X)$ with simple linear transform $NN_\theta(X) = X\cdot \theta$, hence, the current objective function with only $L_0$ regularization is:
$$f(\theta) = \log(1 + \exp(-\la y,  X\cdot \theta \ra)) + \frac{\lambda_{L_0}}{2} \cdot \|\theta\|_2^2.$$

In this case, we can directly compute (see Example \ref{exp: log func L1} for \textit{mushrooms} dataset) and control constants $L_0 = \lambda_{L_0}, L_1 \approx 5.58$. The noise  parameters are $\kappa = 1.5$ and $\sigma = 0.1$. 

First, we set constant stepsize $\gamma = 10^{-1}$ and vary $\lambda_{L_0} \in [0, 10^{-1}, 10^{-2}, 10^{-3}, 10^{-4}, 10^{-5}]$. The results over $10$ runs are depicted in the left graph of Figure \ref{fig: phase trans}. One can see that the final accuracy drops linearly with $\lambda_{L_0}$ until it reaches the noise level. Before the plateau, we observe the fast $L_1$ convergence.  

In the next experiment, we follow the arbitrary tuning strategy and start to decrease stepsizes as $1/\sqrt{k}$ after the plateau. The results are presented in the middle graph of Figure \ref{fig: phase trans}. One can see that now method can slowly reach the same noise level after the first plateau. The speed transition accuracy also drops linearly with $\lambda_{L_0}$.    

Finally, we show that, for functions with $L_0 = 0$, our method with constant stepsize convergences to noise level $\sigma$ despite the value of the constant $L_1$. We set $\lambda_{L_0} = 0$ and vary the noise level $\sigma \in [10^{-1}, 10^{-2}, 10^{-3},10^{-4}, 10^{-5}, 10^{-6}]$. The results are shown in the right graph of Figure~\ref{fig: phase trans}. These results clearly support the theory with alone linear dependency on $\sigma$. 
\begin{figure}
    \centering
    \includegraphics[width=0.3\linewidth]{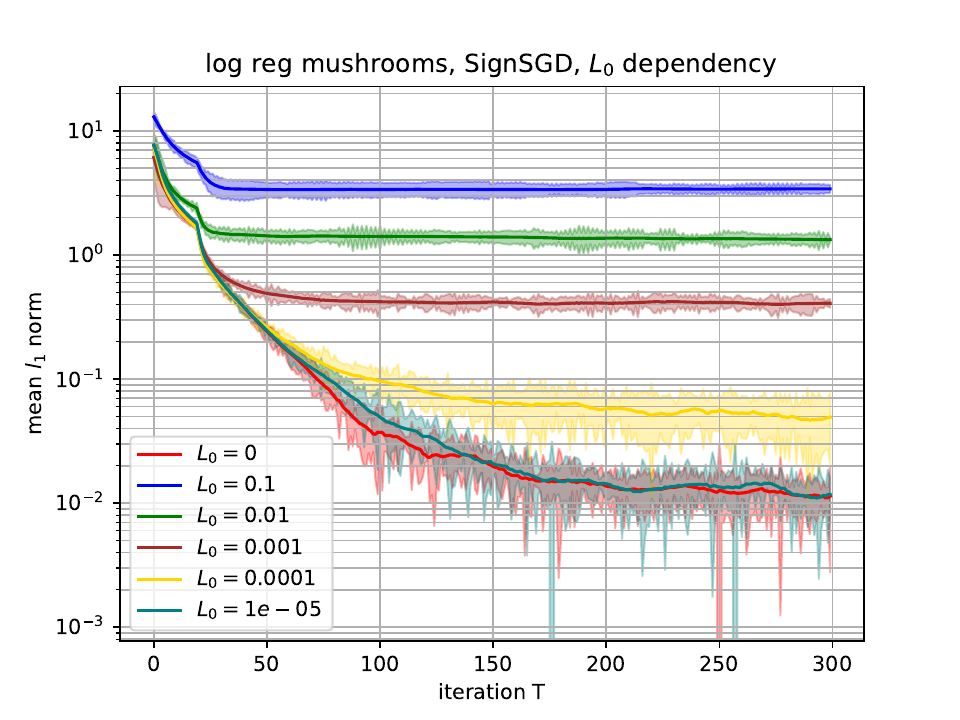}
    \includegraphics[width=0.3\linewidth]{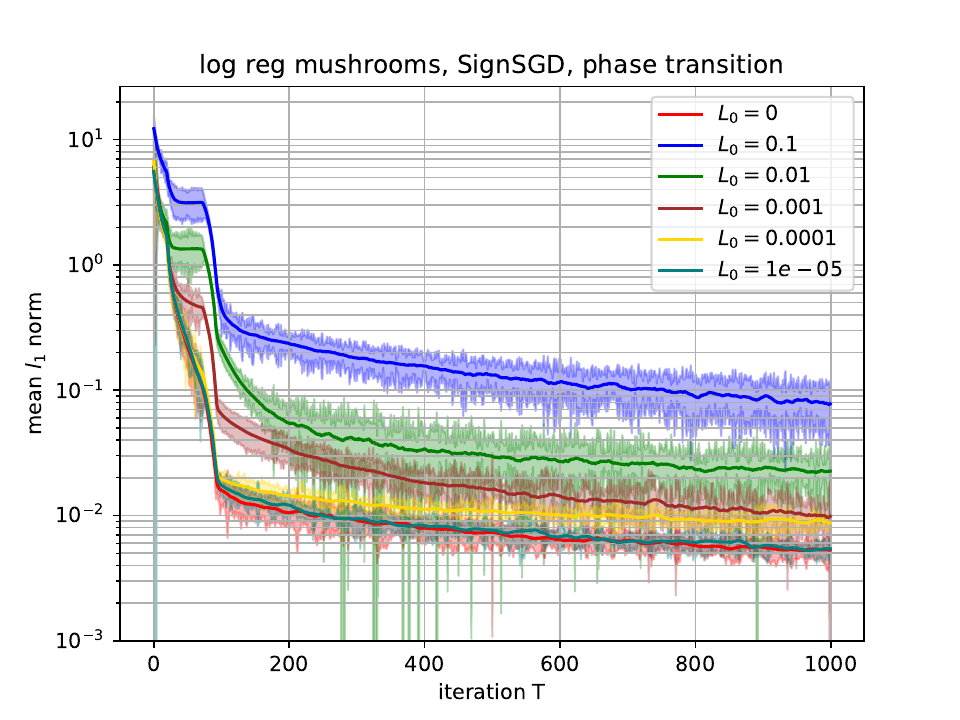}
    \includegraphics[width=0.3\linewidth]{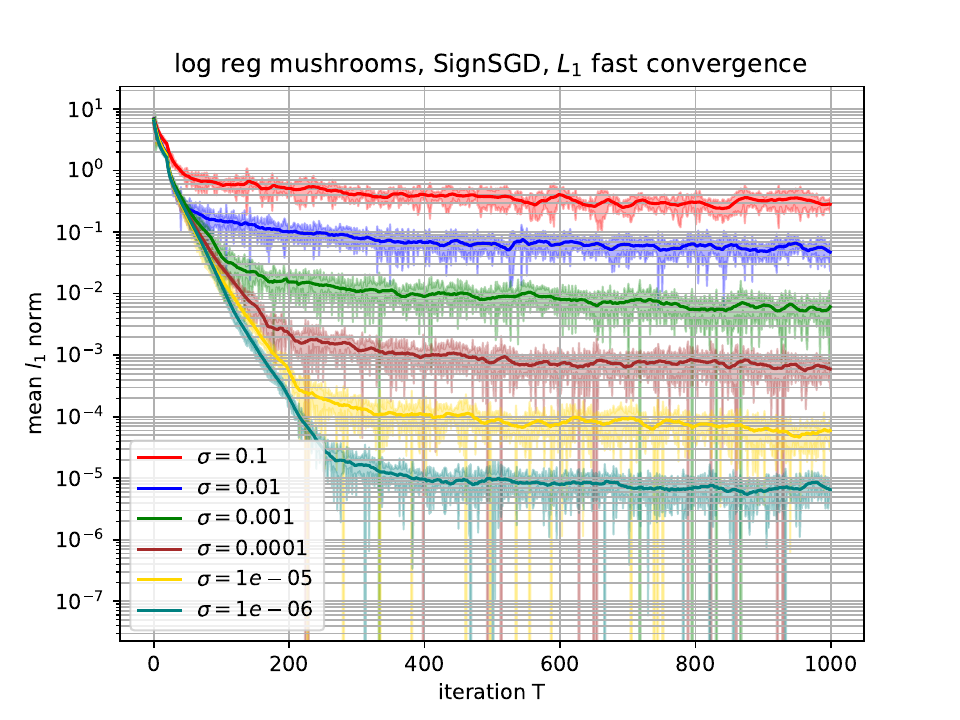}
    \caption{Experimental convergence speed transition for $(L_0, L_1)$-smooth problems.}
    \label{fig: phase trans}
\end{figure}

\section{Additional experiments}\label{sec: moe exps}

\subsection{Mixture of Experts pre-training experiments}

We complement our experiments with another setup -- different architecture and data. In~\Cref{sec:experiments}, we used a dense LLaMA model; now, we have switched to a Mixture of Experts (MoE) architecture based on the same LLaMA model, retaining RoPE and identical activation functions. Our MoE model follows the Switch Transformer~\citep{fedus2022switch} MoE variant with classical top $k=2$ gating and 8 experts, giving us approximately 520M parameters if we have the same configuration as 130M LLaMA. We conduct these experiments on the FineWeb dataset~\citep{penedo2406fineweb} a popular corpus for LLM pre-training.

We run \algname{AdamW}, \algname{M-SignSGD}, \algname{M-NSGD} and \algname{M-ClippedSignSGD} optimizers following the best practices from our earlier setup on dense models.
We train with a batch size of 256 and sequence length 512 for 42k (5.5B tokens) and 336k steps (44B tokens). 
That is for the second training horizon we go far beyond the Chinchilla optimal tokens-per-parameters ration.
The results are presented in~\Cref{tab:moe1,tab:moe2} respectively.

\begin{table}[htbp]
\centering
\caption{Perplexity of LLaMa-base MoE 520M model pre-trained on FineWeb for 42k steps. Lower is better.}
\label{tab:moe1}
\begin{tabular}{lc}
\toprule
\textbf{Optimizer} & \textbf{Perplexity $\downarrow$} \\
\midrule
\algname{AdamW} & \textit{22.85} \\
\midrule
\algname{M-SignSGD} & \textbf{23.19} \\
\algname{M-NSGD} & 23.32 \\
\algname{M-ClippedSignSGD} & 23.30 \\
\bottomrule
\end{tabular}
\end{table}

\begin{table}[htbp]
\centering
\caption{Perplexity of LLaMa-base MoE 520M model pre-trained on FineWeb for 336k steps. Lower is better.}
\label{tab:moe2}
\begin{tabular}{lc}
\toprule
\textbf{Optimizer} & \textbf{Perplexity $\downarrow$} \\
\midrule
\algname{AdamW} & \textit{18.68} \\
\algname{M-SignSGD} & 18.87 \\
\bottomrule
\end{tabular}
\end{table}

We would like to highlight that \algname{M-SignSGD} scales remarkably well with increasing model size, outperforming \algname{M-NSGD} and \algname{M-ClippedSignSGD}. 
Additionally, we encountered difficulties running \algname{M-ClippedSGD} in this setting. 
Consequently, we decided to include a clipped version of \algname{M-SignSGD}, which aligns with our approach since we consider only an EMA of momentum in the update.

\subsection{Robustness with respect to random seed}

To verify the robustness of our approach, we repeated the experiment from~\Cref{tab:pre-training} with three different random seeds. As shown in~\Cref{tab:3seed}, the performance remains highly consistent across all seeds, with a standard deviation $\leq 0.03$ for all the methods. 

\begin{table}[H]
    \caption{Comparison of mean and standard deviation of the validation perplexity for various optimization methods for LLaMA 130M model trained on C4.}
    \label{tab:3seed}
 
    \begin{center}
    \begin{tabular}{c|c}
    \toprule
    \textbf{Method} & {\textbf{Perplexity $\downarrow$}} \\
    \midrule
    Model size & 130M\\
    \midrule
    \algname{M-SignSGD} & \textbf{18.37}$_{\pm .01}$ \\
    \midrule
    \algname{M-NSGD} &  19.28$_{\pm .03}$ \\
    \algname{M-ClippedSGD} & 18.95$_{\pm .03}$ \\
    \algname{AdamW} & 18.67$_{\pm .00}$ \\
    \bottomrule
    \end{tabular}
    \end{center}
    \vspace{-1.em}
\end{table}

\section{Experimental details}\label{sec: exps details}

\subsection{Hyperparameters sweep}\label{app:pre-training}

We adopted a LLaMA-based architecture~\citep{llama} with RMSNorm~\citep{rmsnorm} and SwiGLU~\citep{shazeer2020glu} activations on the C4 dataset~\citep{c4}. 
Following~\cite{relora}, we used a batch size of 512 sequences and a sequence length of 256. 
We used a T5 tokenizer, since it was also trained on C4 with dictionary size equal to 32k.
We trained the model for 100k steps.

For all experiments, while the main model parameters use the respective optimization method, the LM head layer is optimized with AdamW~\citep{loshchilov2017decoupled}. 
This follows prior work~\cite{anything_but_sgd} which demonstrated that the LM head layer requires more fine-grained effective learning rate adaptation across different tokens for optimal performance.
We used the Nesterov acceleration scheme with a momentum value of 0.9 for all methods except AdamW.
For AdamW, we used standard hyperparameters: $\beta_1 = 0.9, \beta_2 = 0.999, \varepsilon=$1e-8. 

We selected the learning rate through a grid search with multiplicative step of $10^{\frac{1}{4}}$ (LM head layer optimized with AdamW and learning rate equal to 1e-3).
We used a cosine learning rate schedule with a warmup of 10\% of the total number of steps and decay of the final learning rate down to 10\% of the peak learning rate.
In addition, we selected the best weight decay value between [0, 0.01, 0.1].

The final best hyperparameters are shown in~\Cref{tab:hyperparams}.

\renewcommand{\arraystretch}{1.1}

\begin{table}[h!]
    \caption{LLaMA 130m pre-raining hyperparameters.}
    \label{tab:hyperparams}
    \begin{center}
        \begin{tabular}{|l|c|c|c|c|}
        \hline
        \textbf{Method} & \algname{M-ClippedSGD} & \algname{M-NSGD} & \algname{M-SignSGD} & \algname{AdamW} \\
        \hline
        \textbf{Learning rate} & $10^{1.5}$ & $10^{0}$ & $10^{-2.75}$ & $10^{-3}$\\
        \hline
        \textbf{Gradient clipping} & $0.03125$ & - & -  & 1.0 \\
        \hline
        \textbf{Weight decay} & $0$ & $0$ & $0.01$ & $0.01$\\
        \hline
        \end{tabular}
    \end{center}
\end{table}

\subsection{Computational Resources}\label{app:compute}

We conducted all experiments described in~\Cref{sec:experiments,sec: moe exps} using NVIDIA A100 GPUs. We utilized 8 GPUs (full node) with \texttt{torch.nn.parallel.DistributedDataParallel} for most of the runs. A complete run for the 130M model (100k steps) took 6 hours, whereas each run for 1.3B model (300k steps) lasted for approximately 2 days.




\section{\algname{minibatch-SignSGD} for distributed optimization}\label{sec: distributed}
Consider distributed optimization with one server and $M$ workers, each of which calculates its own gradient estimate. The server receives all estimates, aggregates them, and sends back the updated solution to the workers. Sign-based methods are so effective in terms of communication \cite{bernstein2018majorityvote,jin2020stochastic}, as sending a sign vector costs only $O(d)$ operations. We use aggregation based on the majority voting. 
\begin{algorithm}[ht!]
\caption{\algname{Distributed-MajorityVote-SignSGD} }
\label{alg:majorityvotesignSGD}   
\begin{algorithmic}[1]
\REQUIRE Starting point $x^1 \in \R^d$, number of iterations $T$, stepsizes  $\{\gamma_k\}_{k=1}^{T}$, batchsizes $\{B_k\}_{k=1}^{T}$.

\FOR{$k=1,\ldots, T$}
\STATE Sample $\{\xi^{k,j}_i\}_{i=1}^{B_k}$ and compute gradient estimate $g^{k,j} = \sum_{i = 1}^{B_k} \nabla \nicefrac{f(x^k, \xi^{k,j}_i )}{B_k}$ for each worker $j \in \overline{1,M}$;
\STATE Send signs $\sign(g^{k,j})$ to server for each worker $j \in \overline{1,M}$;
\STATE Compute on server $g^k =  \sign\left(\sum_{j=1}^M \sign(g^{k,j})\right);$
\STATE Send point $x^{k+1} = x^k - \gamma_k \cdot g^k$ to each worker;
\ENDFOR
\ENSURE uniformly random point from $\{x^1, \dots, x^{T}\}$ . 
\end{algorithmic}
\end{algorithm}

\begin{theorem}[\textbf{HP complexity for \algname{Distributed-MajorityVote-SignSGD}}]\label{thm:dist signsgd conv}
Consider lower-bounded $(L_0,L_1)$-smooth function $f$ (As. \ref{as: bounded}, \ref{as: smooth}) and HT gradient estimates $\kappa \in (1,2]$ (As. \ref{as: pBCM}). Then Alg. \ref{alg:majorityvotesignSGD} with $M$ workers requires the sample complexity $N_M$ per worker  to achieve $\frac{1}{T} \sum_{k=1}^{T}  \|\nabla f(x^k)\|_1 \leq \varepsilon$ with probability at least $1-\delta$ for:

\textbf{Optimal tuning:}  $T = O\left(\frac{\Delta L_1^\delta d^\frac{3}{2} }{\varepsilon}\right), \gamma_k \equiv \frac{1}{48 L_1^\delta d^\frac32} , B_k \equiv  \left(\frac{16\|\Vec{\sigma}\|_1}{\sqrt{M}\varepsilon}\right)^\frac{\kappa}{\kappa-1}$ for $\varepsilon \geq \frac{8L_0}{L_1\sqrt{d}}$ and $T = O\left(\frac{L_0^\delta d }{\varepsilon^2}\right), \gamma_k \equiv \sqrt{\frac{\Delta}{20 L_0^\delta dT}} , B_k \equiv  \left(\frac{16\|\Vec{\sigma}\|_1}{\sqrt{M}\varepsilon}\right)^\frac{\kappa}{\kappa-1} $ for $\varepsilon \leq \frac{8L_0}{L_1\sqrt{d}}$:
\begin{equation}
   N_M = O\left(\left(\frac{\Delta L_0 d }{\varepsilon^2} + \frac{\Delta L_1   d^\frac{3}{2}}{\varepsilon}\right)\left[1 +  \left(\frac{\|\Vec{\sigma}\|_1}{\sqrt{M}\varepsilon}\right)^\frac{\kappa}{\kappa-1}\right]\log \nicefrac{1}{\delta}\right), \label{eq: sign SGD optimal dist} 
\end{equation}
\textbf{Arbitrary tuning:\footnote{These bounds are proved for a metric $\min_{k \in \overline{1,T}}\|\nabla f(x_k)\|_1  \leq \varepsilon$.}}  Until plateau $\gamma_k = \gamma_0 \leq \frac{1}{48L^\delta_1d^\frac32}, B_k = B_0k^2$, after $\gamma_k = \frac{\gamma_0}{\sqrt{k}}, B_k = B_0k$:
\begin{eqnarray}
    \varepsilon \geq \frac{8L_0}{L_1\sqrt{d}} &\Rightarrow& N_M = \tilde{O}\left( B_0\left(\frac{\Delta}{\gamma_0\varepsilon} 
 \right)^3 + \frac{1}{B_0^2}\left(\frac{\|\Vec{\sigma}\|_1}{\sqrt{M} \varepsilon}\right)^\frac{3\kappa}{2(\kappa - 1)}\right), \notag \\
    \varepsilon \ll \frac{8L_0}{L_1\sqrt{d}} &\Rightarrow& N_M =  \tilde{O}\left(\frac{B_0(L_0^\delta \gamma_0 d + \Delta/\gamma_0)^4}{\varepsilon^4}  +  \frac{1}{B_0}\left(\frac{\|\Vec{\sigma}\|_1}{\sqrt{M}\varepsilon}\right)^\frac{2\kappa}{\kappa - 1} \right), \notag
\end{eqnarray} 
where $\Delta = f(x^1) - f^*, L_0^\delta = L_0 \log(\nicefrac{1}{\delta}), L_1^\delta = L_1 \log(\nicefrac{1}{\delta}).$
\end{theorem}

\begin{proof}[Proof of Theorem \ref{thm:dist signsgd conv}]
    This proof completely copies the proof of \algname{minibatch-SignSGD} Complexity Theorem \ref{thm:minibatch SignSGD} from Appendix \ref{subsec: minibatch signsgd proof} with substitution of $\|\Vec{\sigma}\|_1$ with $\frac{\|\Vec{\sigma}\|_1}{\sqrt{M}}$. Such substitution is justified by \algname{MajorityVote-SignSGD} Convergence Lemma \ref{lem: majority signsgd T update} which tells that noise level drops as $\sqrt{M}$ with the growth of worker number $M$. The condition $M \geq 160/\kappa^2$ is satisfied for $\kappa > 1$ automatically after the fixed number of training steps in the beginning.    
\end{proof}


\end{document}